\documentclass[oneside,11pt]{amsart}
\usepackage{amssymb,latexsym,amsmath,amsthm,enumitem,mathrsfs}
\usepackage[margin=1in]{geometry}
\usepackage{hyperref}
\usepackage{fancyhdr}
\usepackage{color}
\usepackage{stmaryrd}
\usepackage{mathrsfs}
\usepackage{lineno}
\usepackage{enumitem}

\numberwithin{equation}{section}

\newtheorem{thm}{Theorem}[section]
\newtheorem*{thm*}{Theorem}
\newtheorem{prop}[thm]{Proposition}
\newtheorem{lemma}[thm]{Lemma}

\newtheorem{cor}[thm]{Corollary}

\theoremstyle{remark}
\newtheorem{remark}[thm]{Remark}

\newcommand{\C}{\mathbb{C}}
\newcommand{\F}{\mathbb{F}}

\newcommand{\R}{\mathbb{R}}
\newcommand{\T}{\mathbb{T}}
\newcommand{\Z}{\mathbb{Z}}

\newcommand{\Fscr}{\mathscr{F}}

\newcommand{\Lscr}{\mathscr{L}}

\newcommand{\Leg}[2]{\left( \frac{#1}{#2}\right)}

\newcommand{\GL}{\mathrm{GL}}

\newcommand{\ep}{\varepsilon}

\newcommand{\con}{\equiv}

\newcommand{\ndiv}{\nmid}
\newcommand{\modd}[1]{\; ( \text{mod} \; #1)}
\newcommand{\bstack}[2]{#1 \atop #2}

\newcommand{\maps}{\rightarrow}

\newcommand{\Union}{\bigcup}
\newcommand{\union}{\cup}

\newcommand{\al}{\alpha}
\newcommand{\be}{\beta}

\newcommand{\ga}{\gamma}
\newcommand{\del}{\delta}
\newcommand{\Del}{\Delta}

\newcommand{\om}{\omega}
\newcommand{\Om}{\Omega}

\newcommand{\sig}{\sigma}

\newcommand{\Dcal}{\mathcal{D}}

\newcommand{\Fcal}{\mathcal{F}}

\newcommand{\Ical}{\mathcal{I}}

\newcommand{\Mcal}{\mathcal{M}}

\newcommand{\Scal}{\mathcal{S}}

\newcommand{\Ucal}{\mathcal{U}}

\newcommand{\Qbf}{\mathbf{Q}}

\newcommand{\Sbf}{\mathbf{S}}

\newcommand{\onebf}{\boldsymbol1}

\newcommand{\abf}{{\bf a}}

\newcommand{\beq}{\begin{equation}}
\newcommand{\eeq}{\end{equation}}

\makeatletter
\def\@tocline#1#2#3#4#5#6#7{\relax
  \ifnum #1>\c@tocdepth 
  \else
    \par \addpenalty\@secpenalty\addvspace{#2}%
    \begingroup \hyphenpenalty\@M
    \@ifempty{#4}{%
      \@tempdima\csname r@tocindent\number#1\endcsname\relax
    }{%
      \@tempdima#4\relax
    }%
    \parindent\z@ \leftskip#3\relax \advance\leftskip\@tempdima\relax
    \rightskip\@pnumwidth plus4em \parfillskip-\@pnumwidth
    #5\leavevmode\hskip-\@tempdima
      \ifcase #1
       \or\or \hskip 1em \or \hskip 2em \else \hskip 3em \fi%
      #6\nobreak\relax
    \hfill\hbox to\@pnumwidth{\@tocpagenum{#7}}\par
    \nobreak
    \endgroup
  \fi}
\makeatother

\newcommand{\Rcal}{\mathcal{R}}

\definecolor{pink}{rgb}{1,.2,.6}
\definecolor{orange}{rgb}{0.7,0.3,0}
\definecolor{blue}{rgb}{.2,.6,.75}
\definecolor{green}{rgb}{.4,.7,.4}
\definecolor{purple}{RGB}{127,0,255}

\newcommand{\xtra}[1]{}

\newcommand{\exendnote}[1]{}

\newcommand{\apphide}[1]{}

\DeclareMathOperator{\Tr}{tr}
\DeclareMathOperator{\SU}{SU}
\renewcommand{\rho}{\varrho}

\begin{document}

\title[On superorthogonality]{On superorthogonality}

\author[Pierce]{Lillian B. Pierce\\$\quad$\\ \MakeLowercase{with an appendix by \uppercase{E}mmanuel \uppercase{K}owalski}}
\address{Department of Mathematics, Duke University, 120 Science Drive, Durham NC 27708 USA }
\email{pierce@math.duke.edu}

	\begin{abstract}
  In this survey, we explore how superorthogonality amongst functions in a sequence $f_1,f_2,f_3,\ldots$ results in direct or converse inequalities for an associated square function.  We distinguish between three main types of superorthogonality, which we demonstrate arise in a wide array of settings in harmonic analysis and number theory. This perspective gives clean proofs of central results, and unifies topics including Khintchine's inequality, Walsh-Paley series, discrete operators, decoupling, counting solutions to systems of Diophantine equations, multicorrelation of trace functions, and the Burgess bound for short character sums. 
	\end{abstract}

\maketitle

\begin{center}
\emph{Dedicated to the memory of Elias M. Stein.}
\end{center}

 \section{Introduction}\label{sec_superorthog}

Let $\{f_n\}_n$ be a sequence of functions associated to a function $f$. Our goal is to understand two types of inequalities:
\\ 
{\bf The direct inequality:}
\[
\| \sum_n f_n \|_{L^p} \leq c_p \| ( \sum_n |f_n|^2)^{1/2} \|_{L^p}.
\]
{\bf The converse inequality:}
\[
 \| ( \sum_n |f_n|^2)^{1/2} \|_{L^p} \leq {c_p}' \| f\|_{L^p}.
\]
Given an operator with a suitable   decomposition
\[ T  = \sum_n T_n,\]
upon setting $f_n = T_n(f)$, 
if both estimates were true, they would imply that
\[ \| Tf \|_{L^p} \leq  {c_p}' c_p\| f\|_{L^p} .
\]

Superorthogonality can be used to prove one or both of these inequalities.  
Superorthogonality  is the property that 
 for any  tuple of functions $f_{n_1}, \ldots, f_{n_{2r}}$ from the given sequence $\{f_n\}_n$,
 \beq\label{super_gen_super_id}
 \int f_{n_1}\bar{f}_{n_2} \cdots f_{n_{2r-1}} \bar{f}_{n_{2r}}   =0 
 \eeq
as long as an appropriate condition is satisfied by the tuple of indices $(n_1, \ldots,  n_{2r})$. 
In this note,    we show that a wide variety of topics in harmonic analysis and number theory can be united within the framework of superorthogonality, and associated direct and converse inequalities.
We exhibit three main types of superorthogonality.\\

\subsubsection*{Type I}
Type I superorthogonality is the case in which (\ref{super_gen_super_id}) holds if the tuple $(n_1, \ldots, n_{2r})$ has the property that some value $n_j$ appears an odd number of times.
We show that any collection of functions with Type I superorthogonality satisfies a direct inequality. 

 Type I superorthogonality classically appeared in Khintchine's inequality for the Rademacher functions, which can be viewed as both a direct and a converse inequality. 
 Furthermore we show that a refinement of Type I superorthogonality   underpins a  recent result of \cite{GGPRY19x}, a philosophical converse to the   proof of the Vinogradov Mean Value Theorem via decoupling  \cite{BDG16}. This notion of superorthogonality shows that   counts for the number of diagonal solutions and near-solutions to a system of Diophantine equations  can imply a direct inequality for a square function; this in turn implies a decoupling inequality  for the extension operator associated to the corresponding curve.\\

\subsubsection*{Type II} Type II superorthogonality is the case in which  (\ref{super_gen_super_id}) holds if the tuple $(n_1, \ldots, n_{2r})$ has the property that some value $n_j$ appears precisely once.  We show that any collection of functions with Type II superorthogonality satisfies both a direct inequality and a  multilinear direct inequality.

Any sequence $\{ f_n\}_n$ in which $f_1,f_2,\ldots,f_n,\ldots$  are mutually independent  random variables, and each has mean zero (in the sense that $\int f_n dx =0$), satisfies the Type II condition. 
Supposing for simplicity the functions are real-valued, the mutual independence guarantees that 
\[ \int f_{n_1} f_{n_2} \cdots f_{n_{2r}} dx  = \prod_\ell ( \int f_{n_\ell}^{m_\ell} dx),\]
where $m_\ell$ is the multiplicity with which $f_{n_\ell}$ occurs in the product $f_{n_1}f_{n_2}\cdots f_{n_{2r}}$. Hence the defining property of Type II superorthogonality holds, since this integral vanishes as soon as at least one function has multiplicity one.

  We  show  that Type II  superorthogonality   also holds in a completely different setting, namely for a sequence of discrete functions $\{ f_{a/q}\}_{a/q}$ acting on $\Z$, indexed by a  collection of rational numbers. Each function is defined according to 
\[ (f_{a/q})\widehat{\;}(\xi)  = m(\ep^{-1}(\xi - a/q)) \widehat{f}(\xi),
\]
where  $m$ is a  periodization of an $L^p(\R)$ multiplier supported in $(-1/2,1/2]$, and   $\ep$ is appropriately small. In this case, verifying Type II superorthogonality requires quite different methods---arithmetic rather than probabilistic, relating to the prime factorizations of the denominators in the rationals $a/q$. 
Using Type II superorthogonality, we  prove  a direct inequality and a multilinear direct inequality related to the collection $\{f_{a/q}\}_{a/q}$.
Furthermore, we prove two types of   converse inequalities in this setting. Taken altogether, these inequalities prove  the $\ell^p$ boundedness  of a discrete operator that is a building block in the celebrated work of Ionescu and Wainger \cite{IW} on discrete singular Radon transforms; see Theorem \ref{thm_discrete_main_goal}.   Our  presentation here serves as a friendly introduction to  the influential method of Ionescu and Wainger.\\

\subsubsection*{Type III} Type III superorthogonality is the case in which  (\ref{super_gen_super_id}) holds if the tuple $(n_1, \ldots, n_{2r})$ has the property that some value $n_j$ appears precisely once and is strictly greater than all other values in the tuple.  

This type of superorthogonality  occurred a few years after Khintchine's inequality, in Paley's work on the Walsh-Paley series \cite{Pal32}, where he was able to use Type III superorthogonality  to prove both a direct inequality and  a converse inequality. Here we develop Paley's  ideas in general terms, to show that any collection of functions with Type III superorthogonality, and two additional properties, satisfies both a direct and a converse inequality.\\

\subsubsection*{Quasi-superorthogonality} Fourth, we introduce the notion of quasi-superorthogonality: we no longer assume that  (\ref{super_gen_super_id}) vanishes, but instead that it exhibits  quantitative cancellation. Now instead of a direct inequality, we obtain a variant that also includes an  ``off-diagonal'' term on the right-hand side. 
Such inequalities are nevertheless very useful. 

In fact, we observe that a deep application of $\ell$-adic cohomology  and the Riemann Hypothesis over finite fields proves that Type I quasi-superorthogonality holds for sequences of   ``trace functions''; this is the statement of  multicorrelation of trace functions proved in \cite{FKM15}.  Hence an approximate direct inequality holds for such functions. Moreover, the source of quasi-superorthogonality of trace functions is a consequence of ``exact'' superorthogonality in the sense of (\ref{super_gen_super_id}) for a different set of functions, combined with the Riemann Hypothesis over finite fields. An appendix by Emmanuel Kowalski makes this phenomenon explicit.  

As an application, we  give a complete proof of  the Burgess bound for character sums \cite{Bur57} from the perspective of  quasi-superorthogonality and an approximate direct inequality for square functions; see Theorem \ref{thm_Burgess}. This is a celebrated result in number theory that has long held the  record for certain problems related to the Generalized Riemann Hypothesis. 
As remarked in \cite{GalMon10}, ``While the original argument [of Burgess] is easily followed line-by-line, it seems hard to comprehend the larger
sense of it, because several technical difficulties are being dealt with at the same time that the main
idea is unfolding.''  Here we give an   intuitive motivation for the method by combining quasi-superorthogonality with simplifying ideas from \cite{GalMon10, HB12}. This also  highlights certain barriers to improving Burgess's result. 

\subsection*{In Memoriam}
It was an honor and delight to learn from   Elias M. Stein  for twenty years. This paper is in many ways a joint product with Eli. It germinated from a brief note Eli wrote to me in the summer of 2018, while we were collaborating on a book manuscript.  At the time, we were interested in the relationship of superorthogonality  to square function estimates.  We noticed  variants of the basic notion   in several settings, and began to divide  superorthogonality into types.
 While the ideas of that hand-written note have now grown and changed, the heart of the matter was already on those foolscap pages. In homage, I follow Eli's words closely in  phrases in the introduction and  in \S \ref{sec_Paley} (particularly \S \ref{sec_TIII'}). The material of \S \ref{sec_TI_discrete} develops a special case of a key theorem in  the book manuscript we were preparing, and represents our shared work. The later sections   move on to connections with number theory, which we also enjoyed discussing that summer.   I have taken the liberty of developing ideas from our conversations, notes, and drafts, in loving debt to Eli; I am of course solely responsible for any  inaccuracies in the current presentation.

\subsection*{Outline}

In  \S \ref{sec_TypeI} we introduce Type I superorthogonality and formally prove a direct inequality; from this we deduce Khintchine's inequality for Rademacher functions and a variant of the Marcinkiewicz-Zygmund theorem, which we apply later.
In  \S \ref{sec_TypeII} we introduce Type II superorthogonality and formally prove a direct inequality. 
In  \S \ref{sec_Paley} we introduce Type III superorthogonality and use it to prove both a direct and  a converse inequality; these apply for example to  Walsh-Paley series. We then mention a variant Type III' that applies to Fourier multiplier operators.  

Having introduced the three main types, in  \S \ref{sec_TI_discrete} we then turn to the core technical work of applying Type II to prove a  theorem about discrete operators.
In \S \ref{sec_decoupling} we refine Type I to Type I* and  exhibit its relationship to decoupling and counting solutions to Diophantine equations.

We then turn to the notion of quasi-superorthogonality and its applications in number theory. In \S \ref{sec_trace} we document why trace functions satisfy Type I quasi-superorthogonality, and deduce an approximate direct inequality. We then  introduce the notion of incomplete sums of trace functions and the P\'olya-Vinogradov method, leading to the difficult question of bounding short sums. In \S \ref{sec_Burgess_method} we develop a schematic approach to bounding short sums via quasi-superorthogonality. We then carry this out precisely, first obtaining a weaker bound with a more intuitive proof, and then refining it to recover the classical Burgess bound.   

Appendix A concerns further details related to the setting of Walsh-Paley series.

Appendix B by Emmanuel Kowalski provides an explicit description of how an instance of exact superorthogonality leads to quasi-superorthogonality for trace functions.
 
 As this note covers territory within both analysis and number theory, it is written to be broadly accessible. In addition to the main ``types'' of superorthogonality we focus on here, we periodically make further remarks about other settings and other types and their variants, but given the universality of the phenomena, we do not intend this survey to be exhaustive. We anticipate that many further instances of superorthogonality will be recognized by readers.

\subsection*{Conventions}

Strictly speaking, when one specifies that a collection of functions $\{f_n\}_n$ satisfies a superorthogonality condition (\ref{super_gen_super_id}), one should specify for which $r$ this holds, the set of indices $n$, and   the measure space in which integration takes place. In the settings we consider,  the superorthogonality property   holds for all integers $r \geq 1$.  
In formal arguments to deduce a direct or converse inequality using superorthogonality, we assume the sum $\sum_n f_n$ is taken over a finite set of indices, and then the desired inequality is proved with a constant that is uniform with respect to the cardinality of this set. In applications in which the set of indices is infinite, this  suffices if appropriate limiting arguments apply.
In formal arguments we suppress notation for the measure space $L^p(\Mcal,d\mu)$ until we state a specific setting, at which point we then work precisely with spaces such as $L^p(\R)$ and $L^p[0,1]$ with Lebesgue measure, or $\ell^p(\Z)$ and $\ell^p(\Z/q\Z)$ with counting measure. In the settings we consider, the functions $f_n$ in the collection $\{f_n\}_n$ are assumed to be distinct.

 Observe that Type I $\Rightarrow$ Type II $\Rightarrow$ Type III, in the sense that any sequence of functions $\{f_n\}_n$ that is of Type I must be of Type II, and so forth. 
  While the condition that defines Type I and Type II superorthogonality is invariant under a change of ordering of the functions $f_1,f_2,\ldots,f_n,\ldots$, the condition that defines Type III is not.
In what follows, we assume that the set $\{f_n\}= \{ f_1, f_2, \ldots, f_n, \ldots \}$ has been given with an ordering.

Constants such as $C_p, c_p, A_p$ and so on, may  indicate certain dependencies, but may change in value from one occurrence to the next. The notation $f \ll_p g$ is also used, and indicates that there is an implicit constant $C_p$ such that $|f| \leq C_p g$.

\section{Type I superorthogonality and the Rademacher functions}\label{sec_TypeI}

We  introduce  a first notion of superorthogonality,   working with real-valued functions for simplicity. It is the condition   that for every $2r$-tuple $f_{n_1}, \ldots, f_{n_{2r}}$ of functions from a sequence $\{f_n\}_n$,
\beq\label{TIII_vanish}
\int  f_{n_1} f_{n_2} \cdots  f_{n_{2r}}   =0
\eeq
as long as \\
{\bf Type I: } the tuple $(n_1,n_2, \ldots, n_{2r})$ has the property that there is a value $n_j$ that appears an odd number of times.
 
Here we show formally that any sequence of functions satisfying the Type I condition obeys a direct inequality; then we observe that this holds for Rademacher functions, and derive Khintchine's inequality and a variant of the Marcinkiewicz-Zygmund theorem. Later we will return to applications of the Type I property in the settings of decoupling and trace functions.

It is an elementary observation that a collection $\{f_n\}$ with Type I superorthogonality satisfies an identity in $L^2$: 
\beq\label{L2_diag}
 \| \sum_n f_n \|_{L^2}^2 = \| (\sum_n f_n^2)^{1/2} \|_{L^{2}}^2.
 \eeq
This follows from expanding the left-hand side and observing that the off-diagonal cross terms vanish, by the superorthogonality assumption.

More generally, if a set of functions $\{f_n\}$ satisfies the Type I condition, we may immediately verify the direct inequality in $L^{2r}$ for each integer $r \geq 1$. 
We expand the $L^{2r}$ norm using a multinomial expansion,
\[\| \sum_n f_n \|_{L^{2r}}^{2r}= \int | \sum_{n  } f_n|^{2r}  = \sum_{(a_1,\ldots, a_s)} C(a_1,\ldots, a_s) \int f_{n_1}^{a_1} \cdots f_{n_s}^{a_s} ,\]
where the sum ranges over all $s \leq 2r$, all pairwise distinct $n_1,\ldots, n_s$ in the (finite) index set, and all $(a_1,\ldots, a_s)$  with   $a_1 + \cdots + a_s = 2r$; here  $C(a_1,\ldots,a_s)=(a_1+\cdots + a_s)! / (a_1!   \cdots a_s!)$.
  By the Type I property, the integral vanishes except for those   $(a_1,  \ldots, a_s)$ with each $a_i$ even, say $a_i = 2b_i$.
  Moreover, non-vanishing terms on the right-hand side must  have $s \leq r$.  
  
On the other hand, observe that    
    \beq\label{B_sum}
    \sum_{(b_1,\ldots,b_s)} C(b_1,\ldots,b_s) \int f_{n_1}^{2b_1} \cdots f_{n_s}^{2b_s} = \int (\sum_{n } f_n^2)^{r} = \| ( \sum_{n } f_n^2)^{1/2} \|_{L^{2r}}^{2r}, 
    \eeq
where the left-most sum  ranges over all $s \leq r$, all pairwise distinct $n_1,\ldots, n_s$ in the  index set, and all  $(b_1,\ldots, b_s)$   with   $b_1 + \cdots + b_s = r$.
 We may conclude that 
\[ \| \sum_{n  } f_n \|_{L^{2r}}^{2r}  \leq C_r \| ( \sum_{n } f_n^2)^{1/2} \|_{L^{2r}}^{2r},\]
where we define
\[ C_r = \max_{(b_1,\ldots,b_s)} \frac{C(2b_1, \ldots, 2b_s)}{C(b_1,\ldots, b_s)},
\]
and the  maximum is taken over all $(b_1,\ldots, b_s)$  with   $b_1 + \cdots + b_s = r$ and $s \leq r$. One can observe for example that $C_r \leq \frac{(2r)!}{r!2^r} < r^r$, but all we require is that it depends only on $r$.
In conclusion, we have verified the direct inequality for the set of functions $\{f_n\}$, for each $p=2r$.

  This argument has been written in the spirit of Paley and Zygmund \cite[Lemma 2]{PalZyg30}, where it was developed to prove the Khintchine inequality for Rademacher functions. As we will require this result later on, and it is a nice illustration of Type I superorthogonality, we now also demonstrate its proof.

\subsection{The Rademacher functions}
We recall the definition of the Rademacher functions \cite[\S VI, p. 130]{Rad22}: for $n=0$,
\[ r_0(t) = 1  \text{ for $0 \leq t < 1/2$}, \quad r_0(t) = -1 \text{  for $1/2 \leq t < 1$,} \quad r_0(t+1) = r_0(t) .\]
Then we set $r_n(t) = r_0(2^n t)$ for each $n=1,2,3,\ldots.$
These satisfy the property that for distinct $n_1,n_2, \ldots, n_s$,
\beq\label{Rad_TypeI}
\int_0^1 r_{n_1}^{a_1}(t) r_{n_2}^{a_2}(t) \cdots r_{n_s}^{a_s} (t)dt =0
\eeq
unless all the integers $a_1,\ldots, a_s$ are even, in which case the integral evaluates to 1.
In particular, $\{r_n\}$ satisfies  Type I  superorthogonality on $L^{2r}[0,1]$ for every integer $r \geq 1$. 

 We may verify this as follows. Since for any $n$, $r_n(t)^2 \con 1$, it suffices to prove that (\ref{Rad_TypeI}) vanishes in the case in which $n_1>n_2> \cdots > n_s$ and all $a_i=1$. Observe that the function $r_{n_2}(t) \cdots r_{n_s}(t)$ is a step function that is constant on $2^{n_2+1}$ intervals of length $2^{-(n_2+1)}$. Thus it suffices to show that on each of these intervals, say $I$,  $\int_{I} r_{n_1}(t)dt=0$.  
In turn, each such  interval $I$ can be dissected into $2^{n_1-n_2}$ intervals of equal length, and on half these intervals $r_{n_1}(t)$ takes the value $+1$ while on the other intervals $r_{n_1}(t)$ takes the value $-1$. Consequently the integral of $r_{n_1}(t)$ over $I$ is zero, and from this we deduce (\ref{Rad_TypeI}).
   This proof is in the spirit of Kaczmarz and Steinhaus, e.g.  \cite[p. 236]{KacSte30}, \cite[p. 125]{KacSte36}; other classic sources are e.g.   \cite{Zyg02,Kac64}.

  We mention that Rademacher proved  that if $\sum_{n=0}^\infty |a_n|^2 < \infty$ then the series $\sum_{n=0}^\infty a_n r_n(t)$ converges pointwise for almost all $t \in [0,1]$  \cite[p. 135-138]{Rad22}; see also \cite[Vol. 1 Ch. V Thm. 82]{Zyg02} for a modern citation.

\subsection{Khintchine's inequality}\label{sec_TI_Khin}

We can apply the formal ideas developed above to deduce a useful inequality. 
This is Khintchine's inequality:  for each $0< p< \infty$, for any sequence $\{a_n\}$ of complex numbers,
\beq\label{TI_Khintchine}
( \sum_{n=0}^\infty |a_n|^2)^{1/2} \ll_p \| \sum_{n=0}^\infty a_n r_n(t) \|_{L^p[0,1]} \ll_p ( \sum_{n=0}^\infty |a_n|^2)^{1/2}.
\eeq
We will call the right-most inequality the direct inequality, and the left-most inequality the converse inequality.
Standard modern proofs  can be found  in e.g. \cite[Appendix D]{SingInt}, \cite[Prop. 4.5]{Wol03} (see \cite{Haa81} for precise constants). 
We will consider the case $p>1$, and our interest is that for $p=2r$ with $r \geq 1$ integral, we can prove this as an application of Type I superorthogonality; this treatment is  in the spirit of older proofs, e.g. \cite[Lemma 2]{PalZyg30}, \cite[Vol. I Ch. V Thm. 8.4]{Zyg02}.  

First, there are various reductions. One can treat the real and imaginary parts separately, so that  we only consider the case in which each $a_n$ is real. Due to the pointwise a.e. convergence mentioned above, it suffices to prove the inequalities for a truncated sum over $0 \leq n \leq N$, uniformly in $N$. 
First note that by (\ref{L2_diag}),  there is an identity on $L^2[0,1]$:  
\[( \int_0^1| \sum_{n=0}^N a_n r_n(t) |^2 dt )^{1/2}=( \sum_{n_1,n_2} a_{n_1} a_{n_2} \int_0^1 r_{n_1}(t) r_{n_2}(t)   dt )^{1/2} 
	=  (\sum_{n=0}^N |a_n|^2)^{1/2}.
\]
For the direct inequality, the main content of (\ref{TI_Khintchine}) thus lies in the case $p>2$, since for $p<2$, H\"older's inequality  shows that $\| \sum a_n r_n \|_{L^p[0,1]} \leq \| \sum a_n r_n \|_{L^2[0,1]} = (\sum |a_n|^2)^{1/2}$; analogously, for the converse inequality the main content lies in the case $p<2$. 
Moreover, for $p<2$ the converse inequality can be deduced from the direct inequality:  let $r>2$ be such that $1/2 = (1/2) (1/p + 1/r)$, so that by H\"older's inequality 
\[ (\sum_{n=0}^N |a_n|^2)^{1/2} = \| \sum a_n r_n \|_{L^2} \leq \|\sum a_n r_n \|_{L^p}^{1/2}   \|\sum a_n r_n \|^{1/2}_{L^r} .
\]
Then upon applying the direct inequality for $L^r$, we conclude that the converse inequality holds for $L^p$. 
Thus it only remains to verify the direct inequality for $p >2$. Moreover, it suffices to consider the case   $p=2r$ with $r \geq 1$ an integer, since given any $p>2$ if we let $r$ denote that integer such that $2(r-1) \leq p < 2r$, then for any function $f$ on the space $[0,1]$, $\| f\|_{L^{2r-2}[0,1]} \leq \|f\|_{L^p[0,1]} \leq \|f\|_{L^{2r}[0,1]}.$

Now let $p=2r$ with $r\geq 1$ an integer.
We may apply our formal argument for Type I functions with $f_n = a_n r_n$. Moreover, using the fact that the integral in (\ref{Rad_TypeI}) evaluates to 1 when it is nonvanishing, we see in (\ref{B_sum}) that 
\[ \sum_{(b_1,\ldots,b_s)} C(b_1,\ldots,b_s) \int f_{n_1}^{2b_1} \cdots f_{n_s}^{2b_s} =\sum_{(b_1,\ldots,b_s)} C(b_1,\ldots,b_s) a_{n_1}^{2b_1} \cdots a_{n_s}^{2b_s} 
	 = (\sum_{n} a_n^2)^r.
\]
Thus the argument concludes  as desired, and
\[ \| \sum_{n  } f_n \|_{L^{2r}}  \leq C_r^{1/2r} ( \sum_n a_n^2)^{1/2}.\]

\subsection{A theorem of Marcinkiewicz-Zygmund}
We state a nice consequence of Khintchine's inequality, which we will apply in our study of discrete operators in \S \ref{sec_TI_discrete}.
We work here with a measure space $(X,d\mu)$; in \S \ref{sec_TI_discrete} we apply it to $\ell^p(\Z)$ with counting measure,  with appropriate associated Fourier transform mapping to functions on $(-1/2,1/2]$ (identified with the torus).
\begin{thm}[Marcinkiewicz-Zygmund]\label{thm_MarZyg_p}
Let $1\leq p < \infty$ be fixed and suppose that $T$ is a bounded linear operator from $L^p(X)$ to $L^p(X)$, with norm $M_p$, that is, for all $f \in L^p(X)$,
\[ \|Tf\|_{L^p(X)} \leq M_p \|f\|_{L^p(X)}.\]
I) Then there exists a constant $C_p$ such that for any sequence $\{ f_j\}$ of functions with $f_j \in L^p(X)$,
\[ 
\| ( \sum_{j=1}^\infty |Tf_j|^2)^{1/2} \|_{L^p(X)} \leq M_p C_p \| ( \sum_{j=1}^\infty |f_j|^2)^{1/2} \|_{L^p(X)} .\]
II) Suppose moreover that $T$ is a translation-invariant operator with corresponding Fourier multiplier $m(\xi)$, and that $\{\xi_j\}_j$ is a fixed set of points. Define for each $j$ the associated operator $T_j$ acting   by $(T_j f)\widehat{\;} (\xi) = m(\xi - \xi_j) \widehat{f}(\xi)$. 
Then 
\[ 
\| ( \sum_{j=1}^\infty |T_jf_j|^2)^{1/2} \|_{L^p(X)} \leq M_p C_p\| ( \sum_{j=1}^\infty |f_j|^2)^{1/2} \|_{L^p(X)} .\]
\end{thm}

 \begin{proof}
To prove part (I), it suffices to consider the case of a finite sequence of functions $f_1,\ldots, f_N$, from which the general statement   follows by the monotone convergence theorem. Recall the Rademacher functions $\{r_j\}_j$. Given $f_1,\ldots, f_N$, we define for $t \in [0,1]$ the function
 \[F(x,t) = \sum_{1 \leq j \leq N} r_j(t)f_j(x). \]
Since $T$ is linear, $TF(x,t) = \sum_j r_j(t) Tf_j(x)$, so that by the assumed boundedness of $T$, 
\[ \int_X | \sum_j r_j(t) Tf_j(x) |^p d\mu(x) \leq M_p \int_X | \sum_j r_j(t) f_j(x)|^p d\mu(x)\]
for each $t$.
By integrating in $t$ and applying Fubini's theorem,  
\[ \int_X \int_0^1 |\sum_j r_j(t) Tf_j(x)|^p dt d\mu(x) \leq M_p \int_X \int_0^1 | \sum_j r_j(t) f_j(x) |^p dt d\mu(x).\]
  Appying Khintchine's inequality for each fixed $x$ then shows that the left and right-hand sides are comparable to 
\[ 
	\int_X ( \sum_j |Tf_j|^2)^{p/2} dx  \quad \text{and} \quad 
\int_X ( \sum_j |f_j|^2)^{p/2} dx ,
\]
respectively.

To prove part (II), observe that  
$(T_jf)(x) = e^{2\pi i x \xi_j} (T(f(\cdot) e^{-2\pi i ( \cdot )\xi_j})(x).$
 As a result, for any sequence $\{f_j\}$,
\[ \sum_j |T_j(f_j)|^2 = \sum_j |T (f_j e^{-2\pi i x \xi_j})|^2.
\]
Thus the conclusion of (II) follows from applying the conclusion of (I) to the right-hand side.
\end{proof}

\section{Type II superorthogonality}\label{sec_TypeII}
We introduce a second notion of superorthogonality, now for complex-valued functions.    It is the condition that for every $2r$-tuple of functions from a sequence $\{f_n\}$,
\[  \int f_{n_1}\bar{f}_{n_2} \cdots f_{n_{2r-1}} \bar{f}_{n_{2r}}  =0 \]
as long as:\\
{\bf Type II:}  the tuple $(n_1,n_2,\ldots, n_{2r})$ has the property that there is a value $n_j$ that appears precisely once, in which case we say that the tuple has \emph{the uniqueness property}.

In this section, we prove that any collection of functions satisfying the Type II condition satisfies a direct inequality. In \S \ref{sec_TI_discrete} we will return to this type in more detail, when we study its application to discrete operators; we will also apply this type in the setting of trace functions, when we prove the Burgess bound.

\subsection{The direct inequality}\label{sec_TI_direct}

In general, a collection $\{f_n\}$ with Type II superorthogonality  satisfies a direct inequality   in $L^{2r}$ for all integers $r \geq 1$.
We expand 
\[
 \| \sum_{n} f_n \|_{L^{2r}}^{2r} = \sum_{(n_1,\ldots, n_{2r})} \int f_{n_1} \bar{f}_{n_2} \cdots f_{n_{2r-1}} \bar{f}_{n_{2r}}  
\]
in which the sum is over all tuples $(n_1, \ldots, n_{2r})$ in the index set. Under the Type II assumption, the contribution vanishes for any such tuple with the uniqueness property; hence we need only consider tuples in which every index appears at least twice, so in particular the indices take at most $r$ distinct values. Thus we can write the above expression as 
\[ \sum_{A} \sum_{\bstack{(n_1,\ldots, n_{2r})}{\{n_1,\ldots, n_{2r}\} = A}} \int f_{n_1}  \bar{f}_{n_2} \cdots f_{n_{2r-1}}\bar{f}_{n_{2r}}  ,\]
in which the first sum is over all subsets $A$  of indices, with $|A| \leq r$. Here we distinguish between a tuple $(n_1,\ldots, n_{2r})$ and the set of (distinct) values $\{n_1,\ldots, n_{2r}\}$ appearing in the tuple. Note that once such a set $A$ is fixed, there are at most $d_r$ possible $2r$-tuples   with values corresponding to the set $A$, for a combinatorial constant $d_r$. 

The right-hand side of the direct inequality can be expanded as 
\[ \| ( \sum_{n}| f_n|^2)^{1/2} \|_{L^{2r}}^{2r} = \sum_A \sum_{\bstack{(n_1,\ldots,n_{r})}{\{n_1,\ldots, n_{r}\} = A}} \int |f_{n_1}|^2 \cdots |f_{n_{r}}|^2  ,\]
where the sum is over all sets $A$ with cardinality at most $r$. To verify the direct inequality, it suffices to show that for each set $A$ with $|A| \leq r$, for every tuple $(n_1,\ldots, n_{2r})$ without the uniqueness property such that $\{ n_1 ,\ldots, n_{2r} \} = A$,
\beq\label{sqfn_direct_A_sum}
 \int| f_{n_1} \bar{f}_{n_2}\cdots f_{n_{2r-1}}\bar{f}_{n_{2r}}|   \leq  \sum_{\{ n_1,\ldots, n_r \} = A} \int  |f_{n_1}|^2 \cdots |f_{n_r}|^2  .
 \eeq
Then upon summing over all such tuples and all such sets $A$, the direct inequality   will hold, with $c_{2r}^{2r}  = d_r$.

In order to verify (\ref{sqfn_direct_A_sum}), we claim the following. Fix any  set $A$ with $|A| \leq r$.  We may partition each   $2r$-tuple $(n_1, \ldots, n_{2r})$ without the uniqueness property whose set of distinct values is $A$,   into two $r$-tuples, say $(n_{i_1,0}, \ldots, n_{i_r,0})$ and $(n_{i_1,1}, \ldots, n_{i_r,1})$, such that 
\[
\{n_{i_1,0}, \ldots, n_{i_r,0}\} = A = \{n_{i_1,1}, \ldots, n_{i_r,1}\}.
\]
 Equivalently, we claim that we can color the entries in the $2r$-tuple so that $r$ of the entries are red and $r$ of the entries are blue, and moreover each entry of $A$ appears in red at least once and  in blue at least once. 

Let us prove this. Suppose the set $A$ has entries $a_1, \ldots, a_s$ for some $s \leq r$. For each value $a_i$ that appears an even number of times in the $2r$-tuple, say $2k_i$ times,  we color $k_i$ of these red and $k_i$ of these blue.  Next we consider the set of all entries $a_i \in A$ that each appear an odd number of times  in the $2r$-tuple, say $2k_i+1$ times, with $k_i\geq 1$. (Each $a_i$ must appear at least 3 times, since the $2r$-tuple does not have the uniqueness property.)  Since $2r$ is even, there must be an even number of such entries $a_i$ in $A$. For half of them, we color $k_i+1$ red and $k_i$ blue, and for the other half we color $k_i$ red and $k_i + 1$ blue, and this proves the claim.  
 
Now we apply the partition to verify (\ref{sqfn_direct_A_sum}). 
Fix a $2r$-tuple $(n_1,\ldots, n_{2r})$ with $\{ n_1, \ldots, n_{2r} \} = A$ and construct the $r$-tuples $(n_{i_1,0}, \ldots, n_{i_r,0})$ and $(n_{i_1,1}, \ldots, n_{i_r,1})$ as above. Then, also using the fact that for any  $\al,\be \geq 0$ we have $2\al \be \leq \al^2 + \be^2$,
\begin{align*}  \int | f_{n_1} \bar{f}_{n_2} \cdots f_{n_{2r-1}} \bar{f}_{n_{2r}}|  
	& = \int  |f_{n_{i_1},0}| \cdots |f_{n_{i_r},0}| \cdot   |f_{n_{i_1},1}| \cdots |f_{n_{i_r},1}| \\
	& \leq \frac{1}{2} \int  |f_{n_{i_1},0}|^2 \cdots |f_{n_{i_r},0}|^2  + \frac{1}{2} \int     |f_{n_{i_1},1}|^2 \cdots |f_{n_{i_r},1}|^2 \\
		& \leq  \sum_{ \{n_1, \ldots, n_{r}\} = A} \int  |f_{n_1}|^2 \cdots |f_{n_r}|^2  .
\end{align*}
This proves (\ref{sqfn_direct_A_sum}) and hence verifies the direct inequality, for $p=2r$ an even integer. 
In general, the converse inequality   needs a different argument, and we will return to this in the specific setting of \S \ref{sec_TI_discrete}.

\section{Type III  superorthogonality in the work of Paley}\label{sec_Paley}

We now introduce a third notion of superorthogonality, again working  with real-valued functions for simplicity: it is the condition that for  every $2r$-tuple of functions from  a sequence $\{f_n\}$ indexed by integers $n$,

\[  \int f_{n_1}f_{n_2} \cdots f_{n_{2r}}   =0 \]
as long as:\\
 {\bf Type III:}  the tuple $(n_1,n_2,\ldots, n_{2r})$ has the property that there is an $n_j>n_\ell$ for all $\ell \neq j$.

 Type III  superorthogonality  was exploited by Paley in his study of the Walsh-Paley series \cite{Pal32}.
Recalling the Rademacher functions $\{r_n\}$, we define   a set of  functions $\{w_n\}$ as follows. Set   $w_0(t)=1$. For $n=2^{n_1} + 2^{n_2} + \cdots + 2^{n_s}$ (with $n_1> \cdots > n_s$) set 
\beq\label{psi_rad}
 w_n(t) = r_{n_1}(t) r_{n_2}(t) \cdots r_{n_s}(t).
 \eeq
The orthogonality property (\ref{Rad_TypeI}) of the Rademacher functions immediately  implies that
\[ 
\int_0^1 w_m(t) w_n(t) dt = \begin{cases} 
	&1 \quad \text{if $n =m$} \\
	& 0 \quad \text{if $n \neq m$.}
	\end{cases}
\]
  Walsh \cite{Wal21}, Kaczmarz \cite{Kac29} and Paley \cite{Pal32} studied the functions $\{w_n\}$ extensively, and Fine \cite[\S 2]{Fin49} recognized them as the characters of the Walsh group or ``dyadic group.''
   
Fundamentally,  the collection $\{w_n\}$ is a complete orthonormal system of functions on $[0,1]$; see e.g. \cite[p. 243]{Pal32}. 
  For each $n \geq 1$, define for any real-valued function $f$ on $[0,1]$ the partial sum
\[  S_n f (t) = \sum_{m=0}^{n-1}c_m(f) w_m(t), \qquad \text{with $c_m(f) =\int_0^1 f(\theta) w_m(\theta) d\theta$}.\]
Following Stein, we call this a Walsh-Paley series.
Paley developed numerous properties of the partial sums $S_n f$, proving for example via  the Hardy-Littlewood maximal function (new at that time),
that for  integrable $f$, the dyadic partial sums converge pointwise as $n \maps \infty$,
\beq\label{ptwise}
S_{2^n} f(t) \maps f(t)
\eeq
 for almost every $t \in [0,1]$ \cite[Thm. IV]{Pal32}. See also the earlier proof of Kaczmarz \cite{Kac29}.
(The pointwise a.e. convergence  of the non-dyadic sums  $S_n f(t) \maps f(t)$ for $f \in L^p[0,1]$ with $p>1$ was much more difficult, and was resolved after Carleson's work; see \cite{Bil67,Sjo69,Thi95}, and see \cite[Theorem 7]{Ste61} for a counterexample on $L^1[0,1]$.)

To illustrate Type III superorthogonality, we will focus on the dyadic differences $f_n$ defined by
\beq\label{PW_dyadic}
 f_n  = S_{2^{n}}f - S_{2^{n-1}}f.
 \eeq
Paley \cite[Thm V]{Pal32} proved a direct inequality and a converse inequality  for the sequence  $\{f_n\}$: for any $1<p<\infty$,  for any (real-valued) $f \in L^p[0,1]$,
\beq\label{TII_dir_con_together}
  \| ( \sum_{n=0}^\infty f_n^2)^{1/2} \|_{L^p[0,1]}  \ll_p
 \| f  \|_{L^p[0,1]} \ll_p \| ( \sum_{n=0}^\infty f_n^2)^{1/2}  \|_{L^p[0,1]}  .
\eeq
From these direct and converse inequalities, Paley deduced  for any fixed $n$ the bound 
\beq\label{non_dyadic}
 \| S_n f \|_{L^p[0,1]} \leq B_p \|f\|_{L^p[0,1]}
 \eeq
for any $1<p<\infty$ \cite[Thm. VI]{Pal32}. Here the deduction of the operator bound from the direct and converse inequalities is not as simple as in the formal setting of the introduction, since the direct and converse inequalities are for dyadic differences, while (\ref{non_dyadic}) is for a non-dyadic partial sum. We provide Paley's clever proof in Appendix A.

In this section, we demonstrate Paley's method to prove the direct and converse inequalities in (\ref{TII_dir_con_together}) in the case of $p=2r$ an even integer. 
In particular, we expose a curious feature of Paley's method, which is that he applies Type III superorthogonality not just for the direct inequality, but also to prove the converse inequality. This  introduces a nonconcentration inequality (Lemma \ref{lemma_noncon}) that will  play a key role in the next section, on discrete operators.

We first work formally, abstracting Paley's ideas to a general sequence of functions $\{g_n\}$ satisfying certain properties, and at the end of the section we verify that the dyadic differences $\{f_n\}$ defined above for Walsh-Paley series satisfy all the requirements of our proof.
We reserve certain details more specific to the setting of Walsh-Paley series (limiting arguments, and the reduction to $p=2r$, which again applies Khintchine's inequality), to Appendix A.

\subsection{Formal setting}\label{sec_TIII_formal}

We now describe the formal setting in which we will work before specializing to the Walsh-Paley series. 
Let  $\{\mu_m\}$ be a sequence of real-valued functions on $[0,1]$ (with small adaptations, a finite measure space will do), let $\{c_m\}$ be a fixed sequence of real numbers, and let $\{\al_m\}$ be a strictly-increasing sequence of non-negative integers.  
For each $n \geq 0$, let $G_n$ denote the partial sum
\[  G_n (t) = \sum_{0 \leq m < \al_n} c_m \mu_m(t).
\]
Then define $g_n(t) = G_{n}(t) - G_{n-1}(t)$, with the convention that $g_0(t)=G_0(t)$ (or analogously $G_{-1}(t)=0$).

We will prove that for every even integer $p=2r$, uniformly in $N \geq 0$,
\beq\label{TII_dir_con_together'}
  \| ( \sum_{n=0}^N g_n^2)^{1/2} \|_{L^p[0,1]}  \ll_p
 \| \sum_{n=0}^N g_n \|_{L^p[0,1]} \ll_p \| ( \sum_{n=0}^N g_n^2)^{1/2}  \|_{L^p[0,1]}  .
 \eeq
 We refer to the right-most inequality as the direct inequality, and the left-most inequality as the converse inequality.
We prove these inequalities under three assumptions. First, we assume that the functions $\{g_n\}$ satisfy the Type III superorthogonality condition, so that 
\beq\label{g_int}
 \int_0^1 g_{n_1}(t) \cdots g_{n_{2r}}(t) dt =0
 \eeq
as long as  $n_1> \max \{n_2,\ldots, n_{2r}\}$.
Second, we assume that uniformly in $N$,
\beq\label{max_bound}
 \| \sup_{0 \leq n \leq N} |G_{n-1} | \|_{L^p[0,1]} \ll_p \|  G_N \|_{L^p[0,1]} .
 \eeq
Third, we assume that uniformly in $N$,
\beq\label{p_fact}
\| (\sum_{n=0}^{N} {g_n}^p )^{1/p} \|_{L^p[0,1]} \ll_p \| G_N \|_{L^p[0,1]} .
\eeq

\subsection{The direct inequality}\label{sec_g_formal}
We now prove the direct inequality for the functions $\{g_n\}$. 
Recall that when we proved the direct inequality for the Type I and the Type II case, we could work quite formally, using 
nothing   but   the superorthogonality condition.  Here, we also require (\ref{max_bound}).

Fix $p=2r$ with $r \geq 1$ an integer. One could try to expand $(\sum_{0 \leq n \leq N} g_n)^p = G_N^p$ directly, hoping to apply the Type III property  wherever possible.  
 But this is more subtle to apply than Type I or Type II superorthogonality, since one needs not just a uniqueness property amongst the indices but a magnitude comparison. 
Instead, Paley introduces a telescoping sum
\[ G_N^p =\sum_{n=0}^N (G_n^p - G_{n-1}^p)
	=	\sum_{n=0}^N (( g_n + G_{n-1})^p - G_{n-1}^p)	.
	\]
For $n=0$, the contribution is $g_n^p$, which will lead to an  acceptable contribution in (\ref{G_below}) below.
For each index $n \geq 1$, we write
\[ \int_0^1 (G_n^p  - G_{n-1}^p)   = \int_0^1 (g_n^p + p g_n^{p-1} G_{n-1} + \cdots +{p \choose 2}g_n^2G_{n-1}^{p-2} + pg_n G_{n-1}^{p-1} ) .\]
For each term $g_n^j G_{n-1}^{p-j}$ with $3 \leq j \leq p-1$ there exists  some $\theta(j) \in (0,1)$ such that 
\beq\label{TII_diff}
 |\int_0^1 g_n^j G_{n-1}^{p-j}| \leq ( \int_0^1 g_n^{2} G_{n-1}^{p-2})^{\theta(j)} ( \int_0^1 g_n^p )^{1- \theta(j)}
 \leq  \int_0^1 g_n^{2} G_{n-1}^{p-2} +  \int_0^1 g_n^p 
 ;\eeq
 the first inequality is by H\"older's inequality, and the second is the simple fact that for any exponent $\theta \in (0,1)$, and $A,B \geq 0$, $A^\theta B^{1 - \theta} \leq \max \{A,B\} \leq A + B.$ (Here we use that $p$ is even so that all quantities are non-negative.)
Of course the last inequality in (\ref{TII_diff}) trivially holds for the cases $j=2$ and $j=p$ as well. 

The case $j=1$ could not be argued in this way, but in fact 
the integral of $g_n G_{n-1}^{p-1}$ vanishes by the Type III condition (\ref{g_int}):  the index $n$ of $g_n$ is strictly greater 
than any index $m$ that appears in the expansion $G_{n-1}^{p-1} = ( \sum_{m=0}^{n-1} g_m)^{p-1}$ so that   Type III superorthogonality applies.
In total, we can conclude that for each $n \geq 0$, 
\beq\label{G_below}
 \int_0^1 (G_n^p  - G_{n-1}^p)  \ll_p  \int_0^1 g_n^{2} G_{n-1}^{p-2} +  \int_0^1 g_n^p, 
 \eeq
so that upon summing  over $0 \leq n \leq N$, 
\beq\label{echo}
  \int_0^1 G_N^p  \ll_p \int_0^1 ( \sum_{n=0}^N g_n^2)( \max_{0 \leq n \leq N} |G_{n-1}f|)^{p-2}  + \int_0^1 \sum_{n=0}^N g_n^p . 
\eeq
We can apply H\"older's inequality to the first term, and trivially apply  $\sum_{n=0}^N g_n^p \leq ( \sum_{n=0}^N g_n^2 )^{p/2}$ to the second (again using that $p$ is even), to conclude that 
\[ \| G_N\|_{L^p}^p \ll_p \| ( \sum_{n=0}^N g_n^2)^{1/2} \|_{L^p}^2 \| \max_{0 \leq n \leq N} |G_{n-1} f|\; \|_{L^p}^{p-2} + \| (\sum_{n=0}^N g_n^2)^{1/2} \|_{L^p}^p.\]
By the assumed maximal bound (\ref{max_bound}),
\beq\label{TII_model_ineq}
  \| G_N\|_{L^p}^p \ll_p \| ( \sum_{n=0}^N g_n^2)^{1/2} \|_{L^p}^2 \| G_N\|_{L^p}^{p-2} + \| (\sum_{n=0}^N g_n^2)^{1/2} \|_{L^p}^p.
  \eeq
This is an inequality of the form $A^p \leq B^2 A^{p-2} + B^p$ for non-negative $A,B$. If $A \leq B$, we have already proved the direct inequality, while if $A \geq B$ so that $B/A \leq 1$, we now  deduce from (\ref{TII_model_ineq}) that $A^2 \leq B^2 ( 1 + (B/A)^{p-2}) \ll B^2$. Thus we conclude that 
\[ \|G_N\|_{L^p} \ll_p  \| ( \sum_{n=0}^N g_n^2)^{1/2} \|_{L^p},\]
proving the direct inequality.

\subsection{The converse inequality}\label{sec_g_formal_converse}

A straightforward expansion of 
\[ \int_0^1 (\sum_{n=0}^{N} g_n^2)^{p/2}  \]
is uninformative for applying  a superorthogonality condition; many tuples of indices in the expansion will not have a   uniqueness property or magnitude comparison property. Instead,
Paley employs   the following useful fact, which we call a nonconcentration inequality, following the nomenclature of Gressman \cite{Gre19x} in a related setting.
\begin{lemma}\label{lemma_noncon}
 For any integer $r \geq 1$, for any non-negative real numbers $a_n$ indexed by  a finite set $I$, 
\beq\label{TII_noncon}
 (\sum_{n \in I} a_n)^r \leq (r(r-1))^{r-1}\sum_{n \in I} a_n^r + 2 \sideset{}{^\sharp}\sum_{(n_1,\ldots, n_r) \in I^r} a_{n_1} \cdots a_{n_r}
\eeq
in which $ \sum^\sharp$ indicates that the sum restricts to those ordered tuples $(n_1,\ldots, n_r )\in I^r$ with all pairwise distinct entries.
\end{lemma}

This shows that the dominant values of the function $(n_1,\ldots, n_r) \mapsto a_{n_1} \cdots a_{n_r}$ cannot concentrate on the   zero-set of the function 
\beq\label{Phi_offdiag}
\Phi(x_1,\ldots,x_r) = \prod_{i \neq j} (x_i - x_j),
\eeq
except at the origin $x_1 = \cdots = x_r=0$.
Of course one must allow for the values to concentrate on  $n_1  = \cdots =n_r$, which could dominate if for example there exists $n \in I$ such that $a_n > a_{n'}$ for all $n' \neq n$. 
 Nonconcentration inequalities are broadly useful in arguments involving superorthogonality, including our next section on discrete operators; they are also frequently used in decoupling (see e.g. the first display equation of \cite[p. 653]{BDG16}). We  provide a proof of the inequality at the end of the section.

Paley applies the nonconcetration inequality   to conclude that for $p=2r$, 
\beq\label{TII_step1}
 \int_0^1 ( \sum_{n=0}^{N} g_n^2 )^{p/2} dt \ll_p  \int_0^1 \sum_{n=0}^{N} g_n^p dt + G^\sharp ,
 \eeq
where
\[ G^\sharp = \sideset{}{^\sharp}\sum_{\bstack{(n_1,\ldots, n_r)}{0 \leq n_1, \ldots, n_r\leq N}} \int_0^1 g_{n_1}^2 \cdots g_{n_r}^2 dt,\]
and as usual the superscript $\sharp$ indicates that the sum restricts to tuples with pairwise distinct entries.
By the assumed bound (\ref{p_fact}), the  first term may be bounded by $\ll_p\|G_N \|_{L^p}^p$.
The main work is to show that
\beq\label{TII_S_conc}
G^\sharp \ll_p \int_0^1 G_N^2 ( \sum_{n=0}^{N} g_n^2)^{\frac{p-2}{2}} dt.
\eeq
Once we have proved this, we can apply these two bounds in (\ref{TII_step1}), followed by  H\"older's inequality, to conclude  that 
\[  \| (\sum_{n=0}^{N} g_n^2)^{1/2} \|_{L^p}^{p}  \ll_p   \| G_N \|_{L^p}^p+ 
	\| G_N \|_{L^p}^2 \| (\sum_{n=0}^{N} g_n^2)^{1/2} \|_{L^p}^{p-2}.\]
This is again an inequality of the form $A^p \ll B^2A^{p-2} + B^p$, so that an argument analogous to that applied to (\ref{TII_model_ineq}) confirms  that the converse inequality holds.

We now demonstrate Paley's proof of  (\ref{TII_S_conc})   using Type III superorthogonality. 
If a tuple $(n_1, \ldots, n_r)$ appears in $G^\sharp$ then $n_1, \ldots, n_r$ are all pairwise distinct and in particular there exists a strict ordering of the indices, which without loss of generality we can assume is $n_r < \cdots < n_2 < n_1 \leq N$. In particular, we  could be in a position to apply Type III superorthogonality, except for the fact that each function appearing in $G^\sharp$ is squared. Paley cleverly circumvents this by considering the quantity
\[ \int_0^1 G_N^2 g_{n_2}^2 g_{n_3}^2 \cdots g_{n_{r}}^2  dt\]
for any $n_r  \leq \cdots  \leq n_3 \leq n_2 \leq N$ (so far not assuming strict inequalities). 
Note that we can write $G_N = g_{N} + g_{N-1} + \cdots  + g_{n_2+1} + G_{n_2}$. Thus we can expand the integral above as
\begin{multline*}
 \int_0^1 G_{n_2}^2 g_{n_2}^2 g_{n_3}^2 \cdots g_{n_{r}}^2  dt
 	+ 
\sum_{n_1=n_2+1}^{N}	 \int_0^1g_{n_1}^2 g_{n_2}^2 g_{n_3}^2 \cdots g_{n_{r}}^2  dt
	 \\
	+ 2 \sum_{n=n_2+1}^{N} \int_0^1  g_{n} G_{n_2} g_{n_2}^2 g_{n_3}^2 \cdots g_{n_{r}}^2dt
	  + 2 \sum_{\bstack{n \neq m}{n_2 < n,m \leq N}} \int_0^1 g_ng_m g_{n_2}^2 g_{n_3}^2 \cdots g_{n_{r}}^2 dt.
\end{multline*}
Now Type III superorthogonality shows that the last term vanishes. Furthermore the penultimate term also vanishes by the Type III property:
  we can write  
$G_{n_2}= \sum_{0 \leq m \leq n_2} (g_m-g_{m-1})$, so that expanding the penultimate integral, we can apply (\ref{g_int}) to see that each summand in the expansion vanishes.

The non-negativity of the first term allows us to conclude that 
\beq\label{TII_n_sum}
 \sum_{n_1=n_2+1}^{N}	 \int_0^1  g_{n_1}^2g_{n_2}^2 g_{n_3}^2 \cdots g_{n_{r}}^2 dt \leq  \int_0^1  G_N^2g_{n_2}^2 g_{n_3}^2 \cdots g_{n_{r}}^2dt.
 \eeq
Now we consider the strictly ordered tuples that appear in $G^\sharp$; in each of these there is a unique largest element in the tuple, which will play the role of $n_1$ above. In particular, summing (\ref{TII_n_sum}) over all possible values of $n_r < \cdots < n_2 \leq N$, 
we see that 
\[ G^\sharp \ll_p \int_0^1 G_N^2 (  \sum_{n_r < \cdots < n_2 \leq N}g_{n_2}^2 \cdots g_{n_r}^2 ) dt \ll_p \int_0^1 G_N^2 ( \sum_{n=0}^{N} g_n^2)^{\frac{p-2}{2}} dt,
\]
where the last inequality follows by   non-negativity of the functions $g_n^2$. 
This verifies (\ref{TII_S_conc}), and the   converse inequality in (\ref{TII_dir_con_together'}) follows.

 \subsection{Application to the Walsh-Paley setting}
 We have proved in a formal setting that the direct and converse inequalities in (\ref{TII_dir_con_together'}) hold for a sequence $\{g_n\}$ of partial sum differences, under three assumptions. 
We now indicate why the Walsh-Paley setting satisfies the required assumptions.   

Recall the definition of $f_n$ from (\ref{PW_dyadic}), defined according to a fixed real-valued function $f$. Thus $f_n$ plays the role of $g_n$, the strictly increasing sequence is $\al_n=2^n$, and $S_{2^n}$ plays the role of $G_n$. (We use the convention that $f_0= S_{2^0}$, or analogously $S_{2^{-1}}f=0$.)

  To see that the sequence $\{ f_n\}$  satisfies   the Type III condition,
  suppose that $m_1 > \max \{m_2,\ldots, m_{2r}\}$; we claim that
\beq\label{Paley_TypeII'}
 \int_0^1 f_{m_1}f_{m_2} \cdots f_{m_{2r}} dt =0.
 \eeq
Recall the expansion of the functions $\{w_m\}$ in terms of   the Rademacher functions via (\ref{psi_rad}).
Observe that 
\[f_{m_1} (t) =  S_{2^{m_1}}f (t)- S_{2^{m_1-1}}f (t) =  \sum_{2^{m_1-1} \leq m < 2^{m_1}} c_m(f)w_m(t),\]
 so that $f_{m_1}$ includes $r_{m_1-1}$ as a factor in every summand of this expansion. On the other hand,
after expanding each $f_{m_j}$ with $j \geq 2$ in terms of the Rademacher functions, under the assumption that $m_j < m_1$, we see that $r_{m_1-1}$ does not occur in any of the expansions, and hence does not occur in the expansion of $f_{m_2} \cdots f_{m_{2r}}$. 
Thus the Type III condition (\ref{Paley_TypeII'})  holds for $\{f_n\}$ by means of the Type I property (\ref{Rad_TypeI}) of the Rademacher functions.  (Note that here we crucially used the fact that $m_1$ was a strict maximum;   the Type I or Type II property need not hold for the sequence $\{f_n\}$.)

\begin{remark}
By the same proof method, the sequence $\{f_n\}$  satisfies a stronger condition, that 
$
 \int_0^1 (f_{m_1})^k f_{m_2} \cdots f_{m_{s}} dt =0
$
 if $k$ is an odd positive integer and $m_1 > \{m_2,\ldots, m_s\}$.
 This type has a relation to both Type III (the case where $k=1$) and to Type I.
 See \cite[Lemma 1.4]{Sjo69}.
\end{remark}

\begin{remark}\label{remark_WPR}
In the formal setting of \S \ref{sec_TIII_formal}, if the functions $\{\mu_m\}$ were themselves of Type III, then the $\{g_n\}$ would inherit this property, for any strictly increasing choice of $\{\al_m\}$. But in the Walsh-Paley setting, while the functions $\{w_n\}$ are orthogonal, they do not themselves possess superorthogonality properties for $2r$-tuples with $r\geq 2$; see Appendix A.
The proof that the differences $\{f_n\}$ of dyadic sums  are of Type III relies on the precise nature of the expansions of the functions $w_n$ in terms of the Rademacher functions,  and the lacunary choice $\al_m=2^m$. (See \cite{Pal32} for further generalizations to other lacunary sequences.) 
\end{remark}

We next record that the maximal bound (\ref{max_bound}) holds in the Walsh-Paley setting, by \cite[Thm. 1]{Pal32}.
Indeed, Paley observes   that  $S_{2^n}f(t)$ is a (normalized) average of $f$ over an interval of length $2^{-n}$ containing the point $t$, and deduces that $|S_{2^n}f(t)| \leq 2 \Mcal f(t)$ pointwise in $t$, uniformly in $n$, where $\Mcal f$ is the (uncentered) Hardy-Littlewood maximal function of $f$.  By the boundedness of the Hardy-Littlewood maximal function, for all $1<p\leq \infty$, for all $f \in L^p$,  
\beq\label{max_bound_WP}
 \| \sup_n |S_{2^n} f| \|_{L^p[0,1]} \ll_p \|   f \|_{L^p[0,1]} .
\eeq
If we apply this with the function $f$ replaced by $S_{2^{N}} f$, and use the fact that for $n \leq N$,  $S_{2^{n-1}}(S_{2^{N}} f) = S_{2^{n-1}}f$, we see that 
$ \| \max_{0 \leq n \leq N} |S_{2^{n-1}} f|\; \|_{L^p[0,1]} \ll_p \| S_{2^{N}} f\|_{L^p[0,1]},$
verifying (\ref{max_bound}).

 Finally, we verify (\ref{p_fact}).
 Paley observes in \cite[Lemma 7]{Pal32} that for each $2 \leq p \leq \infty$, for all $f \in L^p$,  
\[
\| (\sum_{n=0}^{\infty} {f_n}^p )^{1/p} \|_{L^p[0,1]} \ll_p \| f \|_{L^p[0,1]} .
\]
This holds for $L^2(\ell^2)$ (applying both (\ref{L2_diag}) and the fact that $\{w_m\}$ is a complete orthonormal system) and for $L^\infty(\ell^\infty)$ (deduced from the normalized average observation above), and the general result follows by interpolation.
Now for even $p$, we may truncate the sum on the left-hand side to $0 \leq n \leq N$ and still obtain the inequality, by positivity.  If we apply this truncated inequality  with $S_{2^{N}}f$ in place of $f$, the right-hand side is $\|S_{2^N}f\|_{L^p}$, while the summands on the left-hand side are still $f_n$, since   for $n \leq N$, inside each difference defining $f_n$,  $S_{2^n}(S_{2^{N}} f) = S_{2^n} f$. 
This verifies (\ref{p_fact}).

The formal argument   now applies, and we conclude that for each $p=2r$ an even integer,  there exists constants $c_p, c_p'$ such that uniformly for all $N \geq 1$, 
\beq\label{TII_direct}
\| ( \sum_{n=0}^N f_n^2)^{1/2} \|_{L^p}  \leq c_p ' \| \sum_{n=0}^N f_n \|_{L^p} \leq c_p\| ( \sum_{n=0}^N f_n^2)^{1/2} \|_{L^p}.
 \eeq 
In order to obtain the full inequality (\ref{TII_dir_con_together}) for $p=2r$ from this truncated version, one must apply a limiting argument; we remark on this in Appendix A. There we also mention a further use of the Rademacher functions to then deduce the full case $1<p< \infty $ from the even integer case.

 This concludes our discussion of the Walsh-Paley setting as an example of Type III superorthogonality; we now turn briefly  to a natural variant.

\subsection{Type III'  superorthogonality for Fourier multipliers}\label{sec_TIII'}

A variant of Type III superorthogonality  is the property that for a sequence of functions $\{f_n\}$ indexed by integers $n$,  there exists an integer $c \geq 1$ such that
 \beq\label{super_gen_super_id'}
 \int f_{n_1}\overline{f}_{n_2} \cdots f_{n_{2r-1}} \overline{f}_{n_{2r}}  =0 
 \eeq
as long as: 
 
{\bf  Type III':} the tuple $( n_1,\ldots, n_{2r})$ has the property that there is an $n_j  \geq n_\ell + c$ for all $\ell \neq j$. 
  Any sequence that satisfies Type III superorthogonality also satisfies Type III'  superorthogonality (with $c=1$).
 For a sequence with the Type III' property, an $L^2$ identity such as (\ref{L2_diag})  is no longer a simple consequence. 
    Also, the Type III'  condition is not invariant if the functions are re-ordered. 
 
Let us describe a case in which Type III'  superorthogonality holds, involving  multiplier operators $T$, such that $(Tf)\widehat{\:}(\xi)  = m(\xi) \widehat{f}(\xi)$, with $m$ satisfying the usual hypotheses of the Marcinkiewicz-Mikhlin-H\"ormander theorem (e.g. \cite[Ch 4]{SingInt} or \cite[Ch VI, \S 4.4 and \S 7.6]{SteinHA}).
First we need a standard dyadic decomposition of the $\xi$-space: 
\[ 1 = \sum_j \Psi_j(\xi),\]
where $\Psi_j(\xi) = \Psi(2^j \xi)$, and $\Psi$ is smooth, compactly supported in $1/4 \leq |\xi| \leq 4$. Then by Plancherel's identity, Type III'  superorthogonality holds ($c=4$ will do) for the sequence of functions $f_j  = T_j f$, defined by 
\[(T_j f)\widehat{\:}(\xi)  = m(\xi) \Psi(2^j\xi) \widehat{f}(\xi).\]
The Type III condition fails.
Nevertheless, in  general for a sequence $\{f_j\}$ that satisfies Type III'  superorthogonality with a constant $c$, then for each   $1 \leq m \leq c$ we can define a sequence  by taking $f_j^{(m)} = f_{cj+m}$ as $j$ varies, and then for each $m$ the sequence 
$f_1^{(m)}, f_2^{(m)}, \ldots, f_n^{(m)}, \ldots $ satisfies the Type III condition.

It would be interesting to prove that $\| Tf \|_{L^p} \leq c_p \|f\|_{L^p}$ via superorthogonality. 
Construct the $c$ sequences $\{f_j^{(m)}\}_j$ as above.
 By suitably adapting Paley's arguments for the direct inequality (say for $p=2r$ with $r\geq1$ an integer), one could obtain that for each $1 \leq m \leq c$,
\[ \|\sum_j f_j^{(m)} \|_{L^p} \leq c_p \| (\sum_j |f_j^{(m)}|^2)^{1/2} \|_{L^p}.\] 
 Adding these inequalities would then provide the direct inequality in full, since 
\[ \| \sum_j f_j \|_{L^p} \leq \sum_{1 \leq m \leq c} \| \sum_j f_j^{(m)} \|_{L^p} \leq c_p \sum_{1 \leq m \leq c} \| (\sum_j |f_j^{(m)}|^2)^{1/2} \|_{L^p}
	\leq  c \cdot c_p\| (\sum_j |f_j|^2)^{1/2} \|_{L^p}.\]
In fact, reasoning of this type appeared in a direct inequality proved by C\'ordoba in the study of Bochner-Riesz operators \cite[p. 507]{Cor79}.

Paley proved his converse inequality by again exploiting the Type III condition. Suitably adapting such arguments, one could obtain that for each fixed $1 \leq m \leq c$, 
\[ \| ( \sum_j |f_j^{(m)}|^2)^{1/2} \|_{L^p} \leq c_p \| F^{(m)}\|_{L^p},\]
in which $F^{(m)}$ is defined by 
\[ (F^{(m)})\widehat{\;}(\xi) = \sum_j \Psi_{cj+m} (\xi)\widehat{f}(\xi).\]
Thus $F^{(1)} + F^{(2)} + \cdots + F^{(c)} = f$. However one cannot add the converse inequalities above to get the desired converse inequality
$ \| ( \sum_j |f_j|^2)^{1/2} \|_{L^p} \leq c_p \| f\|_{L^p},$ so this approach fails.

If instead one hoped to adapt Paley's approach to accommodate Type III'  superorthogonality within the proof, a critical point   is that the nonconcentration inequality implies (\ref{TII_step1}) so that  (\ref{TII_n_sum}) suffices.   In the case of Type III'  superorthogonality, such an approach seems to require a stronger nonconcentration inequality. There are many questions in this area, such as: in what circumstances is it true that the dominant off-diagonal terms in an expansion $(\sum_{n \in I} a_n)^r$  do not even occur close (within a $c$-neighborhood) to   the zero set of the function (\ref{Phi_offdiag}); or what other nonconentration inequalities arise when we replace  (\ref{Phi_offdiag}) by some other function?

\subsection{Proof of Lemma \ref{lemma_noncon}: Nonconcentration inequality.}
The nonconcentration inequality of Lemma \ref{lemma_noncon} has also appeared explicitly in other works such as \cite[Lemma 2.3]{IW} or \cite[Lemma 2.35]{MSZK18x}, whose proofs we follow here.
The claim is true if $r=1$, and hence we suppose $r \geq 2$. Note that if 
\beq\label{sqfn_direct_a_fails}
  (\sum_{n \in I} a_n)^r \leq  2 \sideset{}{^\sharp}\sum_{(n_1,\ldots, n_r) \in I^r} a_{n_1} \cdots a_{n_r},
  \eeq
then the nonconcentration inequality holds, and so we next assume that this condition fails, and show that 
\beq\label{sqfn_direct_a_desire}
(\sum_{n \in I} a_n)^r \leq (r(r-1))^{r-1}\sum_{n \in I} a_n^r.
\eeq
In terms of the sequence $\abf = \{a_n\}_n$, this is the claim that $\|\abf\|_{\ell^1} \leq (r(r-1))^{1 - 1/r} \|\abf\|_{\ell^r}$. 

In general we can expand the left-hand side of (\ref{sqfn_direct_a_desire}) as $A_1 + A_2$ 
in which $A_1$ is the contribution from ordered tuples  in which all indices are distinct, while $A_2$ is the remaining contribution, so that $A_2= {r \choose 2} (\sum_{n \in I} a_n^2)(\sum_{n \in I} a_n)^{r-2}$. 
Now by  the assumed failure of (\ref{sqfn_direct_a_fails}), 
$(\sum_{n \in I} a_n)^r > 2A_1$
 so that 
\[ \frac{1}{2} (\sum_{n \in I} a_n)^r \leq (\sum_{n \in I} a_n)^r - A_1 = A_2.\]
Recalling the expression for $A_2$ (and using   non-negativity of the $a_n$), we learn that 
\[ (\sum_{n \in I} a_n)^2 \leq r(r-1)(\sum_{n \in I} a_n^2).\]
We recognize this as the statement   that 
$\|\abf \|_{\ell^1} \leq (r(r-1))^{1/2} \|\abf\|_{\ell^2}.$ 
Since $1 \leq 2 \leq r$, by the logarithmic convexity of $\ell^p$ norms,
$\|\abf\|_{\ell^2} \leq \|\abf\|_{\ell^1}^{1 - \theta} \|\abf\|_{\ell^r}^{\theta}$ for that $\theta \in [0,1]$ defined by
$1/2 = (1-\theta)/1 + \theta/r$. We apply this to bound the $\ell^2$ norm, and conclude that 
$\|\abf\|_{\ell^1} \leq (r(r-1))^{\frac{1}{2\theta}} \|\abf\|_{\ell^r}$, which is the desired inequality.

\section{Type II superorthogonality: discrete operators }\label{sec_TI_discrete}
We now examine the role of Type II superorthogonality in a new setting, that of discrete arithmetic operators acting on functions of $\Z$. 
Discrete operators  gained widespread attention with work of Bourgain \cite{Bou88A,Bou88B,Bou88C,Bou89} on discrete maximal Radon transforms, such as the operator defined for a fixed integer $k \geq 2$ by
\[ Mf(n)=  \sup_{r \geq 1} \left|\frac{1}{r}  \sum_{1 \leq m \leq r} f(n-m^k) \right|.\]
Bourgain's motivation was that proving such an operator is bounded on $\ell^p$ for a certain $p$ implies a pointwise ergodic theorem  for 
$ \frac{1}{r} \sum_{1 \leq m \leq r} T^{m^k}f$ as $r \maps \infty$, for $T$ a measure-preserving transformation acting on functions in the relevant $\ell^p$ space.  (More generally $\Z$ can be replaced by $\Z^d$ and $m^k$ can be replaced by any integer-valued polynomial mapping.)

Bourgain's initial work stimulated  further investigation of discrete operators. Many  singular and maximal integral operators initially defined in the real-variable setting have a clear discrete analogue, but the discrete analogue is often surprisingly difficult to handle, because arithmetic comes into play.
One natural and interesting class of  operators is the family of discrete singular Radon transforms, defined for example in a one-dimensional setting by
\beq\label{R-op}
 Rf(n) =  \sum_{\bstack{m \in \Z}{m \neq 0}} f(n-P(m))\frac{1}{m}
 \eeq
for a fixed integer-valued polynomial $P$
(and more generally with $1/m$ replaced by $K(m)$, with $K$ an appropriate Calder\'on-Zygmund kernel, e.g. \cite[\S 1]{IW}).
The  real-variable analogue suggests that this discrete operator should be bounded on $\ell^p$ for $1<p<\infty$, but this remained out of reach until \emph{tour de force} work of Ionescu and Wainger \cite{IW}, which cleverly combined many analytic and arithmetic ideas.

The Ionescu-Wainger method has been extremely influential, appearing in many subsequent papers on discrete operators. 
We show here that their ideas can be framed in terms of    direct and converse inequalities for a certain family of discrete operators, using Type II superorthogonality. We focus on a simplified setting that   highlights the aspects of their work closest to our present focus.

\subsection{Preliminaries}

To set notation, given a function $f \in \ell^1(\Z)$, define the Fourier transform to be the  1-periodic function 
\[ \widehat{f}(\xi) = \sum_{n \in \Z} f(n) e^{-2\pi i n \xi},\]
which we may regard on the torus, identified with $(-1/2,1/2]$.
Given a  1-periodic function $h \in L^2_{\mathrm{loc}}(\R)$ which we may regard on the torus identified with $(-1/2,1/2]$, the Fourier inverse is the function defined on $\Z$ by
\[ \check{h}(n) = \int_{(-1/2,1/2]} h(\xi) e^{2\pi i n \xi} d\xi.\]
A bounded 1-periodic function $m: \R \maps \C$ defines an operator $f \mapsto (m \widehat{f})\check{\;}$, which is bounded  on $\ell^2(\Z)$ by Plancherel's theorem.
We will also use the Euclidean Fourier transform $(\Fscr f) (\xi) = \int_{\R} f(x) e^{-2\pi i x \xi}dx$ and its corresponding inverse $(\Fscr^{-1} g)(x) = \int_{\R} g(\xi) e^{2\pi i x \xi}d\xi$.

We   say that a bounded, measurable function $m$ is an $L^p(\R)$ multiplier of norm $B_p$ if   the operator $T$ defined by $Tf = \Fscr^{-1}(m  \cdot \Fscr f)$ satisfies $\|Tf\|_{L^p} \leq B_p \|f\|_{L^p}$ for  all $f \in  L^p(\R)$.
Now let us assume that $m$ is an $L^p(\R)$ multiplier of norm $B_p$ that in addition is compactly supported in $(-1/2,1/2]$.
Then the operator $T$ is a convolution operator given by $Tf = f*K$, where $K (x)= (\Fscr^{-1}m)(x) = \int_{\R} m(\xi)e^{2\pi i x \xi} d\xi$, and since we assume $m$ is compactly supported in $(-1/2,1/2]$, $K$ is $C^\infty$ and in particular its restriction $\left. K \right|_{\Z}$ to integers is well-defined. Thus we can obtain from $T$ an operator acting on functions of $\Z$ by defining
\[ T_{\mathrm{dis}}f(n)= \sum_{m \in \Z} f(n-m) K(m).\]
Alternatively, since $m(\xi)$ is supported in $(-1/2,1/2]$, we may naturally 1-periodize it by setting
\[ m_{\mathrm{per}}(\xi) = \sum_{\ell \in \Z} m(\xi  -\ell),\]
and then we can define a discrete operator $f \mapsto (m_{\mathrm{per}} \widehat{f})\check{\;}$ acting on functions $f$ of $\Z$. 
These two procedures result in the same discrete operator $T_{\mathrm{dis}}$, and in particular $m_{\mathrm{per}}(\xi)$ is the Fourier multiplier of $T_{\mathrm{dis}}$, that is $(T_{\mathrm{dis}}f)\widehat{\;} = m_{\mathrm{per}} \widehat{f}$, and $m_{\mathrm{per}}  = (\left. K \right|_{\Z}) \widehat{\;}$. See  \cite[\S 2]{MSW} for these deductions.
 We  now apply these formal notions to a specific setting.
 
 \subsection{The discrete operator}

 Let $m$ be an $L^p(\R)$ multiplier that is compactly supported in $(-1/2,1/2]$.
Fix a finite set $Z$ of positive integers, and let $\Rcal(Z)$ denote the set of irreducible fractions with denominators in $Z$, namely
\[ \Rcal(Z) = \{ a/q : q \in Z, 1 \leq a \leq q,  (a,q)=1\}.\]
(If $Z=\{q\}$ is a singleton set, we will denote $\Rcal(Z)$ by $\Rcal(q)$.)
Fix  $\ep>0$.
Given $f \in \ell^1(\Z)  $, for each $a/q \in \Rcal(Z)$ define $f_{a/q}$ by 
\beq\label{f_dfn}
 \widehat{f}_{a/q}(\xi)  =  \sum_{\ell \in \Z} m(\ep^{-1}(\xi - \ell -a/q)) \widehat{f}(\xi)=m_{\mathrm{per}}(\ep^{-1}(\xi   -a/q)) \widehat{f}(\xi),
 \eeq
 which is well-defined by the discussion above.
We will focus on the operator
\beq\label{R_op}
 f \mapsto \sum_{a/q \in \Rcal(Z)} f_{a/q} .
 \eeq
 
Ionescu and Wainger's main result \cite[Thm. 1.5]{IW}  leading to a proof of the $\ell^p$ boundedness of the operator (\ref{R-op}) is as follows. They show that for any $\del_0>0$ and $N \geq 1$,  there exists an enlargement $Z_N$ of the set $\{1,2,3,\ldots, N\}$, obtained  by including certain additional integers of size at most $e^{N^{\del_0}}$ so that the operator 
(\ref{R_op}), summed over  $a/q \in \Rcal(Z_N)$ and with $\ep< e^{-N^{2\del_0}}$, is bounded on $\ell^p(\Z)$ for every $1<p<\infty$, with operator norm at most  $ C_{p,\del_0} (\log N)^{2/\del_0}$. The most difficult aspects of the proof are (i) allowing any $0<\del_0<1$, and (ii) achieving at most logarithmic dependence on $N$ in  the operator norm.

We will focus on a key building block that underlies this theorem: the case where   $Z$ is a \emph{relatively prime set}, namely
$ \gcd (q,q')=1$ for all $q \neq q'\in Z$.  (To rule out certain vacuous cases, we also assume   that $q>1$ for all $q \in Z$; this is no limitation in applications of the method.)

In this section, we present a proof of the following main result, in which for each fixed even $p=2r$ we assume that $f_{a/q}$ has been defined according to an $L^{2r}(\R)$ multiplier $m$ supported in $(-1/2,1/2],$ as above. We  require the notion of $\om(q)$, the number of distinct prime divisors of an integer $q$. Given any set $Z$ of integers, we define
\[ \Omega(Z) = \max \{ \omega(q) : q \in Z\}.\]
 \begin{thm}\label{thm_discrete_main_goal}
  Let $Z$ be a relatively prime set of integers contained in $ (1, q(Z)]$.
 Then for any integer $r \ge 1$, as long as $\ep < r^{-1}q(Z)^{-2r}$, for all $f \in \ell^{2r}(\Z)$,
 \beq\label{discrete_main_goal}
\| \sum_{a/q \in \Rcal(Z)} f_{a/q}\|_{\ell^{2r}(\Z)} \leq C_{2r} (2^{\Om(Z)})^{1-1/r} \|f\|_{\ell^{2r}(\Z)},
 \eeq
 in which the constant $C_{2r}$ is independent of $Z ,\ep$, and $f$.
 \end{thm}

We have isolated this theorem as a special case underlying \cite[Thm. 1.5]{IW} that  best illuminates the   role that direct and converse inequalities play in their method. See \S \ref{sec_gen} for a few remarks on the  general setting.

\subsection{Overview of the proof: direct and converse inequalities}
\subsubsection{Direct inequality}
 Given a set $Z$ of integers, proving that 
  the set of functions $\{f_{a/q}\}_{a/q \in \Rcal(Z)}$ satisfies some notion of superorthogonality requires  Diophantine properties of the irreducible fractions in $\Rcal(Z)$. In general this assumes some arithmetic structure on $Z$, and in our special case we will exploit the assumption that $Z$ is a relatively prime set.

Suppose we could show that for any tuple $(a_1/q_1,\ldots, a_{2r}/q_{2r})$ of elements $a_i/q_i \in \Rcal(Z)$ that has the uniqueness property,
\beq\label{f_sum_intro}
\sum_{x\in \Z} f_{a_1/q_1}(x) \overline{f}_{a_2/q_2}(x) \cdots  f_{a_{2r-1}/q_{2r-1}}(x) \overline{f}_{a_{2r}/q_{2r}}(x) =0.
\eeq
Then the formal argument in \S \ref{sec_TypeII} would immediately imply a direct inequality for the functions $\{ f_{a/q}\}_{a/q \in\Rcal(Z)}$. However, this strong property does not hold (see Remark \ref{remark_aq_unique}), and as a whole the collection $\{f_{a/q}\}_{a/q \in \Rcal(Z)}$ does not exhibit Type II superorthogonality. Instead we proceed in two steps: we first  show that (\ref{f_sum_intro}) vanishes if the tuple of denominators $(q_1,q_2,\ldots,q_{2r-1},q_{2r})$ satisfies the uniqueness property. Second, we develop a multilinear direct inequality that exploits a uniqueness property amongst numerators. This two-step process results in a more complicated direct inequality, which we now state:

   \begin{prop}[Direct inequality]\label{prop_RdF_L{2r}_Sf}
 Let $Z$ be a relatively prime set of integers contained in $ (1, q(Z)]$.
 Then as long as $\ep < r^{-1}q(Z)^{-{2r}}$,
\[  \| \sum_{a/q \in \Rcal(Z)} f_{a/q} \|_{\ell^{2r}(\Z)}
 \leq  C_{2r}\| ( \sum_{a/q \in \Rcal(Z)} |f_{a/q}|^2)^{1/2} \|_{\ell^{2r}(\Z)} + C_{2r}\| ( \sum_{q \in Z} |\sum_{a/q \in \Rcal(q)} f_{a/q} |^{2r})^{1/{2r}}\|_{\ell^{2r}(\Z)} ,
\]
for a constant $C_{2r}$ independent of $Z, \ep$ and $f$.
 \end{prop} 

Next, we require a converse inequality for each of the terms on the right-hand side.
While we have seen superorthogonality   play a role in the proof of   Khintchine's converse inequality (by duality), and  in Paley's converse inequality, superorthogonality seems to be of no help for the converse inequality for the functions $\{f_{a/q}\}$. Instead, for a converse inequality for the first term, we  use ``uniform spacing'' in the Fourier transform, which enables a square function estimate that is adapted to frequency projections onto arbitrary intervals that are regularly (rather than dyadically) spaced. For the second term, we use the ``method of sampling,'' and arithmetic properties of the set $Z$ of denominators.
 
\subsubsection{The first converse inequality}
 In general, we say a countable collection of real numbers $\{ \xi_j \}$ is   \emph{$\delta$-separated} if any open interval  of length $\delta$ contains at most one point of $\{\xi_j\}$. 
We now state a general result: a converse inequality in $\ell^p$ that holds for a sequence of functions $\{f_j\}$ with $f_j = T_j f$ where $T_j$ has multiplier $m(\xi - \xi_j)$, as long as $\{\xi_j\}$ is a $\del$-separated set, and $m$ is an $L^p$ multiplier  supported on a subinterval of $(-1/2,1/2]$ of diameter sufficiently small relative to $\del$.   (In particular, there can be at most $O(\del^{-1})$ points in a  $\del$-separated set contained in $(-1/2,1/2]$, so  in what follows,  the indices $j$  lie in an appropriate finite set.)

To be precise, if $Tf = f * K$ is a convolution operator bounded on $L^p(\R)$ with distribution kernel $K$, then $\Fscr(K)(\xi) = m(\xi) = \int K(x) e^{-2\pi i x\xi}dx$ is a bounded function, and so $m$ is the $L^p (\R)$ multiplier such that $\Fscr(Tf)(\xi) = m(\xi) \Fscr(f)(x)$; if $T$ has norm $A_p$ on $L^p(\R)$,   $m$ has multiplier norm $A_p$. If in addition we assume $m$ is supported in $(-1/2,1/2]$, then as remarked before we can periodize it to $m_{\mathrm{per}}(\xi)$ and define the discrete operator $T_{\mathrm{dis}}$ with Fourier multiplier $m_{\mathrm{per}}$; then Parseval-Plancherel states that $\|T_{\mathrm{dis}} f\|_{\ell^2(\Z)}^2=\| m \widehat{f}\|^2_{L^2(-1/2,1/2]}.$
In what follows, we also consider for each shift $\xi_j$ in a well-separated set, an operator $T_j$ with multiplier $m(\xi-\xi_j)$. In order to regard either $T$ or $T_j$ as an operator on discrete functions, we must periodize $m(\xi)$ and $m(\xi - \xi_j)$ and define the corresponding discrete operators. However, in order to simplify notation in the following theorem, we also denote the discretization of $T_j$ by $T_j$.

\begin{thm}\label{thm_RdF_LP}
Let $0< \del<1$ and $2 \leq p < \infty$ be given. 
Let $m$ be an $L^p(\R)$ multiplier of norm $A_p$ and assume that $m$ is supported in $|\xi| \leq c_0 \delta$ for a constant $c_0 <1/2$.   Given a point $\xi_j \in (-1/2,1/2]$, let $T_j$ be the operator with multiplier $m(\xi - \xi_j)$.  If a set $\{\xi_j\}$ of points in $(-1/2,1/2]$ is  $\del$-separated, then the corresponding discrete operator $T_j$ has the property that for every $f \in \ell^{p}(\Z)$,
\[ \| ( \sum_j |T_j f|^2)^{1/2} \|_{\ell^p(\Z)} \leq C_p \|f \|_{\ell^p(\Z)},\]
for a constant $C_p$ depending only on $c_0$ and $A_p$.
\end{thm}
 
We prove this using ideas of Rubio de Francia \cite{RdF85}. It   implies a converse inequality for the first term on the right-hand side of Proposition \ref{prop_RdF_L{2r}_Sf}, once we show that the points in $\Rcal(Z)$ are sufficiently well-separated.
\begin{lemma}\label{lemma_sep}
If $Z$ is a set of integers contained in $ (1, q(Z)]$, then $\Rcal(Z)$ is $\del$-separated for all $\del < q(Z)^{-2}$. Moreover,
 $\union_{q \in Z} \Rcal(q)$ is a disjoint union. 
\end{lemma}
\begin{proof}
To show that no open interval of length $\del$ can contain two distinct elements $a/q$ and $a'/q'$ in $\Rcal(Z)$, suppose on the contrary that $|a/q- a'/q'| \leq \del$. Since $a/q$ and $a'/q'$ are both irreducible fractions,   $a,q$ and $a',q'$ are distinct as pairs, so that $aq'-a'q$ is a nonzero integer, implying
\[\frac{1}{qq'} \leq \frac{|aq'-a'q|}{qq'} = \left|\frac{a}{q} - \frac{a'}{q'}  \right|\leq \del.
\]
This implies $\del \geq q(Z)^{-2}$, a contradiction. Thus indeed the set $\Rcal(Z)$ is $\del$-separated. 
This argument also shows that $\union_{q \in Z} \Rcal(q)$ is a disjoint union, since arguing as above shows that $a/q \not\in \Rcal(q')$ for any $q' \neq q$.
\end{proof}

As a result, Theorem \ref{thm_RdF_LP} immediately implies the first converse inequality we require, as long as $\ep$ is sufficiently small.

\begin{prop}[First converse inequality]\label{prop_converse_{2r}_RdF}
Let an integer $r \geq 1$ be fixed. Let $Z$ be a  set of integers contained in $ (1, q(Z)]$, and suppose $\ep<  q(Z)^{-2}$. Then 
\[ \| ( \sum_{a/q \in \Rcal(Z)} |f_{a/q}|^2)^{1/2} \|_{\ell^{2r}(\Z)} \leq C_{2r} \|f\|_{\ell^{2r}(\Z)},\]
for a constant $C_{2r}$ depending only on the $L^{2r}$ norm of the multiplier $m$. 
\end{prop}

Note that for this converse inequality and the one below, we no longer require $Z$ to be a relatively prime set, although we did require this for the direct inequality.

\subsubsection{The second converse inequality}
Treating the second term on the right-hand side of  Proposition \ref{prop_RdF_L{2r}_Sf}  requires a different approach; here we apply the ``method of sampling'' developed in another seminal paper on discrete operators, by Magyar, Stein and Wainger \cite{MSW}. We record the outcome of the method of sampling later in Theorem \ref{thm_maxlp_same_denom}, and state the relevant consequence here; we still call it a converse inequality although it is not strictly speaking for a square function.

\begin{prop}[Second converse inequality]\label{prop_converse_{2r}}
Let an integer $r \geq 1$ be fixed. Let $Z$ be a   set of integers contained in $ (1, q(Z)]$, and suppose $\ep <  q(Z)^{-2}$. Then 
\[ \| (\sum_{q \in Z} | \sum_{a/q \in \Rcal(q)} f_{a/q}|^{2r})^{1/{2r}} \|_{\ell^{2r}(\Z)} \leq C_{2r} (2^{\Om(Z)})^{1-1/r} \|f\|_{\ell^{2r}(\Z)},\]
for a constant $C_{2r}$ depending only on the $L^{2r}$ norm of the multiplier $m$.
\end{prop}

These three main propositions directly imply Theorem \ref{thm_discrete_main_goal}. We now begin the proof of each, starting with the direct inequality.
 
 \subsection{Direct inequality (Step 1): Type II superorthogonality among denominators}
  Proof of the direct inequality  in Proposition \ref{prop_RdF_L{2r}_Sf} requires 
Type II superorthogonality in two steps. In Step 1, we apply Type II superorthogonality amongst tuples of functions 
\[f_{a_1/q_1}, f_{a_2/q_2}, \ldots, f_{a_{2r-1}/q_{2r-1}}, f_{a_{2r}/q_{2r}}\]
 in which the tuple of denominators  $(q_1,q_2,\ldots,q_{2r})$ satisfies the uniqueness property. In Step 2, a multilinear  direct inequality follows from Type II  superorthogonality amongst tuples in which the   numerators  in $(a_1/q_1,a_2/q_2,\ldots ,a_{2r}/q_{2r})$ satisfy the uniqueness property (and at most two denominators take any given value).

 For each $q \in Z$, define
\beq\label{F_q_dfn}
 F_q = \sum_{u \in \Rcal(q)} f_u.
 \eeq
The goal of Step 1 is to prove  the direct inequality
\beq\label{stage1}
  \| \sum_{q \in Z} F_q \|_{\ell^{2r}}  \leq C_r \|  ( \sum_{q \in Z} |F_q|^2)^{1/2} \|_{\ell^{2r}}.
\eeq
It suffices to verify Type II superorthogonality holds for the functions 
$\{F_q\}_{q \in Z}$.
We need only show that for any tuple $(q_1,q_2,\ldots, q_{2r-1},q_{2r})$  that has the uniqueness property,
 \[ \sum_{x \in \Z} F_{q_1}(x) \overline{F_{q_2}}(x) \cdots F_{q_{2r-1}}(x) \overline{F_{q_{2r}}}(x) =0.
 \]
In fact we can show a stronger property  holds term by term: for any   tuple of denominators $(q_1,\ldots, q_{2r})$  with the uniqueness property,  for each   tuple $(u_1, \ldots, u_{2r})$ with $u_i  = a_i/q_i \in \Rcal(q_i)$,
 \beq\label{RdF_L{2r}_orthog_f}
  \sum_{x \in \Z} f_{u_1} \overline{f_{u_2}}\cdots f_{u_{2r-1}} \overline{f_{u_{2r}}} (x) =0.
  \eeq
Upon defining
 \[G_{u_1, \ldots, u_{2r}}(\xi) = (f_{u_1} \overline{f_{u_2}} \cdots f_{u_{2r-1}} \overline{f_{u_{2r}}})\widehat{\;}(\xi),\] 
the identity  (\ref{RdF_L{2r}_orthog_f}) is the statement that  $G_{u_1, \ldots, u_{2r}}(0)=0,$ 
so that it suffices to show that the support of $G_{u_1, \ldots, u_{2r}}$ does not contain the origin.

We recall that $Z$ is a relatively prime set of integers contained in $ (1, q(Z)]$, and  $\ep < r^{-1}q(Z)^{-{2r}}$.
Let $\ep \Qbf$ denote the periodization of the scaled unit interval $\ep (-1/2,1/2]$, that is, 
\[\ep\Qbf = \Union_{\ell \in \Z}    (\ell-\ep/2,\ell+\ep/2].\]
For each $u$, the support of $(f_u)\widehat{\;}(\xi)$ is contained in $\ep \Qbf + u$ and the support of $(\overline{f_u})\widehat{\;}(\xi) = \overline{(f_u)\widehat{\;}(-\xi)}$ is contained in $\ep \Qbf -u$. The function $G_{u_1,\ldots,u_{2r}}(\xi) $ is a convolution of $2r$ such functions and  has support 
  contained in  
\beq\label{RdF_L{2r}_neighborhood}
 {2r}\ep \Qbf + u_1 - u_2 + \cdots+ u_{2r-1}  - u_{2r}.
 \eeq
 Recall that each $u_i = a_i/q_i \in \Rcal(q_i)$. 
 Under the uniqueness assumption, we may assume  without loss of generality that $q_1$ is distinct from $q_2,\ldots,  q_{2r}$. 
Let $a'/q'$ denote the (signed) reduced fraction such that $a_1/q_1  - a'/q'  = u_1 - (u_2 - u_3 + \cdots -u_{2r-1} + u_{2r}) \modd{1}$. Then  $q' \leq q(Z)^{2r-1}$, and since $Z$ is a relatively prime set, $(q_1,q')=1$; since $q_1>1$ this implies $q_1 \neq q'$. In particular, the reduced fractions $a_1/q_1$ and $a'/q'$ are distinct.

Now supposing that the  set (\ref{RdF_L{2r}_neighborhood}) does contain the origin, we would have
$
 |a_1/q_1  - a'/q' | \leq r\ep.
$
But since $a_1q'-a'q_1$ is a nonzero integer, this would imply that
\beq\label{RdF_L{2r}_contradiction}
 \frac{1}{q_1q'} \leq \left| \frac{a_1q' - a'q_1}{q_1q'} \right| =  \left|\frac{a_1}{q_1}  - \frac{ a'}{q'} \right| \leq r\ep
 \eeq
and hence $r \ep> q(Z)^{-2r} $, which contradicts our assumption that
 $\ep < r^{-1}q(Z)^{-{2r}}$.
We conclude that $0$ does not lie in the support of $G_{u_1, \ldots,  u_{2r}}$. 
This verifies the superorthogonality property, and completes the proof of the direct inequality (\ref{stage1}) in Step 1.

\begin{remark}\label{remark_aq_unique}
Here we can see that we cannot verify (\ref{f_sum_intro})  if we merely assume the tuple of rationals $(a_1/q_1,\ldots, a_{2r}/q_{2r})$ has the uniqueness property. Indeed, if $q_i=q$ for all $i=1,\ldots, {2r}$ but  $a_1=a_2-a_3 + \cdots -a_{2r-1}+a_{2r}$ with $a_1 \notin\{a_2,a_3,\ldots, a_{2r-1},a_{2r}\}$, then (\ref{RdF_L{2r}_neighborhood}) could   contain the origin. Compare this to   Step 2 below, in which we use an $r$-linear formulation to ensure that no more than two denominators can share the same value.
\end{remark}
 
The direct inequality (\ref{stage1}) in terms of  the functions $\{F_q\}$ is not yet sufficient for the purposes of proving Theorem \ref{thm_discrete_main_goal}, since our converse inequality in Theorem \ref{thm_RdF_LP} does not apply directly to operators such as $F_q$.  (This is because there is no single multiplier $M(\xi)$ such that for every $q$, $F_q$ can be defined according to a multiplier that is a shift of $M(\xi)$. We can for example see this  from the basic observation  that   as $q$ varies, the number of summands in $F_q$ varies.) 

Thus we proceed with a second step: we expand the right-hand side of (\ref{stage1}) and apply the non-concentration inequality of Lemma \ref{lemma_noncon}. This yields
\beq \label{RdF_L{2r}_dist_q}
\| (\sum_{q \in Z} |F_q|^2)^{1/2} \|_{\ell^{2r}(\Z)}^{2r} = 
	\sum_x (\sum_{q \in Z}  |F_q|^2)^r
\leq C_r
\sum_x\sum_{q \in Z}  |F_q|^{2r} +C_r\sum_x \sideset{}{^\sharp}\sum_{(q_1,\ldots, q_r )\in Z^r} |F_{q_1}|^2 \cdots |F_{q_r}|^2,
\eeq
in which  $ \sum^\sharp$ indicates that the sum restricts to those ordered tuples $(q_1,\ldots, q_r)$ with all pairwise distinct entries.
The first ``diagonal'' term we recognize as the second term on the right-hand side in Proposition   \ref{prop_RdF_L{2r}_Sf}.
The second ``off-diagonal'' term we will treat further, by applying a multilinear direct inequality for functions with Type II superorthogonality.

\subsection{A multilinear direct inequality via Type II superorthogonality}

We now show  that   Type II superorthogonality implies an $r$-multilinear direct inequality in $L^{2r}$.
To work in full generality, we let 
 $\Ucal$ be a finite index set, and for each of $j=1,\ldots ,r$ we suppose we are given a set $\{ g_u^{(j)} \}_{u \in \Ucal}$   of functions in $L^{2r}$; as usual in such formal arguments we refer to  $L^{2r}(X,d\mu)$, which we could take for example to be $X=\R$ or $(-1/2,1/2]$ with Lebesgue measure, or $\Z$ with counting measure.  
 
  \begin{prop}[Multilinear direct inequality]\label{prop_maxlp_unique_rlinear}

For every integer $r\geq 1$ there exists a constant $C_r$ such that the following holds.
For each $1 \leq j \leq r$, let  $\{g_u^{(j)}\}_{u \in \Ucal}$ be a  sequence of functions in $L^{2r}.$ 
Suppose that for every $2r$-tuple of indices 
$(u_1, u_2, \ldots, u_{2r}) \in  \Ucal^{2r} $
that has the uniqueness property,  
\[
\int  g_{u_1}^{(1)} \overline{g}^{(1)}_{u_2} \cdots g_{u_{2r-1}}^{(r)} \overline{g}_{u_{2r}}^{(r)}   = 0  .
\]
Then
\[
   \| \prod_{j=1}^r (\sum_{u \in  \Ucal} g^{(j)}_u)^{1/r}  \|_{L^{2r}}  \leq C_r  \| \prod_{j=1}^r  \left( (\sum_{u \in  \Ucal} |g^{(j)}_u|^2)^{1/2}  \right)^{1/r}\|_{L^{2r}}.\]
\end{prop}
 
In the proof, it will be useful to have a notation for the tuple $(u_1,u_2,\ldots, u_{2r})$ that makes it more visible which of these indices are applied to the $j$-th collection of functions $g_u^{(j)}$, for $j=1,\ldots, r$. 
Thus we will now denote any such tuple with the notation 
\[(u_1(0), u_1(1), u_2(0), u_2(1), \ldots, u_r(0), u_r(1)).\] 
We will again use the convention that a tuple is an ordered sequence of entries, while the set 
$\{u_1(0), u_1(1), u_2(0), u_2(1), \ldots, u_r(0), u_r(1)\}$
 denotes the unordered set of distinct elements appearing in the tuple. 
 
We require a sorting lemma based on the uniqueness property.
\begin{lemma}\label{lemma_Hall_marriage}
Let $r \geq 1$ be a fixed  integer, and let $(u_1(0), u_1(1), \ldots, u_r(0), u_r(1))$ be a $2r$-tuple of integers that does not have the uniqueness property. Then there exists a function $\kappa : \{1, \ldots, r\} \mapsto \{0,1\}$ so that as sets, 
\begin{align*}
\{ u_1(0), u_1(1), \ldots, u_r(0), u_r(1) \} &= \{ u_1(\kappa(1)),  \ldots, u_r(\kappa(r))\}\\
& = \{u_1(1-\kappa(1)),  \ldots, u_r(1-\kappa(r))\}.
\end{align*}
\end{lemma}
Let us defer the proof of this momentarily, and see why it suffices for proving the multilinear direct inequality.
We raise both sides of the claimed inequality to the $2r$-th power; then the left-hand side may be expanded as 
\beq\label{LHS}
 \sum_{(u_1(0),u_1(1), \ldots, u_r(0),u_r(1)) \in  \Ucal^{2r}}  \int  \prod_{j=1}^r (g_{u_j(0)}^{(j)}  \overline{g}_{u_j(1)}^{(j)}) .
 \eeq
For any tuple $(u_1(0),u_1(1), \ldots, u_r(0),u_r(1))$ with the uniqueness property, the corresponding integral vanishes, by the assumed superorthogonality.
Define for any subset $A \subseteq \Ucal$ the function
\[ 
S_A(x) = \sum_{\bstack{(u_1(0),u_1(1), \ldots, u_r(0),u_r(1))}{\{u_1(0),u_1(1), \ldots, u_r(0),u_r(1) \} = A}} g_{u_1(0)}^{(1)}(x)  \overline{g}_{u_1(1)}^{(1)}(x)\cdots g_{u_r(0)}^{(r)}(x)  \overline{g}_{u_r(1)}^{(r)} (x)
.
\]
Thus the left-hand side contribution (\ref{LHS}) is   equal to 
\beq\label{maxlp_Z_sum}
 \sum_{  |A| \leq r} 
 	\int  S_A  ,
\eeq
in which we   need only consider   $|A|\leq r$ since any $2r$-tuple without the uniqueness property contains at most $r$ distinct values. 

 On the other hand, for any set $A \subseteq \Ucal$ with $|A| \leq r$, define the function
 \[ T_A(x) =  \sum_{\bstack{(u_1, \ldots, u_r) }{\{u_1 , \ldots, u_r \} = A}} |g_{u_1}^{(1)}(x) |^2  \cdots |g_{u_r}^{(r)}(x)|^2 .\]
 The multilinear direct inequality will be proved if we can verify that 
 \[ \sum_{|A| \leq r} \int   S_A    \leq C_r^{2r}  \sum_{|A| \leq r} \int   T_A   .\]
Note that once a fixed subset $|A| \leq r$ is chosen, there are at most $d_r$ $2r$-tuples such that the set $ \{u_1(0),u_1(1), \ldots, u_r(0),u_r(1)\}$ is equal to $A$, for some combinatorial constant $d_r$. Thus the inequality above will hold (with $C_r^{2r}  =d_r$) if we can show that 
for each set $A$ with $|A| \leq r$, for each tuple with set $ \{u_1(0),u_1(1), \ldots, u_r(0),u_r(1)\}=A$,
\[ \int  |g_{u_1(0)}^{(1)}  \overline{g}_{u_1(1)}^{(1)}\cdots g_{u_r(0)}^{(r)} \overline{g}_{u_r(1)}^{(r)}|   \leq   \int  T_A   .\]
We apply Lemma \ref{lemma_Hall_marriage}.
to rewrite the left-hand side as 
\begin{multline*}
 \int |g_{u_1(\kappa(1))}^{(1)}\cdots g_{u_r(\kappa(r))}^{(r)} | \cdot |g_{u_1(1-\kappa(1))}^{(1)}\cdots g_{u_r(1-\kappa(r))}^{(r)} | 
\\
 \leq \frac{1}{2} \int |g_{u_1(\kappa(1))}^{(1)}\cdots g_{u_r(\kappa(r))}^{(r)} |^2   
 	+ \frac{1}{2} \int |g_{u_1(1-\kappa(1))}^{(1)}\cdots g_{u_r(1-\kappa(r))}^{(r)} |^2  .
 \end{multline*}
 Here we also used the fact that  $AB \leq (1/2)(A^2 + B^2)$ for $A,B$ non-negative real numbers.
By the lemma, each of the tuples $(u_1(\kappa(1)), \ldots, u_r(\kappa(r)))$ and $(u_1(1-\kappa(1)), \ldots, u_r(1-\kappa(r)))$ 
is a term represented in the sum defining $T_A(x)$, and thus in particular the right-hand side is bounded above by $\int T_A  $, as desired.

We now return to the proof of the sorting property in Lemma \ref{lemma_Hall_marriage}, using an argument appearing in \cite[Lemma 2.22]{MSZK18x}. 
We will denote the set $\{ u_1(0), u_1(1), \ldots, u_r(0), u_r(1) \}$ by $A$; since the tuple does not have the uniqueness property, we know that $|A| \leq r$. Let us first see that we need only prove the lemma in the case $|A|=r$. In fact, if for some $s \leq r$  the lemma holds for all sets of cardinality $s$, then the lemma is also proved for all sets with cardinality $s-1$. For suppose that the set of values appearing in the tuple is $A=\{a_1,\ldots, a_{s-1}\}$, with $s-1<r$. Then in the $2r$-tuple $(u_1(0), u_1(1), \ldots, u_r(0), u_r(1))$, one value (say $a_i$) must appear at least four times, or two distinct values (say $a_i$ and $a_j$) must each appear at least three times. We construct a new $2r$-tuple in the first case by changing two occurrences of $a_i$ to a new value $a_s \not\in A$, and in the second case by changing one occurrence of $a_i$ to $a_s$ and one occurrence of $a_j$ to $a_s$. This new tuple $(u_1(0)', u_1(1)', \ldots, u_r(0)', u_r(1)')$ does not have the uniqueness property, and takes $s$ distinct values in the set $A' = A \union \{a_s\}$. The version of the lemma assumed to hold for cardinality $s$ sets now applies, and the map $\kappa$ it provides shows that 
$\{u_1(\kappa(1))',\ldots, u_1(\kappa(1))'\}= A'$ and $\{u_1(1-\kappa(1))',\ldots, u_1(1-\kappa(1))'\} = A'$.
As a consequence, we deduce that
$\{u_1(\kappa(1)),\ldots, u_1(\kappa(1))\} = A$ and $\{u_1(1-\kappa(1)),\ldots, u_1(1-\kappa(1))\} = A$, as desired. 

Now we prove the lemma in the case $|A|=r$, so that each value in $A$ is taken by precisely two elements in the tuple. We can construct a bipartite graph as follows. One set of vertices represents the set of indices $\{1,\ldots, r\}$ and the other set of vertices represents the set of values $\{a_1,\ldots, a_r\}$. We will connect a vertex $i$ and a vertex $a_i$ with an edge if $u_i(0) =a_j$, and with another edge  if $u_i(1) = a_j$. In particular, every vertex in this finite bipartite graph is associated to precisely 2 edges. It follows that the graph is a union of finite cycles, each with an even number of edges. 
In each such cycle, we color the edges red and blue, alternately. In particular, each vertex corresponding to an index $i \in \{1,\ldots, r\}$ has a red edge and a blue edge. We will define $\kappa(i)$ to be the value in $\{0,1\}$ such that the edge between the vertex representing the index $i$ and the vertex representing the value $u_i(\kappa(i))$ is red. Since each vertex corresponding to a value $a_i$ has a red edge and a blue edge, this map has the desired property, and the lemma is proved.

This completes the verification of the lemma, and hence of the multilinear direct inequality.

\subsection{Direct inequality (Step 2): Type II superorthogonality among numerators}
We now apply the multilinear direct inequality to the setting of our functions $\{f_{a/q}\}$.
For any integer $q \in Z$, define the square function  
  \[S_q(f) = (\sum_{a/q \in \Rcal(q)} |f_{a/q}|^2)^{1/2}.\]

  \begin{lemma}[Multilinear direct inequality]\label{lemma_RdF_L4_orthog_numer'}
Let $(q_1,\ldots, q_r)$ be a tuple of distinct integers that are all pairwise relatively prime. Then as long as $\ep< r^{-1}\max\{q_1,\ldots,q_r\}^{-2r}$,
\[ 
\| (F_{q_1})^{1/r} \cdots ( F_{q_r})^{1/r} \|_{\ell^{2r}(\Z)} \leq C_r \| (S_{q_1}(f))^{1/r}\cdots (S_{q_r}(f))^{1/r} \|_{\ell^{2r}(\Z)}.
\]
\end{lemma}

This will follow immediately from  the general inequality in Proposition \ref{prop_maxlp_unique_rlinear}  after we set some notation and verify the appropriate superorthogonality condition.
Define the set $Z' = \{q_1,\ldots, q_r\}$. 
For each $i =1,\ldots, r$, and for each  $u \in \Rcal(Z')$, define 
 \[
 g^{(i)}_u (x) = \onebf_{\Rcal(q_i)}(u)f_{u}(x),
 \]
 so that it detects whether the denominator is $q_i$.
 Then 
\[ 
\| (F_{q_1})^{1/r} \cdots ( F_{q_r})^{1/r} \|^{2r}_{\ell^{2r}(\Z)}
= \sum_{x\in \Z}  |F_{q_1}|^2 \cdots |F_{q_r}|^2 =  \sum_{x\in \Z}  |\sum_{u_1 \in \Rcal(Z')} g^{(1)}_{u_1}(x)|^2 \cdots |\sum_{u_r\in \Rcal(Z')} g^{(r)}_{u_r}(x)|^2.
\]
Thus Proposition \ref{prop_maxlp_unique_rlinear}  provides the inequality claimed in our lemma, as long as 
we can verify that for every tuple $(u_1(0),u_1(1),\ldots, u_{r}(0),u_r(1))$ of elements in $\Rcal(Z')$ with the uniqueness property,
\[ \sum_{x \in \Z} g^{(1)}_{u_1(0)} \overline{g}^{(1)}_{u_1(1)} \cdots  g^{(r)}_{u_{r}(0)} \overline{g}^{(r)}_{u_r(1)}(x) =0.\]
Arguing as in Step 1, this holds as long as the origin is not contained in the set 
\beq\label{RdF_L4_set_v2}
 2r\ep \Qbf + u_1(0) - u_1(1) + \cdots + u_{r}(0) -u_{r}(1).
 \eeq

Recall   that for each $i=1,\ldots, r$, $q_i$ is the denominator of $u_i(0)$ and of $u_i(1)$. 
Without loss of generality, assume  that $u_1(0)$ is distinct from all the others. If we denote $u_1(0)= a_1/q_1$ and $u_1(1) = a_2/q_1$, in particular this means that  $a_0/q_1:=a_1/q_1-a_2/q_1 \neq 0$.
Let $a'/q'$ be the (signed) reduced fraction such that 
$a_0/q_1 - a'/q' = u_1(0) - u_1(1) + (u_2(0) - u_2(1) + \cdots + u_{r}(0) -u_{r}(1)) \modd{1}$.
Note that $q'  \leq q(Z)^{2r-1}$ and $(q_1,q')=1$.  
Thus,  $a_0/q_1$ is distinct from the reduced fraction $a'/q'$,  and $a_0q' - a'q_1$ is a nonzero integer.
Supposing that (\ref{RdF_L4_set_v2}) does contain the origin,  then we would have
\[ \frac{1}{q_1 q'} \leq \left| \frac{a_0q' - a'q_1}{q_1q'} \right| = \left| \frac{a_0}{q_1}  - \frac{a'}{q'} \right| \leq r\ep.\]
This cannot occur if we have chosen $\ep < r^{-1}\max \{q_1,\ldots,q_r\}^{-2r}$. Thus the superorthogonality property holds, and this concludes the proof of Lemma \ref{lemma_RdF_L4_orthog_numer'}.

\begin{remark}\label{remark_two_q}
Note that here it was critical that at most two of the rationals $u_i(j)$ shared the same denominator; pairwise distinct  fractions $a_1/q, \ldots ,a_j/q$ with $j \geq 3$ could have signed sum $a_1/q - a_2/q + \cdots (-1)^{j+1} a_j/q$ equal to zero, in which case the above argument would fail. This is why the application of the non-concentration inequality, and the $r$-linear formulation, is required.
\end{remark}
  
  Now we may complete the proof of the direct inequality in Proposition  \ref{prop_RdF_L{2r}_Sf}.
  Recall  from our application of the non-concentration inequality  in (\ref{RdF_L{2r}_dist_q}) that
\[  \| (\sum_{q \in Z} |F_q|^2)^{1/2} \|_{\ell^{2r}(\Z)}^{2r} \leq C_r
\sum_x\sum_{q \in Z}  |F_q|^{2r} + C_r  \sideset{}{^\sharp}\sum_{(q_1,\ldots, q_r) \in Z^r}  \sum_x|F_{q_1}|^2 \cdots |F_{q_r}|^2.\]
Since $Z$ is a relatively prime set, we may apply Lemma \ref{lemma_RdF_L4_orthog_numer'} to each term in the restricted sum over $q_1,\ldots, q_r$, so that
\begin{align*}
 \sideset{}{^\sharp}\sum_{(q_1,\ldots, q_r )\in Z^r}  \sum_x|F_{q_1}|^2 \cdots |F_{q_r}|^2  
	&\leq  C_r  \sideset{}{^\sharp}\sum_{(q_1,\ldots, q_r) \in Z^r} \sum_x  S_{q_1}(f)(x)^2 \cdots S_{q_r}(f)(x)^2  \\
	& \leq C_r\sum_x ( \sum_{q  \in Z} \sum_{a/q \in \Rcal(q)} |f_{a/q}|^2)^r  \\
	& = C_r\| (\sum_{q  \in Z} \sum_{a/q \in \Rcal(q)} |f_{a/q}|^2)^{1/2} \|_{\ell^{2r}}^{2r} = C_r\| ( \sum_{a/q \in \Rcal(Z)} |f_{a/q}|^2)^{1/2} \|_{\ell^{2r}}^{2r}.
 \end{align*}
     We have proved Proposition \ref{prop_RdF_L{2r}_Sf}, the direct inequality in $\ell^{2r}$ for the functions $\{f_{a/q}\}$.

  \subsection{The first converse inequality}

We now turn to the first converse inequality of Theorem \ref{thm_RdF_LP}, which we have seen immediately implies Proposition  \ref{prop_converse_{2r}_RdF}. We first prove a version in which the multiplier is a $C^\infty$ function supported in $(-1/2,1/2]$, which we will call $\Psi(\xi)$, with corresponding operator denoted by $S_j$. As mentioned before Theorem \ref{thm_RdF_LP},   in order to define the discrete operator associated to $S_j$, we must first periodize the multiplier $\Psi(\xi- \xi_j)$ to $\Psi_{\mathrm{per}}(\xi-\xi_j)$. In order to simplify notation in the statement below, we denote both $S_j$ and its associated discrete operator by $S_j$.  
\begin{thm}\label{thm_RdF_Psi}
Let $\Psi$ be a $C^\infty$ function that is compactly supported in $(-1/2,1/2]$, and fix $0< \del<1$. Given a point $\xi_j \in (-1/2,1/2]$, let $S_j$ be the operator with multiplier $\Psi( \del^{-1}(\xi - \xi_j))$. If a set $\{ \xi_j \}$ of points in $(-1/2,1/2]$ is  $\del$-separated, then the corresponding discrete operator $S_j$ has the property that for each $2 \leq p \leq \infty$,  
\[ \| ( \sum_j |S_j f|^2)^{1/2} \|_{\ell^p(\Z)} \leq C_p \|f \|_{\ell^p(\Z)}\]
for a constant $C_p$ depending only on $\Psi$ and $p$, and independent of $\del$ or the $\del$-separated set $\{\xi_j\}$.
\end{thm}
We can deduce from this the result for a general $L^p$ multiplier via the Marcinkiewicz-Zygmund inequality, which we proved in Theorem \ref{thm_MarZyg_p} as a consequence of Khintchine's inequality. 
Let us see how this deduction goes. In Theorem \ref{thm_RdF_LP}, $0<\del<1$ is fixed and the given $L^p$ multiplier $m(\xi)$ is supported in $|\xi| \leq c_0 \del$ with $c_0 < 1/2$. We  choose a $C^\infty$ function $\Psi$ supported in $(-1/2,1/2]$ such that $\Psi(\xi) = 1$ for $|\xi| \leq c_0$,   so that $\Psi(\del^{-1}(\xi - \xi_j))=1$ on the support of $m(\xi - \xi_j)$, and we define the operator $S_j$ accordingly with multiplier $\Psi(\del^{-1}(\xi - \xi_j))$. Then $T_j = T_jS_j$ as operators acting on functions of $\R$. Similarly, after periodizing each kernel, we obtain $T_j = T_jS_j$ for the corresponding discrete operators  acting on functions of $\Z$. 
By the variant of the  Marcinkiewicz-Zygmund inequality in part (II) of Theorem \ref{thm_MarZyg_p}, 
  for any sequence of functions   $F_j \in \ell^p(\Z)$,  
\beq\label{RdF_TF_j_MZ}
\| ( \sum_j |T_j(F_j)|^2)^{1/2} \|_{\ell^p} \leq C_pA_p \| ( \sum_j |F_j|^2 )^{1/2} \|_{\ell^p} 
\eeq
for a constant $C_p$.  
In particular, given a function $f$ of $\Z$, set $F_ j  = S_j(f)$, and apply this inequality to $T_j (F_j) = T_jS_j(f)  = T_j(f)$ to obtain
\[ \| ( \sum_j |T_j(f)|^2)^{1/2} \|_{\ell^p} \leq C_pA_p  \| (\sum_j |S_j(f)|^2)^{1/2} \|_{\ell^p}.\]
Then Theorem \ref{thm_RdF_LP} follows from invoking Theorem \ref{thm_RdF_Psi}, in which the resulting constant depends on the choice of $\Psi$, and hence on $c_0$.

It remains to prove 
Theorem  \ref{thm_RdF_Psi}. It claims that an operator maps $\ell^{p}(\Z)$ to $\ell^{p}(\Z; \ell^2(j \in \Z))$, and by interpolation it suffices to prove  it for $p=2$ and $p=\infty$.
The case $p=2$ holds by the Parseval-Plancherel theorem: for each $j$,
\[ \| S_j(f) \|_{\ell^2(\Z)}^2=\| (\Psi_{\mathrm{per}}(\del^{-1}(\xi - \xi_j)) \widehat{f}(\xi))\check{\;} \|_{\ell^2(\Z)}^2 = \int_{(-1/2,1/2]} | \Psi(\del^{-1}(\xi - \xi_j))|^2 |\widehat{f}(\xi)|^2 d\xi.\]
Thus 
\[ \| (\sum_j |S_j(f)|^2)^{1/2}\|_{\ell^2(\Z)}^2 = \int_{(-1/2,1/2]} \sum_j |\Psi(\del^{-1}(\xi - \xi_j))|^2 |\widehat{f}(\xi)|^2 d\xi \leq \| \Psi\|_{L^\infty}^2 \int_{(-1/2,1/2]} |\widehat{f}(\xi)|^2 d\xi,\]
since at most one summand $\Psi(\del^{-1}(\xi - \xi_j))$ is non-zero for each $\xi$, in view of the  $\del$-separation  of the set $\{\xi_j\}$. 
By applying Parseval-Plancherel again, we see the right-most side is $\| \Psi\|_{L^\infty}^2\|f\|_{\ell^2(\Z)}^2$, verifying the case $p=2$.

To establish the case $p=\infty$, we use the following general observation, an application of duality. 
Let $\{F_j\}$ be a set of functions on $\Z$. If we can prove that $\| \sum_j \al_j F_j \|_{\ell^\infty(\Z)} \leq C$ for all sequences of complex numbers $\al_j$ with $\sum_j |\al_j|^2 = 1$, then it follows that
\[ \| (\sum_j |F_j|^2)^{1/2} \|_{\ell^\infty(\Z)}  \leq C,\]
with the same constant $C$.
To see this, 
fix $x \in \Z$. By assumption, the sequence of values $\{F_j(x)\}_j$  satisfies $| \sum_j \al_j F_j(x) | \leq C$ for all sequences $\{\al_j\} \in \ell^2(j \in \Z)$ with $\ell^2$ norm 1. By duality, the sequence $\{ F_j(x)\}_j$ thus lies in $\ell^2(j \in \Z)$, with $ (\sum_j |F_j|^2(x))^{1/2} \leq C$. Since this holds uniformly for every $x$, the claim follows.
 
Thus in order to prove the $\ell^\infty$ case of Theorem \ref{thm_RdF_Psi}, it suffices to prove that there is a constant $C$ such that for every $f \in \ell^\infty$ with $\|f\|_{\infty} \leq 1$, for every sequence $\{a_j\}$ with $\sum_j |a_j|^2 = 1$,  
\beq\label{RdF_Sjf_sum}
| \sum_j a_j S_j(f)(n) | \leq C
\eeq
for every $n \in \Z$. Recall that by $S_j$ we denote the discrete operator with Fourier multiplier $\Psi_{\mathrm{per}}(\del^{-1}(\xi - \xi_j))$. Let $K$ denote the convolution kernel of the operator $\sum_j a_j S_j$. Then by Young's inequality, (\ref{RdF_Sjf_sum})   would follow from the estimate 
\beq\label{RdF_K_est}
\sum_{n \in \Z} |K(n)| \leq C.
\eeq
Let $K_0$ denote the convolution kernel of the discrete operator with multiplier $\Psi_{\mathrm{per}}(\del^{-1}\xi)$, so that
\[
 K(n)  =  ( \sum_j a_je^{2\pi i \xi_j n}) K_0(n).
\]
Precisely
\[ K_0(n) = \int_{(-1/2,1/2]} \Psi (\del^{-1}\xi) e^{2\pi i \xi n} d\xi,
\]
and consequently
$ |K_0(n)| \leq c \del (1+|\del n|)^{-2}.
$
 Indeed, due to the support of $\Psi$ we   have $|K_0(n)| \leq \del \|\Psi \|_{L^\infty}$ for all $n$. On the other hand, integrating twice by parts provides the bound $c\del |\del n|^{-2}$.

In combination with this bound for $K_0(n)$ we require a lemma about exponential sums.
\begin{lemma}\label{lemma_RdF_exp_sum}
Given a sequence $\{ a_j\}$ of complex numbers and a set $\{\xi_j\}$ of real numbers in $(-1/2,1/2]$, define for every $n \in \Z$,
\beq\label{RdF_Sn_dfn_0}
 \Sbf(n) = \sum_j a_j e^{2\pi i \xi_j n}. 
 \eeq
 Fix $0<\del<1$.
If $\{\xi_j\}$ is a  $\del$-separated set,
then for any interval $J$ of length $1/\del$, we have 
\beq\label{RdF_exp_sum}
\sum_{n \in J} |\Sbf(n)|^2 \leq \frac{c}{\del} \sum_j |a_j|^2
\eeq
for a constant $c$ that is independent  of $\del$, the sequences $\{a_j\}$ and $\{\xi_j\}$, and of the interval $J$.
\end{lemma}
 
\begin{remark}\label{remark_Parseval}
Let a positive integer $N$ be given. If we take $\xi_j = j/N$ for $j=1,\ldots, N$, then $\{\xi_j\}$ is a  $\del$-separated set in $(-1/2,1/2]$ (identified with the torus) with $\del = 1/N$. In this case (\ref{RdF_exp_sum}) becomes the identity 
\[ \sum_{j=1}^N |\Sbf(n)|^2 = N \sum_{j=1}^N |a_j|^2,\]
which is the Parseval-Plancherel identity for the group $\Z/N\Z$.
 \end{remark}

\begin{proof}[Proof of Lemma \ref{lemma_RdF_exp_sum}]

We will prove (\ref{RdF_exp_sum}) by inserting a smooth weight. We fix an auxiliary function $\phi$, such that $\phi (x) \geq 0$ for all $x\in \R$, $\phi(x) \geq 1$ for $|x| \leq c_1$ for some $c_1 >0$, and $\phi(x)  = \widehat{\Phi}(x)$, where $\Phi$ is $C^\infty$ and is supported in $(-1/2,1/2]$. (To construct $\phi$, let $\Phi_1 \in C^\infty$ be supported in $|x| < 1/4$ with $\int\Phi_1(x) dx  > 1$. Then let $\Phi_1^* (x) =\overline{\Phi_1}(-x)$ and define $\Phi = \Phi_1 * (\Phi_1^*)$, so that $\phi = | \widehat{\Phi_1}|^2$.   In particular, $\phi(0) = | \widehat{\Phi_1}(0)|^2 > 1,$ so that this inequality also holds in a small neighborhood of the origin.)
With this definition, we claim
\beq\label{RdF_phi_identity}
 \sum_{n \in \Z} e^{2\pi i n u} \phi(n\del) = \del^{-1} \Phi(\del^{-1}u) .
\eeq
Indeed, the property  $\phi = \widehat{\Phi}$ yields that for any $u \in \T$,

\begin{align*}
\sum_{n \in \Z} e^{2\pi i n u} \phi(n\del) 
&= \sum_{n \in \Z} e^{2\pi i nu} \int_\R \Phi(x)e^{-2\pi i nx \del}  dx  \nonumber \\
	& = \sum_{n \in \Z} e^{2\pi i nu} \int_\R \del^{-1} \Phi(\del^{-1}x) e^{-2\pi i nx} dx. \label{RdF_exp_sum4}
\end{align*}
However, for $\del \leq 1$, then $\Phi(\del^{-1}x)$ is supported in $(-1/2,1/2]$, and the last sum is the Fourier series expansion of the function $\del^{-1}\Phi(\del^{-1} \cdot  )$ in the interval $(-1/2,1/2]$, and so by the Fourier inversion formula the quantities on each side of the identity   are equal to $\del^{-1} \Phi(\del^{-1}u)$, as claimed. 

We turn to the proof of (\ref{RdF_exp_sum}). It suffices to consider the case in which the interval $J$ is taken to be $\{ n : |n| \leq c_1/\del\}$ with $c_1$ as above. 
Indeed, once we have proved the result for such an interval $J$, it automatically holds for every translate $J+h$, because the coefficients $a_j$ are then simply replaced by   $a_j e^{2\pi i \xi_j h}$, which have the same $\ell^2$ norm. Moreover any interval   of length $1/\del$ is covered by at most a bounded number of such translates (depending only on $c_1$). 

Having reduced to this case, we observe that
\beq\label{RdF_exp_sum2}
\sum_{|n| \leq c_1/\del} |\Sbf(n)|^2 \leq \sum_{n \in \Z} |\Sbf(n)|^2 \phi(n\del).
\eeq
On the other hand, squaring gives 
\[
|\Sbf(n)|^2 = \sum_{j,j'} a_j \overline{a_{j'}} e^{2\pi i (\xi_j - \xi_{j'})n}.
\]
We insert this in (\ref{RdF_exp_sum2}) and then apply (\ref{RdF_phi_identity})  with $u = \xi_j - \xi_{j'}$,   to conclude that
\beq\label{RdF_Squash_Diag}
 \sum_{|n| \leq c_1/\del} |\Sbf(n)|^2 \leq \del^{-1} \sum_{j,j'} a_j \overline{a_{j'}} \Phi (\del^{-1}(\xi_j - \xi_{j'}))
	= c \del^{-1} \sum_{j} |a_j|^2 
	\eeq
	with $c= \Phi(0)$, since $\Phi(\del^{-1}(\xi_j - \xi_{j'}))=0$ for all $j \neq j'$, by the assumed $\del$-separation of the set $\{\xi_j\}$. The lemma is  proved.

\end{proof}

We apply Lemma  \ref{lemma_RdF_exp_sum} to prove the    upper bound (\ref{RdF_K_est}) for $\| K \|_{\ell^1}$. 
Recall that we are given  a sequence $\{a_j\}$ with $\sum_j |a_j|^2=1$, and we define $K$ as before to be the kernel of the operator $\sum_j a_j S_j$, so that $K(n) = \Sbf(n) K_0(n)$. By the upper bound for $K_0(n)$, 
$ |K(n)| \leq c |\Sbf(n)| \del (1+ |\del n|)^{-2},
$
and we will apply Lemma \ref{lemma_RdF_exp_sum} to bound $\Sbf(n)$.
We break the summation in (\ref{RdF_K_est}) over $\Z$ into intervals of length $1/\del$, according to a disjoint union 
 $ \Z = \Union_{k \in \Z} I_k $
with 
\[ I_k = \{ n : \frac{k-1/2}{\del} < n \leq \frac{k+1/2}{\del} \}.\]
Applying Cauchy-Schwarz to each sum over $n \in I_k$,  the left-hand side of (\ref{RdF_K_est}) is majorized by 
\[ c \sum_{k \in \Z} A_k B_k,\]
 where 
\beq\label{RdF_AkBk}
 A_k  = (\sum_{n \in I_k} |\Sbf(n)|^2)^{1/2}, \qquad B_k = \del ( \sum_{n \in I_k} (1 + |\del n|)^{-4})^{1/2}.
 \eeq
However, $A_k \leq c\del^{-1/2}$ by Lemma \ref{lemma_RdF_exp_sum}, while $B_k \leq c \del^{1/2}$ if $k=0$ and $B_k \leq c \del^{1/2} |k|^{-2}$ if $k \neq 0$.
As a result the sum is majorized by 
$ c( 1 + \sum_{k \neq 0} |k|^{-2}) \leq c'$
and  (\ref{RdF_K_est}) is proved. This concludes the proof of Theorem \ref{thm_RdF_Psi} in the case $p = \infty$. Consequently we have also completed the proof of  the first converse inequality in Theorem \ref{thm_RdF_LP}.

\subsection{The second converse inequality}
The final step  is to prove the second converse inequality in  Proposition \ref{prop_converse_{2r}},  an $\ell^{2r}$ bound for the operator
\[f \mapsto ( \sum_{q \in Z} |\sum_{a/q \in \Rcal(q)} f_{a/q} |^{2r})^{1/2r}.\]
This will require not just separation properties of elements in $\Rcal(Z)$, but more intricate arithmetic information as well.
In particular, we recall the notation
\[ \Omega(Z) = \max \{ \omega(q) : q \in Z\},\]
in which $\om(q)$  is the number of distinct prime divisors of an integer $q$.
We state a general theorem that implies the second converse inequality of  Proposition \ref{prop_converse_{2r}}.
\begin{thm}\label{thm_RdF_diag_F_Lp}
 Let an integer $r \geq 1$ be fixed. Let $Z$ be a finite set of integers contained in $ (1, q(Z)]$ and fix $\del <  q(Z)^{-2}$.  Let $m$ be an $L^{2r}(\R)$ multiplier of norm $A_{2r}$ and assume furthermore that $m(\xi)$ is supported in $|\xi| \leq c_0 \delta$ for a constant $c_0 <1/2$.    Given a point $u \in (-1/2,1/2]$, let $T_{u}$ be the operator with multiplier $m(\xi - u)$. Then the corresponding discrete operators $T_u$ have the property that for all $f \in \ell^{2r}(\Z)$,
\[
  \|(\sum_{q \in Z}  | \sum_{u \in \Rcal(q)} T_{u} f |^{2r})^{\frac{1}{2r}}\|_{\ell^{2r}(\Z)}
	\leq C_{2r}(2^{\Omega(Z)})^{1 - 1/r} \| f\|_{\ell^{2r}(\Z)},
\]
in which $C_{2r}$ is a constant depending only on $c_0$ and $A_{2r}$.
\end{thm}
 Proposition \ref{prop_converse_{2r}} immediately follows, as long as  $\ep$  is sufficiently small that $m(\ep^{-1}(\cdot))$ satisfies the hypotheses of the theorem;   $\ep<  q(Z)^{-2}$ suffices.

Theorem \ref{thm_RdF_diag_F_Lp} is a consequence of the ``method of sampling,'' as developed in seminal work of  \cite[\S 2]{MSW} on the discrete spherical maximal function. Let us recall the main consequence of this principle.
 In the notation introduced at the beginning of the section, the method of sampling shows how the norm of the real-variable operator $T$ with $L^p$ multiplier $m$ controls the norm of the corresponding discrete operator $T_{\mathrm{dis}}$, as long $m$ is supported in $(-1/2,1/2]$. The variant we state here is an arithmetic consequence of this general result when the multiplier is shifted by $b/Q$ for all $1 \leq b \leq Q$, valid in the case that the multiplier has an even smaller support in $(-1/(2Q),1/(2Q)]$. 
\begin{thm}[The method of sampling]\label{thm_maxlp_same_denom}
Let $1 \leq p  \leq \infty$ be fixed. Let $\mu$ be an $L^p(\R)$ multiplier of norm $M_p$ that is   compactly supported in $(-1/2,1/2]$.
Fix an integer $Q \geq 1$ and $\ep< Q^{-1}$.
Then the discrete operator $S$ with Fourier multiplier 
\[ \sum_{b=1}^Q \mu_{\mathrm{per}}(\ep^{-1}(\xi - b/Q)) \]
extends to a bounded operator on $\ell^p(\Z)$, with 
$ \| S  f\|_{\ell^p(\Z)} \leq C_p \| f \|_{\ell^p(\Z)},
$
in which the constant $C_p$ may depend on $p$ and  $M_p$    but is independent of $Q$ and $\ep$. 
\end{thm}
 
 A key accomplishment of this result is that the discrete operator norm is independent of both $Q$ and $\ep$.  Giving a full proof would take us too far afield, so we refer to \cite[Prop. 2.1 and Cor. 2.1]{MSW}.

Note that Theorem \ref{thm_maxlp_same_denom} considers \emph{all} fractions of denominator $Q$ when defining the multiplier, but in our application we   consider only the \emph{irreducible} fractions $\Rcal(Q)$. 
As a first observation, we  deduce a corollary for multipliers defined with only irreducible fractions, albeit with a norm that depends on $Q$.
\begin{cor}\label{cor_sampling}
Let $1 \leq p  \leq \infty$ be fixed. Let $\mu$ be an $L^p(\R)$ multiplier of norm $M_p$ that is   compactly supported in $(-1/2,1/2]$.
Fix an integer $Q \geq 1$ and $\ep< Q^{-1}$.
Then   the discrete operator $R$ with 
Fourier multiplier 
\[ \sum_{u \in \Rcal(Q)} \mu_{\mathrm{per}}(\ep^{-1}(\xi - u)) \]
extends to a bounded operator on $\ell^p(\Z)$, with 
$ \| R  f\|_{\ell^p(\Z)} \leq C_p 2^{\omega(Q)} \| f \|_{\ell^p(\Z)},
$
in which the constant $C_p$ is the constant in Theorem \ref{thm_maxlp_same_denom}.
\end{cor}
Let us see how to deduce this corollary from the theorem. It is convenient to define the counterpart   to $\Rcal(q)$, namely the ``full'' set of fractions
\[ \Fcal(q) = \{a/q: 1 \leq a \leq q\}.\]
  The key to deducing the corollary is a simple identity.
  \begin{lemma}\label{lemma_split}
   Let  $h$ be the periodization of a function compactly supported in $(-1/2,1/2]$.  
Fix an integer $q$ with prime factorization $q=p_1^{\al_1} \cdots p_k^{\al_k}$ with distinct primes $p_1,\ldots, p_k$. Then
  \beq\label{arith_sum}
  \sum_{u \in \Rcal(q)} h(u) =  \sum_{\ep_1, \ldots, \ep_k \in \{0,1\}} (-1)^{|\underline{\ep}|} \sum_{w \in \Fcal(p_1^{\al_1-\ep_1} \cdots p_k^{\al_k - \ep_k})} h(w),
 \eeq
 where for each $\underline{\ep} = (\ep_1,\ldots, \ep_k) \in \{0,1\}^k$ we have defined $|\underline{\ep}| = {\sum_j \ep_j}$.
 \end{lemma}
Notice that for a given integer $Q$, there are $2^{\om(Q)}$ choices for the tuple $\underline{\ep}$. Thus  we can apply this identity to write the multiplier in the corollary as a signed sum of $2^{\om(Q)}$ multipliers, each of which is of the form considered in the theorem. Since the operator norm in the theorem is independent of $Q$, applying the theorem to each of the $2^{\om(Q)}$ terms then proves the corollary. 

\begin{proof}[Proof of Lemma \ref{lemma_split}]
When $q=p^\al$ is  a prime power, (\ref{arith_sum})
is the claim that
\[  \sum_{u \in \Rcal(q)} h(u) =   \sum_{w \in \Fcal(p^{\al} )} h(w) -  \sum_{w \in \Fcal(p^{\al-1} )} h(w) .\]
 This holds since $\Fcal(p^\al) = \Rcal(p^\al) \sqcup \Fcal(p^{\al-1})$ as a disjoint union,
namely
\[\{ 1 \leq a \leq p^\al \} = \{ 1 \leq a \leq p^\al : (a,p^\al)=1\} \sqcup \{ p a' : 1 \leq a' \leq p^{\al-1}\}.\]
   More generally,  we  will apply the facts that if $q_1$ and $q_2$ are relatively prime, then 
\[ \Fcal(q_1q_2) = \Fcal(q_1) + \Fcal(q_2), \qquad \qquad \Rcal(q_1q_2) = \Rcal(q_1) + \Rcal(q_2).\]
These are consequences of the Chinese Remainder Theorem. Here we regard elements in the sets $\Fcal(\cdot)$ and $\Rcal(\cdot)$ modulo 1 (as we may in our application), and we use  the notation that given sets $\Scal_1$ and $\Scal_2$, $\Scal_1 + \Scal_2 = \{ s_1+ s_2 : s_1 \in \Scal_1, s_2 \in \Scal_2\}$.

Thus if $q=p_1^{\al_1} \cdots p_k^{\al_k}$ with distinct primes $p_i$, it follows that 
\[
 \Rcal(q) = \Rcal(p_1^{\al_1} \cdots q_k^{\al_k}) = \Rcal(p_1^{\al_1}) + \cdots + \Rcal(p_k^{\al_k}),
\]
so that
 \[ \sum_{u \in \Rcal(q)} h(u)  
 	 = \sum_{u_1 \in \Rcal(p_1^{\al_1})} \cdots \sum_{u_k \in \Rcal(p_k^{\al_k})} h(u_1 + \cdots +u_k).\]
 Now we apply the   prime power case to each sum over $\Rcal(p_j^{\al_j})$, so that the right-hand side becomes
 \[
      \sum_{\ep_1, \ldots, \ep_k \in \{0,1\}} (-1)^{|\underline{\ep}|} 
       \sum_{\ga_1 \in \Fcal(p_1^{\al_1-\ep_1})} \cdots \sum_{\ga_k \in \Fcal(p_k^{\al_k - \ep_k})} h(\ga_1 + \cdots +\ga_k).\]
Finally noting that $ \Fcal(p_1^{\al_1-\ep_1})+ \cdots +\Fcal(p_k^{\al_k - \ep_k}) = \Fcal(p_1^{\al_1-\ep_1} \cdots p_k^{\al_k - \ep_k}), $
we recognize the right-hand side of (\ref{arith_sum}), and the identity is proved. 
\end{proof}

 While this corollary is useful,   it does not immediately suffice to prove 
  Theorem \ref{thm_RdF_diag_F_Lp}.
Fix $r \geq 1$ and let $T_u$ be the discrete operator associated to the multiplier $m(\xi-u)$, where we recall that $m(\xi)$ is supported in $|\xi| \leq c_0 \del$, with $c_0< 1/2$, $\del \leq  q(Z)^{-2}$. Examine the norm
 \beq\label{2r_expansion}
  \|(\sum_{q \in Z}  | \sum_{u \in \Rcal(q)} T_{u} f |^{2r})^{\frac{1}{2r}}\|^{2r}_{\ell^{2r}(\Z)} 
 =  \sum_{ q \in Z} \| \sum_{u \in \Rcal(q)} T_{u} f \|^{2r}_{\ell^{2r}(\Z)}. 
 \eeq
The operator $\sum_{u \in \Rcal(q)} T_u$ has Fourier multiplier
 \[\sum_{u \in \Rcal(q)} m_{\mathrm{per}}(\xi-u). \]
For each fixed $q \in Z$, $\del$ is sufficiently small that we can apply Corollary \ref{cor_sampling} to conclude that 
\[ \| \sum_{u \in \Rcal(q)} T_{u} f \|_{\ell^{2r}(\Z)} \leq C_{2r} 2^{\om(q)} \| f\|_{\ell^{2r}(\Z)} \]
with $C_{2r}$ depending on the norm $A_{2r}$ but independent of $q$ and $\del$. 
But this does not suffice to prove Theorem \ref{thm_RdF_diag_F_Lp},  since a trivial summation over $q \in Z$  would lead to an unacceptably large operator norm of size $|Z|C_{2r} 2^{\Om(Z)}$.   
Instead, we will prove a version of Theorem \ref{thm_RdF_diag_F_Lp} for a smooth multiplier, and then pass from the $L^p$ multiplier to the smooth multiplier via an arithmetic factorization identity. We state the theorem for the smooth multiplier:

\begin{thm}\label{thm_RdF_diag_F_Lp_smooth}
Suppose $\Psi$ is a $C^\infty$ function that is compactly supported in $(-1/2,1/2]$, and fix $0 < \del < 1$. Given a point $u \in (-1/2,1/2]$ let $S_u$ be the operator with multiplier $\Psi(\del^{-1}(\xi - u))$.   
  
  Let $Z$ be a finite set of integers  contained in $ (1, q(Z)]$ and suppose the set
$ \Union_{q \in Z} \Rcal(q) 
$
is a  disjoint union and is $\del$-separated. Then if $\del<q(Z)^{-2}$  the corresponding discrete operators $S_u$ have the property that for every $2 \leq p \leq \infty$,
\[
  \|(\sum_{q \in Z}  | \sum_{u \in \Rcal(q)} S_{u} f |^{p})^{\frac{1}{p}}\|_{\ell^{p}(\Z)}
	\leq C_p (2^{\Omega(Z)})^{1 - 2/p} \| f\|_{\ell^{p}(\Z)},
\]
in which $C_p$ is a constant depending only on $p$ and $\Psi$.
\end{thm}

To deduce Theorem \ref{thm_RdF_diag_F_Lp} from this, we require the following lemma.

\begin{lemma}\label{lemma_factor}
Fix an integer $q > 1$ and $\del\leq q^{-2}$. If $\mu(\xi)$ and $\phi(\xi)$ are functions compactly supported in $|\xi| \leq \del/2$, then 
\beq\label{RdF_diag_R_m_fact}
 \sum_{u \in \Rcal(q)} \mu(\xi - u) \phi(\xi - u) = ( \sum_{w \in \Fcal(q)} \mu (\xi - w))\cdot ( \sum_{u\in \Rcal(q)} \phi(\xi -u)),
 \eeq
 regarding the elements in $\Fcal(q)$ and $\Rcal(q)$ modulo 1.
Correspondingly, this holds for the periodizations $\mu_{\mathrm{per}}$ and $\phi_{\mathrm{per}}$ as well.
 \end{lemma}

 \begin{proof}
 It is convenient to define $\Dcal(q)$ to be the set of reducible fractions with denominator $q$,  
 \[ \Dcal(q) = \{ a/q : 1 \leq a \leq q, (a,q)>1\}.\]
 Then   $\Fcal(q) = \Rcal(q) \sqcup \Dcal(q)$ as a disjoint union.  
Suppose we can verify two facts: first,
\beq\label{RdF_sum_factor}
 \sum_{u\in \Rcal(q)} \mu (\xi -u )\phi(\xi - u )= ( \sum_{u \in \Rcal(q) } \mu(\xi - u))\cdot ( \sum_{u'\in \Rcal(q)} \phi(\xi - u')),
 \eeq
 and second, 
\beq\label{RdF_sum_vanishes_0}
 ( \sum_{v \in \Dcal(q)} \mu (\xi - v))\cdot ( \sum_{u'\in \Rcal(q)} \phi(\xi -u')) =0.
\eeq
Then we could add (\ref{RdF_sum_vanishes_0}) to the right-hand side of (\ref{RdF_sum_factor}) and  use $\Fcal(q) = \Rcal(q) \sqcup \Dcal(q)$ to conclude the lemma holds.

To verify (\ref{RdF_sum_factor}) it  suffices to show that  the only nonvanishing terms on the right-hand side occur when $u=u'$; due to the support constraints of $\mu,\phi$ it thus suffices to show that the set $\Rcal(q)$ is $\del$-separated. This of course holds as long as $\del<1/q$.

To verify (\ref{RdF_sum_vanishes_0}), first observe that by identifying each ratio $a/q$ in $\Dcal(q)$ (as a point on the real line) with its associated reduced ratio,   each element of $\Dcal(q)$ lies in some set $\Rcal(q')$ with $q'<q$.
It  then suffices to prove that as long as $\del<  q^{-2}$,   any collection of intervals of length $\del$ centered at points in $\Union_{1 \leq q' < q} \Rcal(q')$  is disjoint from any collection of intervals of length $\del$ centered at points in $\Rcal(q)$. 
Assume on the contrary that two such intervals intersect; then there would be a point $\xi$ for which 
\[ |a/q - a'/q'| \leq |\xi - a/q| + |\xi - a'/q'| \leq \del/2 + \del/2,\]
with $(a,q)=1$ and $(a',q')=1$, where $q' < q$. But then $aq' - a'q$ is a nonzero integer, and it would follow that
\[ q^{-2} < (qq')^{-1} \leq |a/q-a'/q' |\leq \del,\]
 which is impossible, under the assumption that $\del \leq  q^{-2}.$
Thus for any fixed $\xi$, for every pair $v \in \Dcal(q), u' \in \Rcal(q)$ on the left-hand side of (\ref{RdF_sum_vanishes_0}) the supports of $\mu(\xi - v)$ and $\phi(\xi - u')$ are disjoint,   thus proving the identity. This proof holds for their periodization as well. 
\end{proof}

In the process of deducing Theorem  \ref{thm_RdF_diag_F_Lp} from Theorem \ref{thm_RdF_diag_F_Lp_smooth},
we apply Lemma \ref{lemma_factor} with $\mu=m$ the $L^{2r}$ multiplier and $\phi = \Psi$ a smooth multiplier. This allows us to replace the sum of  $T_u$ over $u \in \Rcal(q)$ in (\ref{2r_expansion}) by two operators, one summing $T_u$ over $\Fcal(q)$, which can be controlled by the method of sampling, and another summing a smooth multiplier operator over $\Rcal(q)$, which we control using Theorem \ref{thm_RdF_diag_F_Lp_smooth}. 

Precisely, fix $r \geq 1$. Recall that $\del< Z(q)^{-2}$ is fixed  and $m(\xi)$ is supported in $|\xi| \leq c_0 \del$, with $c_0< 1/2$. Choose a smooth function $\Psi$ supported in $(-1/2,1/2]$ and such that  $\Psi(\xi) = 1$ for $|\xi| \leq c_0$, so that $\Psi(\del^{-1}\xi )=1$ on the support of $m(\xi )$.  Since $\del < q(Z)^{-2}$, 
then for any  $q \in Z$, by the choice of $\Psi$ and Lemma \ref{lemma_factor},
\begin{eqnarray*}
  \sum_{u \in \Rcal(q)} m_{\mathrm{per}}(\xi - u)  
  & =& \sum_{u \in \Rcal(q)} m_{\mathrm{per}}(\xi - u)  \Psi_{\mathrm{per}}(\del^{-1}(\xi-u))\\
 	&=&   (\sum_{w \in \Fcal(q)} m_{\mathrm{per}}(\xi-w) ) (\sum_{u \in \Rcal(q)} \Psi_{\mathrm{per}}(\del^{-1}(\xi - u) )).
\end{eqnarray*}
By the method of sampling (Theorem \ref{thm_maxlp_same_denom}),   the operator with Fourier multiplier corresponding to $\sum_{w \in \Fcal(q)} m_{\mathrm{per}}(\xi-w)$ is bounded on $\ell^{2r}(\Z)$, with norm $C_{2r}$ \emph{independent} of $q, \del$. 
Let $S_u$ denote the discrete operator with smooth Fourier multiplier $\Psi_{\mathrm{per}}(\del^{-1}(\xi - u))$.
Applying this in (\ref{2r_expansion}), we see that
\[
   \|(\sum_{q \in Z}  | \sum_{u \in \Rcal(q)} T_{u} f |^{2r})^{\frac{1}{2r}}\|^{2r}_{\ell^{2r}(\Z)} 
 \leq C_{2r}^{2r}  \sum_{ q \in Z} \| \sum_{u \in \Rcal(q)} S_{u} f \|^{2r}_{\ell^{2r}(\Z)} = C_{2r}^{2r}\| ( \sum_{q \in Z} | \sum_{u \in \Rcal(q)} S_u f|^{2r})^{1/2r} \|_{\ell^{2r}(\Z)}^{2r}. \]
We then apply Theorem  \ref{thm_RdF_diag_F_Lp_smooth} to bound this by $C_r' \|f\|^{2r}_{\ell^{2r}}$,
and we have verified Theorem  \ref{thm_RdF_diag_F_Lp}.
	
 \subsubsection{Proof of Theorem  \ref{thm_RdF_diag_F_Lp_smooth}}
 The final remaining step is to prove Theorem \ref{thm_RdF_diag_F_Lp_smooth}.
We introduce the notation  that for a function $g(x,q)$ of $x \in \Z$ and $q \in Z$,   for $1 \leq p < \infty,$
\beq\label{RdF_mixed_norm_dfn}
 \| g(x,q) \|_{\ell^p(\Z \times Z)} = ( \sum_{x \in \Z} \sum_{q \in Z} |g(x,q)|^p)^{1/p},
 \eeq
 and for $p=\infty,$  
\[  \| g(x,q) \|_{\ell^\infty(\Z \times Z)} = \sup_{(x,q) \in \Z \times Z} |g(x,q)|.\]
Then the theorem claims that for each  $2 \leq p \leq \infty$,  
\beq\label{RdF_diag_p}
  \| \sum_{u \in \Rcal(q)} S_u f\|_{\ell^{p}(\Z \times Z)}
	\leq C_p(2^{\Omega(Z)})^{1-2/p} \| f\|_{\ell^{p}(\Z)}.
	\eeq
It therefore suffices to prove a bound  for $\ell^2(\Z \times Z)$ and for $\ell^\infty(\Z \times Z)$, and the intermediate cases follow by interpolation.

For the $\ell^2$ bound, first rewrite
\[ \| (\sum_{q \in Z}  |  \sum_{u \in \Rcal(q)} S_u|^{2})^{\frac{1}{2}}\|_{\ell^2(\Z)}^2=
\sum_{q \in Z}    \|   \sum_{u \in \Rcal(q)} S_u f\|_{\ell^2(\Z)}^2.
 \]
Applying the Parseval-Plancherel identity for each $q$ shows that this is equal to
\[ \sum_{q \in Z}  \|  \sum_{u \in \Rcal(q)} \Psi(\del^{-1}(\xi - u)) \hat{f}(\xi)\|_{L^2(\T)}^2
 =  \int_{(-1/2,1/2]}  | \hat{f}(\xi)|^2\sum_{q \in Z}| \sum_{u \in \Rcal(q)} \Psi(\del^{-1}(\xi - u))|^2d\xi.\]
Since $\del< q(Z)^{-2}$, by Lemma \ref{lemma_sep},  $\union_q \Rcal(q)$ is $\del$-separated (and is a disjoint union), so that for each fixed $\xi$ at most one term is present as $q,u$ vary. Thus we may conclude, after applying Plancherel again, that 
\[ \| (\sum_{q \in Z}  |  \sum_{u \in \Rcal(q)} S_u|^{2})^{\frac{1}{2}}\|_{\ell^2(\Z)}^2  \leq \| \Psi \|^2_{L^\infty(-1/2,1/2]} \| f \|_{\ell^2(\Z)}^2.\]
(Here we used the  $\del$-separation of the points, but no explicitly arithmetic properties.)
 
 For the $\ell^\infty$   bound, it suffices to show that uniformly in $q \in Z$, 
 \beq\label{S_infty}
   \|  \sum_{u \in \Rcal(q)} S_u  f \|_{\ell^\infty(\Z)}
	\leq C2^{\Omega(Z)} \| f\|_{\ell^{\infty}(\Z)},
	\eeq
	for a constant $C$ independent of $q$.  
 Note that if $q$ has prime factorization $q=p_1^{\al_1} \cdots p_k^{\al_k}$,  by Lemma \ref{lemma_split} the Fourier multiplier of this operator can be written as 
 \[ \sum_{u  \in \Rcal(q)} \Psi_{\mathrm{per}}(\del^{-1}(\xi - u)) = 
  \sum_{\ep_1, \ldots, \ep_k \in \{0,1\}} (-1)^{|\underline{\ep}|} \sum_{w \in \Fcal(p_1^{\al_1-\ep_1} \cdots p_k^{\al_k - \ep_k})} \Psi_{\mathrm{per}}(\del^{-1}(\xi-w)).
 \]
 There are $2^{\om(q)} \leq 2^{\Omega(Z)}$ terms on the right-hand side.
 Thus to prove (\ref{S_infty}) it suffices to prove that each of the terms on the right-hand side corresponds to an operator with uniformly bounded norm. That is, it suffices to prove that for any integer $Q \geq 1$,
 \[\| \sum_{u \in \Fcal(Q)} S_u f \|_{\ell^\infty(\Z)} \leq C\|f\|_{\ell^\infty(\Z)}.\]
  Let $K(n)$ denote the kernel of the operator $\sum_{u \in \Fcal(Q)} S_u$, so that it suffices to show that 
 \beq\label{K_bound}
  \sum_{n \in \Z} |K(n)| \leq C.
  \eeq

  There is an analogy to the use of Lemma \ref{lemma_RdF_exp_sum} to prove the $\ell^\infty$ case of Theorem \ref{thm_RdF_Psi}, but now instead of using $\del$-separation of a collection of points $\{\xi_j\}$ we can use precise arithmetic information (compare to Remark \ref{remark_Parseval}).
   We   compute that
 \begin{align*}
  K(n)  & = \sum_{u \in \Fcal(Q)} (\Psi_{\mathrm{per}} (\del^{-1}(\cdot - u)))\check{\;} (n)  \\
  & =   \sum_{u \in \Fcal(Q)} \int_{(-1/2,1/2]} \Psi  (\del^{-1}(\xi - u)) e^{2\pi i n \xi} d\xi\\
  & = ( \sum_{u \in \Fcal(Q)} e^{2\pi i n   u}) \cdot  \int_{\R} \Psi(\del^{-1}\xi) e^{2\pi i n \xi} d\xi  \\
  & = Q \onebf_Q(n) \cdot  \del (\Fscr^{-1}\Psi)(\del n) .
  \end{align*}
 We have applied first the support constraint of $\Psi$,  followed by the fact that 
 \[   \sum_{u \in \Fcal(Q)} e^{2\pi i n   u} = Q\onebf_{Q}(n)
 	 := \begin{cases} Q & \text{if $n \con 0 \modd{Q}$} \\
	 				0 & \text{otherwise}.\end{cases}\]
We see under the change of variables $n=Qm$ that 
\[ \sum_{n \in \Z}|K(n)|  = \sum_{\bstack{n \in \Z}{Q|n}}|K(n)|  = \sum_{m \in \Z}  Q\del |(\Fscr^{-1}\Psi)(\del Qm)|. \]
The last sum is uniformly bounded, so that (\ref{K_bound}) is confirmed. Indeed, since  $\Psi$ is $C^\infty$, its Fourier inverse has rapid decay, so that  $|(\Fscr^{-1}\Psi)(x)| \leq c(1+|x|)^{-M}$ for any $M \geq 1$ of our choice. Breaking the sum into a contribution from small $m$ and from large $m$, we see that each portion is bounded: 
\[ Q\del\sum_{|m| \leq (\del Q)^{-1}}   |(\Fscr^{-1}\Psi)(\del Qm)| \leq Q\del \sum_{|m| \leq (\del Q)^{-1}} O(1) = O(1), \]
and 
\[ Q\del \sum_{|m| > (\del Q)^{-1}}  |(\Fscr^{-1}\Psi)(\del Qm)| \leq Q\del \sum_{|m| > (\del Q)^{-1}} O((\del Q |m|)^{-M}) = O(1). \]
This completes the proof of Theorem \ref{thm_RdF_diag_F_Lp_smooth}, and hence also   the proof of 
the second converse inequality in Theorem \ref{thm_RdF_diag_F_Lp}.

 \subsection{Further remarks on the general setting}\label{sec_gen}
This concludes  our study of direct and converse inequalities for the discrete operator 
  \beq\label{Z_op}
   f \mapsto \sum_{a/q \in \Rcal(Z)} f_{a/q}
   \eeq
   in the case that $Z$ is a relatively prime set.
Our main result Theorem \ref{thm_discrete_main_goal}  was stated in terms of $\ell^{2r}$ bounds, with a constraint on $\ep$ depending on $r$ and the maximum size $q(Z)$ of an element in $Z$.    In the more general setting of Ionescu and Wainger, a parameter $0<\del_0<1$ is specified, and the goal is to construct a set $Z_N$  that is an enlargement of the set $\{1,\ldots, N\}$ with elements at most of size $e^{N^{\del_0}}$, such that the resulting operator (\ref{Z_op})  with $Z=Z_N$ is bounded on $\ell^p$ with norm at most $C_{p,\del_0}(\log N)^{2/\del_0},$ for each $1<p< \infty$. By specifying that $\ep< e^{-N^{2\del_0}},$ the constraint $\ep< r^{-1}q(Z_N)^{-1/2r}$ holds for all $r \geq 1$ (for $N$ sufficiently large relative to $r,\del_0$), and an $\ell^{2r}$ bound may be obtained for the operator. Then by interpolation the result of the theorem holds for all $2 \leq p  <\infty$; duality then shows the result also holds for $1 < p  \leq 2$. Thus the focus of Theorem \ref{thm_discrete_main_goal} on $\ell^{2r}$ bounds is not unduly restrictive.

One crucial aspect of the work of Ionescu and Wainger is that they can allow $\del_0 >0 $ to be arbitrarily small. (The case $\del_0 >1$ can be treated relatively simply.) To handle cases in which $\del_0$ is close to 0, Ionescu and Wainger construct a set   $Z_N$ that is a product set $Z_1 \cdots Z_s$ of relatively prime sets $Z_j$,   with the further property that all elements in $Z_j$ are relatively prime to all elements in $Z_{j'}$ if $j \neq j'$, and with $s$ on the order of $1/\del_0$. The setting of Theorem \ref{thm_discrete_main_goal} illustrates the special case $s=1$. When $s>1$, one  proves a direct inequality using an induction based on computations similar to those we exhibited here. The final direct inequality playing the role of Proposition \ref{prop_RdF_L{2r}_Sf} then   has $2^s$ terms that are various hybrids  of the two types of  terms we exhibited here. 
To prove the converse inequalities for the hybrid terms, one combines appropriately the methods of proof we illustrated here for the first and second converse inequalities, stated as Propositions \ref{prop_converse_{2r}_RdF} and \ref{prop_converse_{2r}}.

A final crucial aspect of the work of  Ionescu and Wainger is that simultaneously with the considerations above, they  must ensure that the factor $2^{\Om(Z_N)}$ appearing in the analogue of Proposition \ref{prop_converse_{2r}} does not exceed the allowed   norm $C_{p,\del_0}(\log N)^{2/\del_0}$. This is difficult, since $Z_N$ must be an enlargement of the set $\{1,\ldots,N\}$ and in particular can include integers with many distinct small prime divisors. Thus as a first step, Ionescu and Wainger separate out from $\Rcal( \{1,\ldots, N\})$ all fractions with denominators divisible by many small prime factors. These fractions  are, roughly speaking, carried along inside the multiplier to which the above method is applied, until the method of sampling is finally applied to also treat these terms. This aspect of the work of Ionescu and Wainger is very interesting from an arithmetic point of view, but the special case we focused on was designed to remove such considerations, in order to illuminate more simply the role that superorthogonality plays.

\section{A return to Type I superorthogonality: diagonal behavior}\label{sec_decoupling}

A refinement of Type I superorthogonality arises  naturally in a question related to decoupling. This refinement is the condition that for every $2r$-tuple $f_{n_1}, \ldots, f_{n_{2r}}$ of (complex-valued) functions from a sequence   $\{f_n\}$,
\beq\label{TIII_vanish'}
\int f_{n_1} \bar{f}_{n_2} \cdots f_{n_{2r-1}} \bar{f}_{n_{2r}}  =0
\eeq
as long as \\
{\bf Type I*:} the tuple $(n_1, n_3 ,\ldots , n_{2r-1})$ is not a permutation of the tuple $(n_2,n_4, \ldots, n_{2r})$.

 In particular, if a sequence $\{f_n\}$ is of Type I*, it is certainly of Type I. 
Under the Type I condition, a direct inequality holds for any   sequence $\{f_n\}$, by the argument given in \S \ref{sec_TypeI} (suitably modified for complex-valued $f_n$). But for Type I* this argument can be refined to show that for any integer $r \geq 1$,  
 \beq\label{TIII_arg}
 \| (\sum_n |f_n|^2)^{1/2} \|_{L^{2r}}^{2r} \leq  \| \sum_{n} f_n \|_{L^{2r}}^{2r}  \leq r!  \| (\sum_n |f_n|^2)^{1/2} \|_{L^{2r}}^{2r}.
 \eeq
Indeed, after expanding the middle term, (\ref{TIII_vanish'}) shows that the only nonvanishing terms can be written as
\[  \int  (\sum_n f_n)^r (\sum_{n'} \bar{f}_{n'})^r 
= \sum_{(a_1,\ldots,a_s)} C(a_1,\ldots,a_s)^2 \int  |f_{n_1}|^{2a_1} \cdots |f_{n_s}|^{2a_s},\]
in which the sum on the right-hand side is over all $s \leq r$, all pairwise distinct $n_1,\ldots, n_s$ in the (finite) index set, and all $(a_1,\ldots,a_s)$ with $a_1 + \cdots + a_s = r$. Here, as before, $C(a_1,\ldots,a_s) = (a_1+\cdots+a_s)!/( a_1! \cdots  a_s!)$.
 On the other hand, 
 \[  \int  (\sum_n |f_n|^2)^r  
= \sum_{(a_1,\ldots,a_s)} C(a_1,\ldots,a_s) \int  |f_{n_1}|^{2a_1} \cdots |f_{n_s}|^{2a_s},\]
in which the sum varies over the same parameters as described above.
The claim follows, since $\max_{(a_1,\ldots,a_s)} C(a_1,\ldots,a_s) \leq r!$ and $\min_{(a_1,\ldots,a_s)} C(a_1,\ldots,a_s) \geq 1$.

 \subsection{An extension operator associated to a nondegenerate curve}
 In this section we demonstrate a family $\{f_n\}$ with Type I* superorthogonality that relates to   both harmonic analysis and number theory.
Let $\ga:[0,1] \maps \R^n$ be a nondegenerate curve, that is,  
\[ \det (\ga'(t), \ga''(t), \ldots, \ga^{(n)}(t)) \neq 0 \quad \text{for every $t \in [0,1]$.}\]
The prototypical example is the moment curve, $\ga(t) = (t,t^2, \ldots, t^n)$. 
We may associate to the curve $\ga$ an extension operator that maps functions on an interval $I \subset [0,1]$ to functions of $\R^n$, by 
\[ E_I f(x) = \int_I e^{2\pi i x \cdot \ga(t)} f(t) dt.\]
Then $E_If$ can be thought of as a function whose Fourier transform is supported as a distribution on a portion of the curve $\{ \ga(t): t \in [0,1]\}$.

Suppose that we dissect $[0,1]$ into a set of disjoint intervals $\{I_n\}_n$. Then we can ask whether the sequence of functions $\{E_{I_n} f\}_n$ satisfies a direct inequality in $L^p$; if it does, this also implies an $\ell^2 L^p$ decoupling inequality (in an appropriate range of $p$). 
This setting has been  considered in recent work of Gressman, Guo, Roos, Yung and the author \cite{GGPRY19x}.
As explained there, to state the relevant direct inequality precisely (so that both sides may be finite),  we must  define the $L^p$ norms with an appropriate weight. There are many choices for such a weight (with comparable outcomes in the setting of decoupling, see e.g. \cite[Lemma 4.1]{BouDem17a}), and we will make a particularly convenient choice here.

 Let $\phi$ be a  non-negative Schwartz function on $\R^n$ with the property that $\phi \geq 1$ on the unit ball centered at the origin, and $\widehat{\phi}$ is supported on the unit ball centered at the origin (see the proof of Lemma \ref{lemma_RdF_exp_sum} for a construction). 
Define $\phi_R(x) = \phi(R^{-n}x)$, and define the weighted norm
\[ \|f\|_{L^p(\phi_R)} = ( \int_{\R^n} |f(x)|^p \phi_R(x) dx)^{1/p}.\]
We may think of this weighted norm as capturing  the average behavior of $f$ over at least a ball of radius $R^n$;
the weight $\phi_R(x)$ has the effect of ``blurring'' the support of the Fourier transform of $|f(x)|^p$ on the scale of an $O(R^{-n})$ neighborhood; see e.g. \cite[\S 8.1.3]{Pie19}.

 The main result of Guo, Gressman, Pierce, Roos and Yung \cite{GGPRY19x} is the following direct inequality:   there exists a constant $C(\ga,n) \in (0,\infty)$ such that for each integer $1 \leq r \leq n$, for  every $R \geq 1$, for every $f \in L^{2r}(\phi_R)$,
\beq\label{TIII_direct}
 \|E_{[0,1]}f\|_{L^{2r}(\phi_R)} \leq C(\ga,n) \| ( \sum_{|I| = R^{-1}}|E_I f|^2)^{1/2} \|_{L^{2r}(\phi_R)}.
 \eeq
Here the summation is over the intervals $I$ in a dissection of $[0,1]$ into subintervals   of length $R^{-1}$. 

 For any fixed $r \geq 1$, such a direct inequality is stronger than the corresponding $\ell^2 L^{2r}$ decoupling inequality, which would replace the norm on the right-hand side by   $( \sum_{|I| = R^{-1}}\|E_I f \|_{L^{2r}(\phi_R)}^2)^{1/2}$ (see e.g. \cite[\S 5.3.2]{Pie19} for a formal comparison). The celebrated work \cite{BDG16} shows that the $\ell^2 L^{2r}$ decoupling inequality is valid in the much larger range of any real $r \leq n(n+1)/2$. See further remarks below, on barriers to extending the current proof of the direct inequality to values $r>n$.

 \subsection{Type I* superorthogonality for   extension operators}
An advantage of the direct inequality (\ref{TIII_direct}) for integers $r \leq n$ is its comparatively simple proof.
Fundamentally, the argument is an application of Type I* superorthogonality: if the Type I* condition holds for  functions $\{ E_I f\}_I$ as $I$ varies over a finite set $\Ical$, then (\ref{TIII_arg}) shows that
 \beq\label{TIII_direct'}
  \|\sum_{I \in \Ical} E_I f \|_{L^{2r}(\phi_R)}  \leq C_r \| ( \sum_{I \in \Ical}E_I f|^2)^{1/2} \|_{L^{2r}(\phi_R)}.
  \eeq
  From this inequality, certain reductions (reviewed momentarily) show that (\ref{TIII_direct}) holds.
  
In this weighted context, the Type I* criterion is the statement that for any tuple $(I_1,I_2,\ldots, I_{2r})$ of intervals in the collection $\Ical$, 
  \beq\label{TIII_cond_ex}
   \int_{\R^n} E_{I_1} f(x) \overline{E_{I_2} f}(x) \cdots E_{I_{2r-1}}f (x) \overline{E_{I_{2r}}}f(x) \phi_R(x) dx=0 
   \eeq
  unless $(I_1, I_3, \ldots, I_{2r-1})$ is a permutation of $(  I_2,I_4, \ldots, I_{2r})$.
  Upon expanding the definition of the extension operators, the integral (\ref{TIII_cond_ex}) is identical to the expression
  \[
   \int_{I_{\mathrm{odd}}} \int_{I_{\mathrm{even}}} R^{n^2} \widehat{\phi}\left( R^n \sum_{i=1}^r (\ga (t_i) - \ga(s_i))\right) f(t_1) \cdots f(t_r) \overline{f(s_1) \cdots f(s_r)} dt_1 \cdots dt_r ds_1 \cdots ds_r;\]
here we let $(t_1,\ldots, t_r) \in I_{\mathrm{odd}} := I_1 \times I_3 \times \cdots \times I_{2r-1}$, and analogously for $(s_1,\ldots, s_r) \in I_{\mathrm{even}}:= I_2 \times I_4 \times \cdots \times I_{2r}$.
This integral will vanish unless $R^n \sum_{i=1}^r (\ga (t_i) - \ga(s_i))$ lies in the support of $\widehat{\phi}(\xi)$, which requires that
\beq\label{TIII_squeeze}
| \sum_{i=1}^r (\ga (t_i) - \ga(s_i)) | \leq R^{-n}.
\eeq
The central technical result of \cite[Prop. 1.3]{GGPRY19x}  is that nondegeneracy of the curve $\ga$ guarantees that  for each integer $1 \leq r \leq n$, as long as the intervals in $I \in \Ical$ are sufficiently well-spaced, (\ref{TIII_squeeze}) can only occur if $(I_1, I_3,\ldots, I_{2r-1})$ is a permutation of $(I_2,I_4,\ldots, I_{2r})$. This verifies that the Type I* condition holds as long as the intervals $I \in \Ical$ are sufficiently well-spaced.

For completeness, we recall a few details of this result, to confirm that it reduces the full proof of (\ref{TIII_direct}) to the case of (\ref{TIII_direct'}).
  Precisely, the nondegeneracy of $\ga$ guarantees the 
existence of constants $\del_0(\ga,n) \leq 1$ and $c_0(\ga,n) \geq 10$ (with $\del_0^{-1},c_0 \in \Z$) such that for any collection $\Ical$ of intervals from a dissection of $[0,1]$ into pairwise disjoint intervals of length $R^{-1}$ with the property that
\beq\label{TIII_collection}
\mathrm{dist}(I,I') \geq c_0(\ga,n)R^{-1} \quad \text{for $I \neq I ' \in \Ical$},
	\quad \text{and} \quad \mathrm{diam}(\Union_{I \in \Ical} I) \leq \del_0(\ga,n),
\eeq
then (\ref{TIII_squeeze}) can only hold if 
$(t_1,\ldots, t_r) \in I_{n_1} \times \cdots \times I_{n_r}$ and 
$(s_1,\ldots, s_r) \in I_{n_1}' \times \cdots \times I_{n_r}'$
where $(I_{n_1},\ldots, I_{n_r})$ is a permutation of $(I_{n_1}',\ldots, I_{n_r}')$.
We conclude that for any  collection $\Ical$  of intervals satisfying (\ref{TIII_collection}), the integral (\ref{TIII_cond_ex}) vanishes unless 
$(I_1, I_3, \ldots, I_{2r-1})$ is a permutation of $(  I_2,I_4, \ldots, I_{2r})$.  
Thus to prove (\ref{TIII_direct}), one begins with a dissection of $[0,1]$ into intervals of length $R^{-1}$, and then   cuts this dissection into $\leq \del_0^{-1}$ subcollections, each lying in a subinterval of length at most $\del_0$. Then one cuts each such subcollection further, taking every $(c_0 +1)$-th interval in the subcollection to make one of the desired collections $\Ical$ to which we can apply (\ref{TIII_direct'}). This verifies (\ref{TIII_direct}) with the constant $(c_0+1) \del_0^{-1}$.

\subsection{Further remarks: diagonal vs. off-diagonal solutions}
What about direct inequalities like (\ref{TIII_direct}) in $L^{2r}(\phi_R)$ for a curve in $\R^n$, for $r> n$? 
The essential ingredient in the argument above is that for integers $r \leq n$ and a non-degenerate curve $\ga :[0,1] \maps \R^n$,
 the $2r$-iterated system of equations
\beq\label{TIII_sys}
 \ga(x_1) + \ga(x_3) + \cdots + \ga(x_{2r-1}) = \ga(x_2) + \ga(x_4) + \cdots + \ga(x_{2r}) 
 \eeq
has the property that its only solutions (or near solutions) are ``essentially diagonal,'' in the sense that $x_1, x_3, \ldots, x_{2r-1}$ is a permutation of $x_2, x_4, \ldots, x_{2r}$, or at least very nearly. 
 This can fail to be true for $r>n$. 
 
 We can gain an intuition for this obstacle from the case of the moment curve $\ga(t) = (t,t^2,\ldots, t^n) \subset \R^n$
 by studying integral solutions to the system of equations (\ref{TIII_sys}). This is the Vinogradov system of degree $n$ in $2r$ variables,
 \[ x_1^d + x_3^d + \cdots + x_{2r-1}^d = x_2^d + x_4^d + \cdots + x_{2r}^d, \qquad 1 \leq d \leq n.\]
Integral solutions with $1 \leq x_i \leq X$ correspond to $X^{-1}$-separated points  on $\ga$. 
 The Vinogradov system is known to have only diagonal integral solutions as long as $r \leq n$; for all large $n$ it is an open problem to determine the  least $r>n$ for which an off-diagonal integral solution  exists (the Prouhet-Tarry-Escott problem). 
 For $ 1 \leq n \leq 9$ and $n=11$, an off-diagonal integral solution has been exhibited for $r=n+1$. 
 Moreover it is  known that for $r>n$, as soon as one off-diagonal integral solution with all $|x_i| \ll X$ exists, at least $\gg X^2$ off-diagonal integral solutions exist. Thus  if a direct inequality in $L^{2r}$ (if true)  is to be obtained  for some $r>n$,   the method of proof must be able to accommodate a profusion of off-diagonal solutions.
We refer to \cite[\S 2]{GGPRY19x} for details, and a summary of literature related to counts for off-diagonal integral solutions.

\section{Quasi-superorthogonality and trace functions}\label{sec_trace}

We now introduce the notion of quasi-superorthogonality: we no longer assume that the integral of
$  f_{n_1} \bar{f}_{n_2} \cdots f_{n_{2r-1}} \bar{f}_{n_{2r}}  
$
vanishes, but   that it exhibits cancellation relative to a trivial bound. 
This will  lead us to central questions in number theory.

Let $(M,\mu)$ be a finite-measure space and denote $|M| = \mu(M) \geq 1$. Let $\{f_n\}$ be a sequence of uniformly bounded functions with finite index set $\Ical$; for simplicity we assume that $\|f_n\|_{L^\infty} \leq 1$ for all $n$. We suppose that  there is a real number $0<\nu<1$ such that  for every $r \geq 1$, for every $2r$-tuple of functions in  $\{f_n\}$,
\beq\label{TIV_vanish}
|\int_M   f_{n_1} \bar{f}_{n_2} \cdots f_{n_{2r-1}} \bar{f}_{n_{2r}} d\mu|  \leq C_r|M|^{\nu}
\eeq
as long as:  \\
{\bf Type I quasi-superorthogonality:} the tuple  $(n_1,n_2,\ldots, n_{2r})$ has the property that some value $n_j$ appears an odd number of times.\\
{\bf Type II quasi-superorthogonality:} the tuple  $(n_1,n_2,\ldots, n_{2r})$ has the property that some value $n_j$ appears precisely once (the uniqueness property).  
 
We no longer hope to prove a direct inequality, as such. Instead,  we expand the norm
\[  \|\sum_{n \in \Ical} f_n \|_{L^{2r}(M,\mu)}^{2r}
	= \sideset{}{^*}\sum_{(n_1, \ldots, n_{2r}) \in \Ical^{2r}} \int_M   f_{n_1} \bar{f}_{n_2} \cdots f_{n_{2r-1}} \bar{f}_{n_{2r}} d\mu
		 + \sideset{}{^{**}}\sum_{(n_1, \ldots, n_{2r}) \in \Ical^{2r}}   \int_M   f_{n_1} \bar{f}_{n_2} \cdots f_{n_{2r-1}} \bar{f}_{n_{2r}} d\mu ,
	\]
	in which the first sum is over those tuples with the uniqueness property and the second sum is over those tuples without the uniqueness property (so at most $r$ distinct values appear in each such tuple).
	Suppose that $\{f_n\}$ has Type I or  Type II quasi-superorthogonality with parameter $\nu$.
Then the first term may be bounded above by $C_r |\Ical|^{2r}  |M|^\nu$ by  (\ref{TIV_vanish}), while the second term
is controlled by a direct inequality, by the argument of \S \ref{sec_TI_direct}. Thus
\beq\label{approx_direct}
  \|\sum_{n \in \Ical} f_n \|_{L^{2r}}^{2r} \leq C_r |\Ical|^{2r}  |M|^\nu + C_r' \| ( \sum_{n \in \Ical} |f_n|^2)^{1/2}\|_{L^{2r}}^{2r}
  \ll_r |\Ical|^{2r} |M|^{\nu} + |\Ical|^{r} |M|.
\eeq
	This should be compared to the trivial upper bound $|\Ical|^{2r}|M|$.
In particular,   to minimize the right-hand side over the cardinality of $\Ical$, we would consider a set of indices $\Ical$ with  
	\beq\label{optimize}
	|\Ical| = |M|^{\frac{1-\nu}{r}}.
	\eeq
If $\Ical$ is of this size, then we obtain the bound 
\[  \|\sum_{n \in \Ical} f_n \|_{L^{2r}(M,\mu)} \ll_r |M|^{\frac{2-\nu}{2r}}.
\]
 This is better than  the trivial bound $|\Ical| |M|^{1/2r}$ as long as $|\Ical|$ is at least an order of magnitude larger than $|M|^{\frac{1-\nu}{2r}}$, which certainly is true under our assumption on the size of $\Ical$.

We will now describe an important setting in which quasi-superorthogonality arises:  trace functions,  which are ubiquitous in number theory.  
In this section, we demonstrate that trace functions exhibit  quasi-superorthogonality,  due to a theory built up recently in great generality  by   Fouvry, Kowalski and  Michel  \cite{FKM15} (see also \cite{FKMS19}), building on ideas of N. Katz and using the truth of the Riemann Hypothesis over finite fields, due to Deligne. In Section \ref{sec_Burgess_method} we then demonstrate that quasi-superorthogonality provides a clear way to motivate a proof of Burgess's celebrated bound for short multiplicative character sums.

\subsection{Trace functions }
While harmonic analysis nucleated around the study of periodic functions on the unit circle, analytic number theory nucleated around properties of functions of period $q$, where $q$ is a fixed integer (often prime).  
For example, such functions arise because of the ubiquity of the Fourier transform on the group $\Z/q\Z$, which introduces the functions $x \maps  e^{2\pi i a x/q}$ for   $1 \leq a \leq q$. Functions of period $q$ also arise in many other ways, for example in sieve methods, which  (roughly speaking) test whether a property holds ``globally'' in $\Z$ by testing whether it holds ``locally'' modulo $q$ for many primes $q$. 
Functions  of period $q$ are also closely intertwined with the important role played by congruences $n \con a \modd{q}$, for example in Dirichlet's theorem on the infinitude of primes in arithmetic progressions. This leads  to other well-known examples of $q$-periodic  functions such as the   Legendre symbol $x \mapsto \Leg{x}{q}$, or more generally any multiplicative Dirichlet character $x \maps \chi(x)$, a homomorphism  of the multiplicative group $(\Z/q\Z)^* \mapsto \C^*$ (if nontrivial, extended to act on $\Z/q\Z$ by setting $\chi(x)=0$ if $(q,x)>0$). One can also consider $x \mapsto e^{2\pi i x^{-1}/q}$ or $x \mapsto e^{2\pi i (x^{-1} + x)/q}$, in which $x^{-1}$ denotes the multiplicative inverse of $x$ modulo the prime $q$. 

All the functions $x \mapsto F(x)$ mentioned above are examples of trace functions. A few more examples of trace functions include normalized Gauss and Kloosterman sums such as 
\[ x \mapsto \frac{1}{q^{1/2}} \sum_{x \in \F_q} e^{2\pi i a x^2/q}, \qquad x \mapsto \frac{1}{q^{1/2}} \sum_{\bstack{x,y \in \F_q^*}{xy = a}} e^{2\pi i (x+y)/q},
\]
or hyper-Kloosterman sums. Moreover, certain procedures applied to appropriate trace functions  produce more trace functions, such as addition of trace functions, multiplication, ``convolution'' and taking the Fourier transform (due to Laumon, see \cite[Thm. 6.6]{FKMS19}). 

The fully general definition of a trace function (associated to an appropriate $\ell$-adic sheaf) would take us too far afield,
and instead we ask the reader to keep   the above examples in mind.
For the full definition of trace functions, an extensive overview of technical results, and many far-reaching applications, we refer to the excellent ``Lectures on applied $\ell$-adic cohomology'' of Fouvry, Kowalski, Michel and Sawin \cite{FKMS19}, which we will cite as we illustrate the connection to superorthogonality. (See in particular   \cite[Dfn. 3.5]{FKMS19} for a general definition.)

\subsection{Complete sums}

One of the most useful properties of trace functions is  a ``square-root cancellation'' upper bound for the sum of a trace function $F(x)$ over all elements in $\F_q$ for a prime $q$. 
We state this as: if $F$ is a trace function associated to an appropriate $\ell$-adic sheaf (precisely, weight 0 and geometrically irreducible or isotypic with no trivial component) then  
\beq\label{sqrt}
 \sum_{x \in \F_q} F(x) \ll C_F q^{1/2},
 \eeq
where $C_F$ is the conductor of $F$. 
This is a consequence of the Grothendieck-Lefschetz trace formula \cite[Thm. 4.1]{FKMS19} and Deligne's resolution of the Weil Conjectures \cite{Del80}, verifying  the Generalized Riemann Hypothesis over finite fields \cite[Thm. 4.6]{FKMS19}.  
This may be compared to the  ``trivial'' $O(q)$ upper bound  that holds  because the values of  a (weight 0)  trace function $F(x)$ associated to $\F_q$ are uniformly bounded (by the rank of the trace function).  
 (In \cite{FKMS19}, see Definition 3.1 for the rank, Definition 3.10 for the uniform upper bound for $F(x)$, Definition 4.3 for the conductor,  Corollary 4.7 for the statement of square-root cancellation as in (\ref{sqrt}), and \S 3.4 for the notions of geometrically irreducible and isotypic. Remark 3.11 indicates why the weight 0 case is sufficiently general.)

It is hard to overstate the importance of trace functions as tools in analytic number theory.
The Weil-Deligne bound  (\ref{sqrt}) is of critical importance in many proof techniques. But another type of sum is also often unavoidable in analytic methods in number theory: an incomplete sum of the form 
\[ \sum_{x \in \Ical} F(x) \]
where $\Ical$ is a proper subset of $\F_q$. 
When $\Ical$ is an interval, identified with $[a,b] \subsetneq [1,q]$, such incomplete sums can be further divided into two types: ``long sums'' in which $|\Ical| \gg q^{1/2}$ and the more difficult ``short sums'' in which $|\Ical| \ll q^{1/2}$. 
Our goal in the remainder of the paper is to show that quasi-superorthogonality of trace functions is a natural language in which to frame the current best method for bounding short sums, the Burgess method. First we must explain why quasi-superorthogonality holds.

\subsection{Quasi-superorthogonality of trace functions}
  Let $q$ be a prime. Let an element $\ga \in \mathrm{PGL}_2(\F_q)$
act on $x \in \F_q$ by fractional linear transformation, denoted by $x\mapsto \ga \cdot x$. Consider for a trace function $F$ associated to $q$ the sum
\beq\label{multi}
 \sum_{x \in \F_q} F(\ga_1 \cdot x) \overline{ F}(\ga_2 \cdot x)\cdots  F(\ga_{2r-1} \cdot x) \overline{ F}(\ga_{2r} \cdot x)   e^{2\pi i xh/q} 
 \eeq
where $\ga_1, \ldots, \ga_{2r} \in  \mathrm{PGL}_2(\F_q)$ and $h \in \F_q$. 
If   all the $\ga_i$ occur in pairs, we might not expect cancellation to occur in the sum. For example, suppose that $h=0$ and $F(x) = \chi(x)$ is a non-principal Dirichlet character modulo $q$; then if   $\ga_{2i-1} = \ga_{2i}$ for each $i=1,\ldots, r$, the sum would evaluate to $q$.
  But otherwise, we might hope that significant cancellation occurs. 
 Recently a broad set of results of this type has been codified by Fouvry, Kowalski and Michel \cite{FKM15}. 
We expose here that their results precisely fit the notion of quasi-superorthogonality for the sequence of functions $\{F(\ga_n \cdot \; )\}_n$. 
  
The key result  is as follows:  appropriate  trace functions $F$ associated to $q$
have the property that the sequence $\{F(\ga_n \cdot \; )\}_n$ with $\ga_n \in \mathrm{PGL}_2(\F_q)$ satisfies Type I quasi-superorthogonality with parameter $\nu=1/2$. Thus the approximate direct inequality (\ref{approx_direct}) holds with $M = \F_q$ and $|M|=q$. 

We can state this precisely in terms of a weighted Type I quasi-superorthogonality condition:   
 for all $h \in \F_q$,
\beq\label{superK}
|\sum_{x \in \F_q} F(\ga_1 \cdot x) \overline{ F}(\ga_2 \cdot x) \cdots  F(\ga_{2r-1} \cdot x) \overline{ F}(\ga_{2r} \cdot x)   e^{2\pi i xh/q} |	\ll_{r,C_F} q^{1/2},
\eeq
 as long as at least one fractional linear transformation $\ga_i$ appears an odd number of times in the tuple $(\ga_1,\ga_2,\ldots, \ga_{2r})$. 
 This holds for $F$ being a Dirichlet character, as well as normalized Gauss sums and Kloosterman sums, and other familiar trace functions. 
The precise details for the general setting can be found  in  \cite{FKM15} and \cite[\S 14]{FKMS19}, in which  the left-hand side of (\ref{superK}) is called a multicorrelation sum.  
More general forms of the relation (\ref{superK}) also hold, in which the trace function $F$ can itself vary from factor to factor, as explained in \cite{FKM15,FKMS19}.
The condition that the tuple $(\ga_1,\ga_2,\ldots, \ga_{2r})$ has an entry that occurs an odd number of times  was called an ``ad hoc'' definition in \cite[Dfn. 14.2]{FKMS19}; it is very pleasing that we now see it is precisely the Type I condition.

The result (\ref{superK}) is very deep; we will return momentarily to the underlying reasons it holds.
First, we highlight  an important special case, known since the 1940's, which will play a critical role in the next section on the Burgess method.

\subsection{The case of Dirichlet characters}
 Let us specify that  $F=\chi$ is a multiplicative Dirichlet character of order $\Del$ modulo a prime $q$, $h=0$, and each $\ga_i \cdot x = x+n_i$ for some $n_i \in \F_q$. Then (\ref{multi}) takes the form 
\beq\label{f_sum}
 \sum_{x \in\F_q}\chi((x+n_1)(x+n_2)^{\Del-1} \cdots (x+n_{2r-1})(x+n_{2r})^{\Del-1}) = \sum_{x \in \F_q} \chi (f_{\underline{n}}(x)),
 \eeq
say. This sum is trivially bounded above by $q$. In this special case, the square-root cancellation bound  (\ref{superK}) was known already to Weil, as a consequence of Weil's proof of the Riemann hypothesis for curves \cite{Wei41}; for a more recent source, see e.g. \cite[Thm. 11.23, Cor. 11.24]{IK}. We record the Weil bound explicitly, and then show that it verifies Type II quasi-superorthogonality for the set of functions $\{\chi( \cdot + n )\}_n$.
\begin{lemma}[Weil bound]
Let $\chi$ be a non-principal multiplicative Dirichlet character of order $\Del$ modulo a prime $q$. Then for any polynomial $f \in \Z[t]$ that has $m$ distinct roots and cannot be written as $f(t) = c h(t)^\Del$ for some $h \in \overline{\F}_q[t]$, 
\beq\label{Weil}
 |\sum_{x \in \F_q} \chi (f(x))| \leq (m-1)q^{1/2}.
 \eeq
\end{lemma}

To verify Type II quasi-superorthogonality, let $(n_1,\ldots,n_{2r})$ be a fixed tuple and define the polynomial $f_{\underline{n}}(x)$ as in (\ref{f_sum}).  If $f_{\underline{n}}$ is a constant multiple of a $\Del$-th power of a polynomial over $\overline{\F}_q$ then it also is over $\F_q$ (see for example \cite[Lemma 3.1]{PieXu19}). If the tuple $(n_1,\ldots, n_{2r})$ has the uniqueness property then some root of  $f_{\underline{n}}$ either appears only once or only $\Del-1$ times, and hence $f_{\underline{n}}$ cannot be a $\Del$-th power over $\F_q$, so by the lemma we have 
 \beq\label{Dir_sum}
| \sum_{x \in \F_q} \chi(x+n_1) \overline{\chi}(x+n_2) \cdots \chi(x+n_{2r-1}) \overline{\chi}(x+n_{2r}) | \leq (2r-1) q^{1/2}.
\eeq

This verifies that the sequence of functions $\{ \chi(\cdot + n)\}_n$, where $n$ varies over any finite set $\Ical$ of integers, satisfies Type II quasi-superorthogonality with parameter $\nu=1/2$,
 and the approximate direct inequality (\ref{approx_direct}) holds.
For later reference, we record this  in the case that $\Ical$ is the set of integers in an interval $(k_1,k_2]$:
 \beq\label{approx_direct_chi}
  \|\sum_{n \in (k_1,k_2]} \chi(\cdot + n) \|^{2r}_{\ell^{2r}(\Z/q\Z)}   \ll_r (k_2-k_1)^{2r} q^{1/2} +(k_2-k_1)^{r}q.
\eeq

\subsection{The source of quasi-superorthogonality}

The  source of the quasi-superorthogonality   exhibited in (\ref{superK}) is very interesting in its own right, and further exemplifies the universality of ``exact'' superorthogonality, in the meaning (\ref{super_gen_super_id}) considered earlier in this paper. To see this, let us consider in general the role of the Riemann Hypothesis over finite fields for studying a sum of the form 
\beq\label{Fsum} \sum_{x \in \F_q} F_1( x) \overline{ F}_2( x) \cdots  F_{2r-1}(  x) \overline{ F}_{2r}( x) e^{2\pi i hx/q},
\eeq
for appropriate trace functions $F_1, \ldots, F_{2r}$. 
We will be  informal here, and refer to the rigorous statements of  \cite[Proposition 1.1]{FKM15} for details.
(Strictly speaking, one first restricts the sum over $x \in \F_q$ to a  certain subset $U(\F_q) \subset \F_q$ upon which the $F_i$ are suitably well-behaved (``unramified,'' see Appendix B). In the cases of interest, the sum over $x \in U(\F_q)$ differs from the sum over $x \in \F_q$ by an error term that depends only on the trace functions $F_i$, and is independent of $q$  and hence negligible for all sufficiently large $q$. We suppress this consideration   here.)

The Grothendieck-Lefschetz trace formula  expresses the sum (\ref{Fsum}) as a signed sum of three terms indexed by $i=0,1,2$. For each $i$, the $i$-th term is the trace of an endomorphism (associated to a Frobenius conjugacy class) of the $i$-th cohomology group of a sheaf associated to the product $F_1(  \cdot )   \cdots   \overline{ F}_{2r}( \cdot ) e^{2\pi i h ( \cdot) /q}$ (see e.g. \cite[Thm. 4.1]{FKMS19}). Thus in order to bound the sum (\ref{Fsum}), one needs to bound each of these three terms from above. In fact, the contribution from the $0$-th cohomology group vanishes because of the product structure of $F_1(  \cdot )   \cdots   \overline{ F}_{2r}( \cdot ) e^{2\pi i h ( \cdot) /q}$. This leaves the first  and second cohomology groups. 

One way to bound a trace is   to control the dimension of the representation and the maximal size of the associated eigenvalues. 
The dimension of the first cohomology group  can be bounded above in terms of the conductors of $F_1, \ldots, F_{2r}$. 
Deligne's proof of the Riemman Hypothesis over finite fields \cite{Del80} then provides the crucial information that  all eigenvalues of the endomorphism (associated to a Frobenius conjugacy class) of the first cohomology group have absolute value $\leq q^{1/2}$. In the case of (\ref{superK}), this leads to the $O(q^{1/2})$ term   on the right-hand side.

To summarize, the Riemann Hypothesis over finite fields  shows that quasi-superorthogonality with parameter $\nu=1/2$ holds for trace functions $F_1,\ldots, F_{2r}$ as long as the contribution of the second cohomology group vanishes. This vanishing always occurs  in the case that $h\neq 0$; see \cite[\S 4]{FKM15}. The case $h=0$ is more subtle.
Interestingly, this criterion on the second cohomology group  for the case $h=0$ is  equivalent to ``exact'' superorthogonality of a different $2r$-tuple of functions,  in the original sense of an integral condition as in (\ref{super_gen_super_id}).
 Thus, the source of quasi-superorthogonality, of a certain type, for  trace functions $F_1,\ldots, F_{2r}$ is exact superorthogonality, of that same type, for an associated tuple of functions.

We thank Emmanuel Kowalski for pointing this out. This phenomenon was already understood in the key results on trace functions in works such as \cite{KowRic14, FGKM14, FKM15,FKMS19}, but has typically been presented in quite different terms. 
(See  however  \cite[Remark 4.2 and Prop. 4.4.]{KowRic14} and \cite[Prop. 3.2]{FGKM14} for instances closest to the interpretation mentioned here.)
In order to make this phenomenon explicit in the literature, we include Appendix B by Emmanuel Kowalski.

 \subsection{Further types for which quasi-superorthogonality holds}
 Fouvry, Kowalski and Michel \cite{FKM15}   have also proved the square-root cancellation bound (\ref{superK}) under further conditions on the index tuple $(\ga_1,\ldots, \ga_{2r})$ of $\mathrm{PGL}_2(\F_q)$ transformations. We briefly describe these conditions on the index tuple here, and compare them to the Type I* condition. The precise requirements on the trace function $F$ under which the following types for $(\ga_1,\ldots,\ga_{2r})$ suffice to verify  quasi-superorthogonality can be found in \cite[Thm. 1.5 and Cor. 1.6]{FKM15}; see also \cite[\S 14.1]{FKMS19}.

Given an element $\ga$, let $N(\ga)$ denote the number of times $\ga$ appears in the $r$-tuple $(\ga_1,\ga_3, \ldots, \ga_{2r-1})$ and let $N'(\ga)$ denote the number of times $\ga$ appears in the $r$-tuple $(\ga_2,\ga_4, \ldots, \ga_{2r})$. 
For appropriate trace functions $F$ with weight 0 and rank $k$,   (\ref{superK}) holds as long as at least one element $\ga$ appearing in $(\ga_1,\ldots, \ga_{2r})$ has the property
that $N(\ga) \not\con N'(\ga) \modd{k}$. 

This type can be interpreted as a weaker version  of Type I* quasi-superorthogonality.
For comparison, in the notation defined here, the condition for Type I* superorthogonality of a sequence of functions $\{f_n\}$ could be stated as follows: $\int f_{n_1} \bar{f}_{n_2} \cdots \bar{f}_{n_{2r}}=0$ as long as there is at least one value $n$ appearing in the $2r$-tuple $(n_1,n_2,\ldots,n_{2r})$ such that  that $N(n)  \neq  N'(n)$.

Here is another variation on Type I* quasi-superorthogonality:
 a class of trace functions $F$ with weight 0 and rank $k$ also has an associated involution, a transformation $\rho \in \mathrm{PGL}_2(\F_q)$ with $\rho^2 = \mathrm{Id}$.  Let $N(\ga),N'(\ga)$ be defined as above, and now let  $\tilde{N}(\ga)$ denote the number of times $\rho\ga$ appears in the $r$-tuple $(n_1,n_3, \ldots, n_{2r-1})$, and $\tilde{N}'(\ga)$ denote the number of times $\rho\ga$ appears in the $r$-tuple $(n_2,n_4, \ldots, n_{2r})$. For  appropriate   trace functions, (\ref{superK}) holds as long as at least one element $\ga$ appearing in $(\ga_1,\ldots, \ga_{2r})$ has the property
that $N(\ga) - N'(\ga) \not\con \tilde{N}(\ga) -  \tilde{N}'(\ga)  \modd{k}$. 

Examples of tuples satisfying each of these conditions can be found in \cite[Example 1.4]{FKM15}.

\subsection{A first look at incomplete sums: The P\'olya-Vinogradov method}\label{sec_PV}

We have exhibited the quasi-superorthogonality of trace functions. Now we turn to the task of bounding incomplete sums of trace functions, which play a central role in analytic number theory. We begin by recalling the classical method of P\'olya and Vinogradov, which will indicate why there is a dichotomy between long and short incomplete sums. After we see that the P\'olya-Vinogradov method suffices in the first case,  we will focus our attention on the short case for the remainder of the paper.

We start with general considerations.
Given a function $F: \Z \maps \C$ of period $q$, what can we learn about the size of
\[ \sum_{x \in \Ical} F(x) 
\]
as a function of $q$, relative to the length of the sub-interval $\Ical \subsetneq [1,q]$?
We will focus on the case of $q$ prime, later specifying properties of  $F$ as well.
We first re-write the sum via Plancherel's identity on the group $\Z/q\Z$ as
\beq\label{Parseval}
 \sum_{1 \leq x \leq q} F(x) \overline{\onebf_{\Ical}}(x) =  \sum_{1 \leq y\leq q} \widehat{F}(y) \overline{\widehat{\onebf_{\Ical}}}(y),
 \eeq
where for any function $F$ on $\Z/q\Z$ we define the Fourier transform as
\[ \widehat{F}(y) = \frac{1}{q^{1/2}} \sum_{1 \leq x \leq q} F(x) e^{2\pi i xy/q}.\]
We claim that
$ |\widehat{\onebf_{\Ical}}(y) | \ll q^{-1/2}\min\{|\Ical|, \|y/q\|^{-1}\},$ where $\|y/q\|$ denotes the distance from $y/q$ to the nearest integer.
The first bound holds if $y \con 0 \modd{q}$. If $y \not\con 0 \modd{q}$, we use the notation $\Ical = [a,b]$ and write
\[ 
\widehat{\onebf_{\Ical}}(y) =  \frac{1}{q^{1/2}} \sum_{a \leq x \leq b}   e^{2\pi i xy/q} = \frac{e^{\pi i (a+b)y/q}}{q^{1/2}}\frac{\sin (\pi (b-a+1)y/q)}{\sin (\pi y/q)}.
\]
Since $|\sin (\pi y/q)| \geq 2\| y/q\|$, this proves the claim.
As a result, $\| \widehat{\onebf_{\Ical}}\|_{\ell^1} \ll q^{-1/2} |\Ical| + q^{1/2} \log q \ll q^{1/2} \log q$,
and we conclude that 
\beq\label{F_bound}
 | \sum_{x \in \Ical} F(x) | \ll  q^{1/2} \log q  \cdot \|\widehat{F}\|_{\ell^\infty (\F_q)}.
 \eeq
This method has transferred the work to studying the size of the Fourier transform of $F$. 
In this context, it is worth recalling that if $F$ is an appropriate trace function,  the complete sum of  $F$ in (\ref{sqrt}) is identically equal to $q^{1/2} \widehat{F}(0)$, so that the square-root cancellation bound (\ref{sqrt}) implies that $|\widehat{F}(0)| \ll C_F $.   

The more general result we now need is as follows: if $F$ is a trace function of ``Fourier class,''  
\[ \| \widehat{F} \|_{\ell^\infty(\F_q)} \ll C_F^2,\]
where $C_F$ is the conductor of $F$ (see \cite[Thm. 5.2]{FKMS19}).
For trace functions of this Fourier class, we thus deduce the   upper bound
\beq\label{PV}
\sum_{x \in \Ical} F(x) \ll C_F^2 q^{1/2} \log q.
\eeq

Here we have followed  the exposition of \cite[\S 7.2]{FKMS19}, to which we refer for the definition of the Fourier class  \cite[Dfn. 7.1]{FKMS19}.
Simplistically, the Fourier class rules out the example $F(x) = e^{2\pi i a x/q}$, in which case $\widehat{F}(x) = q^{1/2} \del_{x \con -a \modd{q}}(x)$ can take ``large'' values. But it allows the example  $F(x) = \chi(x)$, where $\chi$ is a non-principal multiplicative Dirichlet character modulo $q$, in which case $|\widehat{F}(x)| \leq q^{-1/2}|\tau(\chi)|$ as long as $\gcd(x,q)=1$, where $\tau (\chi)$ denotes the Gauss sum of $\chi$; it is known that $|\tau(\chi)| = q^{1/2}$, and hence $ |\widehat{F}(x)| \leq 1$ for all $x \in \F_q$.
In the case that $F$ is a multiplicative Dirichlet character $\chi$,   the bound (\ref{PV}) was originally proved (for any integer $q$) in \cite{Pol18,Vin18} and is known as the P\'olya-Vinogradov bound.

One way to interpret the P\'olya-Vinogradov bound (\ref{PV}) is that it is nontrivial as long as $|\Ical| \gg_{C_F} q^{1/2}\log q$. How can we improve on this? 
Fouvry et al developed a method in \cite{FKMRRS17} to remove the factor of $\log q$  at the threshold $|\Ical| \approx q^{1/2}$.
Alternatively it can be advantageous to smooth a sum before estimation; we see this directly in (\ref{Parseval}), since if $\onebf_{\Ical}$ is replaced by a smoother function, its Fourier transform will have better decay properties, and potentially smaller $\ell^1$ norm. 
In the setting of trace functions, smoothing allows one to remove the $\log q$ for any interval $\Ical$  \cite[Prop. 6.5]{FKMS19}.

Nontrivial bounds when $|\Ical| = o(q^{1/2})$ seem to be out of reach of the general ideas we have mentioned so far.  
These are the so-called short sums, and are the focus of the next section.

\section{Quasi-superorthogonality and the Burgess method}\label{sec_Burgess_method}

Given a function $F: \Z \maps \C$ of period $q$, we consider the short sum
\beq\label{F_sum}
 \sum_{x \in \Ical} F(x) 
 \eeq
over an interval $\Ical \subsetneq [1,q]$ with  $|\Ical| = O(q^{1/2})$.
If $F$ is a trace function (and $q$ is prime), an ambitious and powerful goal is to prove a nontrivial $o(|\Ical|)$ bound,  perhaps even a bound as small as $O(|\Ical|^{1/2})$, as long as $|\Ical| \gg q^{\ep}$ for some $\ep>0$. This  would be a very deep result. For example, in the case where $F$ is a Dirichlet character, this is intimately connected to the Lindel\"of Hypothesis for the associated Dirichlet $L$-function (see for example Conjecture $C_n$ in  \cite[\S 9]{FIMR13} as well as their remark on equation (9.6) in that paper).

We now develop a formal chain of ideas to prove a nontrivial bound for short sums, starting from first principles.  Motivated by the P\'olya-Vinogradov method, we allow ourselves to guess that it will be advantageous to link the incomplete sum to a complete sum such as (\ref{sqrt}).
A first idea might be to  distribute multiple copies of the sum (\ref{F_sum}) via affine transformations so as to cover the complete set of residues $\Z/q\Z$, but this naive approach could lead to an error term   on the order of $ |\Ical|$, leading us back where we started. In his thesis in the 1950's, D.  Burgess had another idea: to employ many different changes of variables to redistribute copies of the sum (\ref{F_sum}) sufficiently densely over $\Z/q\Z$ that the \emph{starting points} of the transformed intervals nearly cover a complete set of residues. 
We now sketch this approach quite formally, to reveal three key components, including Type II quasi-superorthogonality.
 Once we have understood the three key requirements of the method, we will begin to work precisely.

\subsection{A formalism for short sums}\label{sec_short_formalism}

 Initially we only assume that $F$ is of period $q$ and $|F| \leq 1$.
Suppose that  $\sig$ is an invertible change of variables acting on $\Z/q\Z$,  with the image of $\Ical$ under $\sig$ denoted by $\sig(\Ical)$; for simplicity, let us also suppose temporarily that $F$ is invariant under the change of variables (and its inverse). Then 
\[ \sum_{x \in \Ical} F(x) = \sum_{y \in \sig(\Ical)} F(y). \]
If $L$ such changes of variables are denoted by $\sig_1,\ldots, \sig_L$, then averaging yields
\[ \sum_{x \in \Ical} F(x)  = \frac{1}{L} \sum_{1 \leq \ell \leq L}\sum_{y \in \sig_\ell(\Ical)} F(y) .\]
Suppose that each image $\sig_\ell(\Ical)$ is a collection of translated copies of some set, say $X$; we will return later to what $X$ could be. Define a non-negative function $a(m)$ on $[1,q]$ by setting $a(m)$ to be the number of $\ell$ such that $X+m$ appears in $\sig_\ell(\Ical)$.
Then 
\beq\label{rearrange}
 \sum_{x \in \Ical} F(x) = \frac{1}{L} \sum_{1 \leq m \leq q} a(m) \sum_{x \in X} F(m+x). 
 \eeq
In order to separate out the function $a(m)$, which counts redundancies among the images but contains no information about the function $F$, it is natural to apply
H\"older's inequality for some $1/p + 1/p'=1$, obtaining
\beq\label{TIq_pp'}
 | \sum_{x \in \Ical} F(x)| \leq 
\frac{1}{L} (\sum_{1 \leq m \leq q} a(m)^{p'})^{1/p'}  ( \sum_{1 \leq m \leq q} | \sum_{x \in X} F(m+x)|^{p})^{1/p}.
\eeq
 In this arithmetic setting, it is natural to assume that $p$ is an even integer, say $p=2r$ with $r \geq 1$, so that we can expand the $p$-th power in the last term. Note that in this case, $1/p' = 1-1/2r$ and $p' = 2r/(2r-1) \leq 2$. The nesting property of discrete $\ell^p$ spaces shows that $\|\{a(m)\}\|_{\ell^2} \leq \|\{a(m)\}\|_{\ell^{p'}}$, but we are more likely going to succeed at estimating the second moment of the sequence $\{a(m)\}$ then a fractional moment.  Since $a(\cdot)$ takes its values in non-negative integers, $\sum_m a(m)^{p'} \leq \sum_m a(m)^2$, so we can write $\| \{a(m)\}\|_{\ell^{p'}} \leq \| \{a(m)\}\|_{\ell^2}^{2/p'}$. It now suffices to understand two quantities:
   the second moment  
 \beq\label{m_moment}
  \sum_{1 \leq m \leq q} a(m)^2  ,
 \eeq
and the $\ell^{2r}$ norm
 \begin{align}
 \| \sum_{x \in X} F(m+x) \|_{\ell^{2r}(\Z/q\Z)}^{2r}
   &= \sum_{(x_1,\ldots, x_{2r}) \in X^{2r}} \sum_{1 \leq m \leq q} F(m+x_1) \overline{F}(m+x_2) \cdots F(m+x_{2r-1}) \overline{F}(m+x_{2r}) \nonumber\\
  & = \sum_{(x_1,\ldots, x_{2r}) \in X^{2r}} q^{1/2}\widehat{F_{\underline{x}}}(0), \label{2r-moment}
 \end{align}
 where
\beq\label{Fx}
 F_{\underline{x}}(m)= F(m+x_1) \overline{F}(m+x_2) \cdots F(m+x_{2r-1}) \overline{F}(m+x_{2r}).
 \eeq

 For which tuples $\underline{x}$ would we expect good control of $q^{1/2}\widehat{F_{\underline{x}}}(0)$? 
 This is characterized by the condition of  quasi-superorthogonality.  
Define $f_n(m) = F(m+n)$, so that the
left-hand side of (\ref{2r-moment}) is
\beq\label{fjsum}
 \| \sum_{ n \in X} f_n\|_{\ell^{2r}(\Z/q\Z)}^{2r}.
 \eeq
Correspondingly, the contribution to the right-hand side from a tuple $(x_1,\ldots, x_{2r}) \in X^{2r}$ is 
\[   \sum_{1 \leq m \leq q} f_{x_1}(m) \overline{f}_{x_2}(m) \cdots f_{x_{2r-1}}(m) \overline{f}_{x_{2r}}(m).\]
If the sequence of functions $\{ f_n \}_{n \in X}$ has Type I (or Type II) quasi-superorthogonality with parameter $\nu$, 
then the approximate direct inequality (\ref{approx_direct})  shows that
\beq\label{F_X}
 \sum_{1 \leq m \leq q} | \sum_{ x\in X} F(m+x)|^{2r}  \ll |X|^{2r} q^{\nu} + |X|^r q.
 \eeq
Even at this level of abstraction, we have learned something: this general approach is optimized if each change of variables $\sig_\ell $ transforms $\Ical$ into  a collection  of even smaller sets $X$ of cardinality $\approx q^{(1-\nu)/r}$. 
In particular, if $F$ is a trace function that satisfies the quasi-superorthogonality property (\ref{superK}) with $\nu=1/2$, 
we would see that this upper bound is minimized as a function of $|X|$ if  $|X| = q^{1/(2r)}$.

These speculations suggest an approach to bounding the short sum of $F$ over the interval $\Ical$:
\begin{enumerate}[label=(\roman*)]
\item    construct appropriate changes of variables that replace $\Ical$ by sets of ``short short'' intervals of length $\approx q^{1/(2r)}$ that are well-distributed over $[1,q]$; 
\item    control the redundancy of the intervals after such changes of variables via the second moment (\ref{m_moment}); 
\item   prove that the family $\{ f_n\}_n=\{ F(\cdot+n)\}_n$ exhibits Type I or Type II quasi-superorthogonality in $\ell^{2r}(\Z/q\Z)$. 
\end{enumerate}
Although Burgess's exposition in \cite{Bur57} is framed very differently, these can be seen as the three pillars of his method.

We   state the special case of Burgess's theorem that we will prove by making these three principles precise.
\begin{thm}[Burgess]\label{thm_Burgess}
Let $\chi$ be a multiplicative Dirichlet character modulo a prime $q$. Then for every integer $r \geq 1$, 
\beq\label{good_Burgess}
 \left|\sum_{m \in (N, N+H]} \chi(m) \right| \ll_r H^{1-\frac{1}{r}} q^{\frac{r+1}{4r^2}} (\log q)^2.
 \eeq
\end{thm}

Burgess further developed his method to apply when $F$ is a primitive multiplicative Dirichlet character with composite modulus $q$, but with some restrictions:  if $q$ is not cubefree, then  $ r \leq 3$; see \cite{Bur62B, Bur63A, Bur86}. Burgess's work set  record bounds for short character sums (essentially still standing), a question of Vinogradov on the least quadratic non-residue modulo $q$ \cite{Vin27} (essentially still standing), and subconvexity bounds for Dirichlet $L$-functions (only recently broken by the Weyl-strength bound of Petrow and Young \cite{PetYou19x}).

The upshot of Burgess's work is that a nontrivial bound for a sum of length $H$  holds when $F$ is a non-principal (resp. primitive) multiplicative Dirichlet character with $q$ prime (resp. $q$ cubefree), as long as $H \gg q^{1/4+\ep}$ for some small $\ep>0$. We see this by computing the optimal choice of $r$ in (\ref{good_Burgess}) for a given $H$. If $H = q^{1/4+\kappa}$ for some small $\kappa$, then the Burgess bound (ignoring the logarithmic factor) proves the upper bound $\ll Hq^{-\del}$ with $\del = (4\kappa r-1)/(4r^2)$. Computing the maximum of $\del$ as a function of $r$, it is advantageous to choose $r$ to be the nearest integer to $1/(2\kappa)$, so that as $\kappa \maps 0$ the savings is on the order of $\del \approx \kappa^2$. 

 \begin{remark} Burgess proves a bound with $(\log q)$ instead of the factor $(\log q)^2$   we demonstrate here. One logarithm comes from the density of prime numbers; the  extra logarithm in our presentation comes from the application of the Menchov-Rademacher inequality in Corollary \ref{cor_norm_MR}.
\end{remark}

We now provide a rigorous exposition of how to achieve the three points (i), (ii), and (iii) and prove Theorem \ref{thm_Burgess}.
Initially, in order to rely only on these three  principles, we  omit an averaging step from Burgess's original argument. This makes the method   more intuitive but the result is weaker than (\ref{good_Burgess}) by a factor of $q^{1/4r^2}$; see  Theorem \ref{thm_weak_Burgess}. This nevertheless provides a bound that is nontrivial for $H> q^{1/4+\kappa}$ for any $\kappa>0$ once we choose $r$ optimally.
 Finally, we show how to include the additional averaging step and recover  Theorem \ref{thm_Burgess}.
See \S \ref{sec_remarks_Burgess} for citations of recent proofs that inspire two aspects of our exposition here, and    further remarks.

 It is  an open, and very important, question whether some version of Burgess's ideas can be applied to more general (non-multiplicative) trace functions $F$.
 It would also be very significant to prove nontrivial bounds for  Dirichlet character sums that are shorter than $q^{1/4+ \ep}$.
We highlight a barrier for each of these goals below.

\subsection{Property (i): changes of variables}
 Initially, let $F: \Z \maps \C$ be a function of period $q$, for an integer $q$, and with $|F| \leq 1$. Let $\Ical \subset [1,q]$ be an interval of length at most $q^{1/2}$, which we will denote by $(N,N+H]$. By periodicity of $F$, we may assume that $0 \leq N < q$.
We will make further assumptions about $F$ as the need arises.

Fix $r \geq 1$. We seek a collection of changes of variables $x \mapsto \sig_\ell(x)$ such that for each $\ell$, the image $\sig_\ell(\Ical)$ is a set of short-short intervals, each of which is of length $\approx q^{1/2r}$.  
One idea is to break $(N,N+H]$ according to congruence classes.
Fix an integer $\ell$ and   $0 \leq b \leq \ell-1$. Then $n=b+m_b \ell$ varies over the integers in $(N,N+H]$ that are congruent to $b \modd{\ell}$ as $m_b$ varies over integers in the interval $(\frac{N-b}{\ell}, \frac{N-b}{\ell} + \frac{H}{\ell}]$. 
This motivates us to think of  $\sig_\ell$ as a collection of maps from $n$ to $m_b$, for each $0 \leq b < \ell$. 
Optimizing the approximate direct inequality (\ref{F_X}) motivates us to choose  $\ell$ so that $H/\ell \approx q^{1/2r}$.

We set $L=(1/2)Hq^{-1/2r}$ and suppose
$
 \Lscr \subseteq [L,2L] 
$
 is a set of   integers $\ell$  that are relatively prime to $q$, which we will specify precisely  later in (\ref{L_final}).
 To ensure that $L\geq 1$, we suppose from now on that $H \geq 2q^{1/2r}$. We may do so, since for $H < 2q^{1/2r}$ the bound in Theorem \ref{thm_Burgess} already holds.  
Once we fix $r$, we also assume that $q$ is sufficiently large that $H/L=2q^{1/2r} \geq 1$. 
 Finally, we can assume that $H \leq q^{1/2+1/4r}$, since otherwise the P\'olya-Vinogradov bound supersedes Theorem \ref{thm_Burgess}. 

Now for each $\ell \in \Lscr$ we write
\[ \sum_{x \in (N,N+H] } F(x) = \sum_{0 \leq b < \ell} \sum_{\bstack{x \in (N,N+H]}{x \con b \modd{\ell}}} F(x)
	 = \sum_{0 \leq b < \ell} \sum_{m \in (\frac{N-b}{\ell}, \frac{N-b}{\ell} + \frac{H}{\ell}]} F(b + \ell m).\]
We still must make $F$ invariant under this change of variables, as we required in (\ref{rearrange}). We can use the periodicity of $F$ to our advantage;  observe that as long as $(\ell,q)=1$ then there is a bijection between the sets $\{b:  b \modd{\ell}\}$ and $\{ aq: a \modd{\ell}\}$. Thus the last expression is identical to 
\beq\label{mult}
 \sum_{0 \leq a < \ell} \sum_{m \in (\frac{N-aq}{\ell}, \frac{N-aq}{\ell} + \frac{H}{\ell}]} F(aq + \ell m) = 
	\sum_{0 \leq a < \ell} \sum_{m \in (\frac{N-aq}{\ell}, \frac{N-aq}{\ell} + \frac{H}{\ell}]} F( \ell m),
	\eeq
	under the periodicity of $F$. 
	
	In order to achieve uniformity with respect to $\ell$, we must make an assumption about $F$: we assume that $F$ is \emph{totally multiplicative}, meaning that $F(\ell m) = F(\ell) F(m)$ for all integers $\ell,m$. 
	 This is a significant restriction: any function $F : \Z \maps \C$ that has the property that it is periodic of period $q$, totally multiplicative, and $F(n)$ is nonzero if and only if $(n,q)=1$ is a multiplicative Dirichlet character modulo $q$; see e.g. \cite[Thm. 6.15]{Apo76}. 
	 Thus from now on we assume that $F$ is a multiplicative Dirichlet character.

Even after writing $F(\ell m ) = F(\ell) F(m)$ in (\ref{mult}),  the resulting expression is not completely invariant; by using the property that $|F|  \leq 1$, we can achieve invariance if we take absolute values, writing
\[ |\sum_{x \in (N,N+H] } F(x) |\leq    \sum_{0 \leq a < \ell} |\sum_{m \in (\frac{N-aq}{\ell}, \frac{N-aq}{\ell} + \frac{H}{\ell}]} F( m)|.\]
This is a transformation of the original sum, according to one such choice of $\ell$.
We average the above inequality over all $\ell \in \Lscr$, leading to 
\beq\label{stopping_time}
| \sum_{x \in (N,N+H] } F(x) | \leq  |\Lscr|^{-1} \sum_{\ell \in \Lscr} \sum_{0 \leq a < \ell} | \sum_{m \in (\frac{N-aq}{\ell}, \frac{N-aq}{\ell} + \frac{H}{\ell}]} F( m)|.\eeq
We define $a(m)$ to count the redundancies of the starting points,
 \beq\label{a_dfn}
 a(m) = \# \{ \ell \in \Lscr, 0 \leq a < \ell: \lfloor (N-aq)/\ell \rfloor  =m\}.
 \eeq
In particular,  
we can rewrite (\ref{stopping_time}) as 
 \[| \sum_{x \in (N,N+H] } F(x) | \ll |\Lscr|^{-1}   \sum_m a(m)   \max_{k \leq 2H/L} | \sum_{x \in (m, m+k]}F( x) |. \] 
In our simple paradigm, the next step is to apply H\"older's inequality. First, it is worth noting that the sum over $m$ is in fact finite, since by construction $a(m)$ is supported inside the set $[-q,q]$. 
Thus for   $p=2r$ with $1/p + 1/p'=1$, again recalling $p' \leq 2$ and that the sequence $\{a(m)\}$ takes its values in non-negative integers, we can write
 \beq\label{step2step3}
 | \sum_{x \in (N,N+H] } F(x) | \ll  |\Lscr|^{-1}   ( \sum_m a(m)^2)^{1-1/2r} \| \max_{k \leq 2H/L} | \sum_{x \in (0,k]}F( \cdot + x) |\|_{\ell^{2r}(\Z/q\Z)}. \eeq
This is an echo of (\ref{echo}) in Paley's proof of the direct inequality for the partial sums of the Walsh-Paley series.

 \subsection{Property (ii): redundancy of short-short intervals }

By the definition of $a(m)$, the average is bounded by
 \[ \sum_m a(m) \ll \sum_{\ell \in \Lscr, 0 \leq a < \ell } 1  \ll L^2.\]
We will show that the short-short intervals are well-distributed in the sense that the second moment has the same upper bound:
\beq\label{2moment_0}
 \sum_{m} a(m)^2 \ll L^2,
 \eeq
 as long as we restrict the values in $\Lscr$ to be prime. Thus we now formally define
 \beq\label{L_final}
  \Lscr  := \{ \mathrm{primes} \;  \ell \ndiv q, \ell \in [L,2L]= [Hq^{-1/2r}/2,Hq^{-1/2r}]  \} .
  \eeq
 By the Prime Number Theorem, $|\Lscr| \gg L/\log L$ as soon as $q$ is larger than an absolute constant depending only on $r$, as we may assume by possibly enlarging the implicit constant in the bound of Theorem \ref{thm_Burgess}.
  
The main input to proving (\ref{2moment_0}) is a lemma that counts the number of starting points that lie within a short distance of each other.
\begin{lemma}\label{lemma_M}
Fix $B \geq 1$ and $L \geq 1$ with $BL^2 < q$. For any integers $\ell, \ell' $ let
\[ \Mcal(\ell,\ell') = \# \{ 0 \leq a < \ell, 0 \leq a' < \ell': |(N-aq)/\ell - (N-a'q)/\ell'| \leq B\}.\]
As $\ell,\ell'$ range over a set $\Lscr$ of prime values in $[L,2L]$,
\[ \sum_{\ell, \ell' } \Mcal(\ell,\ell') \ll L^2.\]
\end{lemma}

We apply this with $B=1$ and $\Lscr$ as defined above.
Then $ \sum_m a(m)^2 \ll  \sum_{\ell ,\ell' \in \Lscr} \Mcal (\ell,\ell')$
and the second moment bound (\ref{2moment_0}) follows from the case $B=1$. Note that the assumption that $L^2< q$ is met since we assume $H \leq q^{1/2+1/4r}$.

\begin{proof}[Proof of Lemma \ref{lemma_M}]
We prove the lemma by adapting a method of Heath-Brown \cite[\S 4]{HB12}.
If $\ell = \ell'$, 
\[ \Mcal(\ell,\ell) =  \# \{ 0 \leq a,a' < \ell : |a - a'| \leq \frac{B \ell}{q} \leq \frac{2BL}{q} \leq 2\} \ll L, 
\]
which suffices.

The case  $\ell \neq \ell'$ is more intricate. We assume that $N \geq 1$; the case $N=0$ may be handled by a simpler adaptation. The first step is to replace $N$ by a multiple of $q$ (with an acceptable error relative to the scale $B$), so that a factor of $q$ can be pulled out of all terms in the inequality defining $\Mcal(\ell,\ell')$. 
We motivate this as follows.
 If we suppose that for some $t$ we have $N/\ell - tq/\ell = O(B)$ this would require that $N-tq = O(BL)$, but by hypothesis $BL\leq BL^2< q$ and we cannot necessarily   replace $N$ by an integral multiple of $q$ with an error smaller than $ q/2$.  So we must allow for $t$ itself to be a rational number, which we denote by $t_1/t_2$. Thus we suppose for some integers $t_1,t_2$ that  
 \beq\label{Nbb}
 |N - \frac{t_1q}{t_2} | \leq 2BL.
 \eeq
 This will hold for some $q/(BL) < t_2 \leq 2q/(BL)$ and $ Nt_2/q < t_1 \leq 2Nt_2/q$ that we will specify momentarily.

Assume such $t_1,t_2$ exist for the moment. Fix $\ell \neq \ell'$. For any  $a,a'$ counted by $\Mcal(\ell,\ell')$ we then see that
\[ | (t_1q/t_2 - aq)/\ell - (t_1q/t_2 - a'q)/\ell' | \leq B + 2BL/\ell + 2BL/\ell' \leq 5B. \]
Thus 
\[  |t_1(\ell' - \ell)  - (a\ell' - a'\ell)t_2| \leq 5 \ell \ell' t_2B/q  \leq 20 L t_2 (BL)/q \leq 40 L.  \]
For each $d$, let $D(\ell,\ell';d)$ denote $\# \{ 0 \leq a < \ell, 0 \leq a' < \ell' : a\ell' - a' \ell = d\}$. Then let $D = \max_{d,\ell \neq \ell'} D(\ell,\ell';d)$. 
We have shown that
 \beq\label{Mll_0}
 \sum_{\ell \neq \ell' \in \Lscr} \Mcal (\ell,\ell') \ll D \sum_{|m| \leq 40 L} \# \{ \ell \neq \ell' \in \Lscr : t_1(\ell'-\ell) \con m \modd{t_2}\}.
 \eeq
 We claim that if $\Lscr$ contains only prime values then $D \leq 1$. 
 We further claim that we can choose $t_1,t_2$ satisfying the constraints above, with $t_2$  prime and $(t_1,t_2)=1$. 
 Assume these two claims, which we prove momentarily. Then given $m$, the congruence $t_1 (\ell' - \ell) \con m \modd{t_2}$ identifies $(\ell' - \ell)$ uniquely modulo $t_2$.
In $\Z$, the difference $(\ell' - \ell)$ is at most $L$, and under the hypothesis $BL^2<q$ we see that $L <t_2$ so that   $(\ell' - \ell)$ is uniquely identified in $\Z$ as well. Thus once $\ell$ is chosen freely, $\ell'$ is uniquely chosen, and the sum over $m$ on the right-hand side of (\ref{Mll_0}) is $\ll L^2$, which suffices as long as $D \leq 1$. 

We prove the two remaining claims. 
We choose $t_2$ to be a prime in the interval $(q/(BL) , 2q/(BL)]$, which exists by Bertrand's postulate (or alternatively by the Prime Number Theorem if we may assume that $q/(BL)$ is larger than an absolute constant, which we may in our application).   
Given $t_2$, we choose $t_1$ to be either $\lceil Nt_2/q \rceil $ or $\lceil Nt_2/q \rceil+ 1$, so that it is relatively prime to $t_2$. Note that (\ref{Nbb}) then holds with these choices.  
 
 Finally, we bound $D$.   It suffices to observe that under the assumption that $\ell \neq \ell'$ are primes, for each $d$, there is at most one pair $a,a'$ with $0 \leq a < \ell$ and $0 \leq a'< \ell'$ solving $a\ell' - a'\ell=d$. Otherwise, suppose $a\ell' - a'\ell=d = b\ell' - b'\ell$ for $0 \leq a, b < \ell$ and $0 \leq a',b' < \ell'$. Then because we have assumed that $\ell \neq \ell'$ are primes, this shows that $\ell | (a-b)$ and $\ell' | (a'-b')$, which can only occur for $a,a',b,b'$ in the allowed ranges if both differences are zero in $\Z$.  
 \end{proof}
This completes the verification of (\ref{2moment_0}) for property (ii).

\subsection{Property (iii): Type II quasi-superorthogonality and the maximal operator}
We now need to bound the maximal partial sum norm in (\ref{step2step3}), using property (iii).
Recall from (\ref{approx_direct_chi}) that  as an application of Type II quasi-superorthogonality for Dirichlet characters, for any integers $k_1< k_2$,
\beq\label{nonmax}
\| \sum_{x \in (k_1,k_2]}F(\cdot + x) \|_{\ell^{2r}(\Z/q\Z)} \ll_r (k_2-k_1)q^{1/4r} + (k_2-k_1)^{1/2} q^{1/2r} .
\eeq
We now need to deduce an upper bound for the norm of the maximal partial sum operator, $\| \max_{k \leq 2H/L}| \sum_{x \in (0,k]} F(\cdot + x) |\|_{\ell^{2r}(\Z/q\Z)}$. We will do so 
via   the ``method of bisection,'' which originated in the same paper of Rademacher we encountered earlier \cite[p. 118 and p. 129]{Rad22}. (This method also appeared independently in Menchov \cite{Men23}. See e.g. \cite{Bed06} for references to modern proofs; it can also be used to prove the related Kolmorogov-Doob  inequality for martingales \cite[Ch. III Thm. 2.1, Ch. VII Thm. 3.2]{Doo53}.)
We encapsulate the method in two general statements.

 \begin{lemma}[Menchov-Rademacher]\label{lemma_MR}
 Given a sequence $\{b(n)\}$ of complex numbers, for any integer $t \geq 0$ and any $p \geq 1$,
 \[ \max_{0 \leq n \leq 2^t} |b(n) - b(0)|^{p} \leq   (t+1)^{p-1} \sum_{i=0}^t \sum_{0 \leq v < 2^i} | b( (v+1)  2^{t-i}) - b( v 2^{t-i}  )|^{p} .\]
 \end{lemma}
 
 The key point of this lemma is that the length of the sum over $i$ on the right-hand side is logarithmic in scale, compared to the range $0 \leq n \leq 2^t$ of the maximum on the left-hand side.
We deduce from this a useful fact: relative to the norm of a partial sum operator, the norm of an associated maximal partial sum operator (over a finite range)   increases by at most a logarithm.
  
 \begin{cor}\label{cor_norm_MR}
Let $(\Mcal,\mu)$ be a measure space. Let $\{a_k (u)\}_k$ be a sequence of complex-valued functions in $L^p(\Mcal,d\mu)$. Define for each integer $k \geq 0$, 
 \[ S_k(u) = \sum_{0 <  m \leq k} a_m(u), \qquad S_{k_1,k_2}(u) = \sum_{k_1 <  m  \leq  k_2}a_m(u).\] 
  Fix $ 1 \leq p < \infty$. Suppose that uniformly in $k_2> k_1$,
  \[\|S_{k_1,k_2}(\cdot)\|_{L^p(\Mcal)} \leq c_p |k_2-k_1|^{\al_p}.\]
 Then as long as $p \al_p \geq 1$, for every $K \geq 2$,
  \[  \| \max_{0 \leq k \leq K} |S_k (\cdot) |\|_{L^p(\Mcal)} \ll_p c_p K^{\al_p} (\log K) .\]
 \end{cor}

We defer the proof of the lemma and its corollary to the end of the section.
 We apply the corollary to the partial sums of $F(\cdot + x)$, using the uniform upper bound (\ref{nonmax}), and   $\Mcal=\Z/q\Z$ with counting measure.
  This proves that for any $K \geq 2$,
\beq\label{max}
\| \max_{k \leq K} | \sum_{x \in (0,k]}F(\cdot + x) | \, \|_{\ell^{2r}(\Z/q\Z)} \ll_r (K q^{1/4r} + K^{1/2} q^{1/2r}) (\log K) .
\eeq
For our application in (\ref{step2step3}) we take $K=2H/L = 4q^{1/2r}$, so that the right-hand side is $\ll_r q^{3/4r}(\log q)   $.

\subsection{Deduction of  a weak Burgess bound}
We input the consequences   (\ref{2moment_0}) and   (\ref{max}) of properties (ii) and (iii) into our key relation (\ref{step2step3}). 
Upon recalling the definition of $\Lscr$ and $L$, this yields the following result, which  is larger than  the classical Burgess bound  by a factor of $q^{1/4r^2}$.  
\begin{thm}[Weak   bound]\label{thm_weak_Burgess}
Let $\chi$ be a non-principal multiplicative Dirichlet character modulo a prime $q$. Then for every integer $r \geq 1$,
\beq\label{bad_Burgess}
 | \sum_{x \in (N,N+H] } \chi(x) | 
 \ll_r H^{1-1/r} q^{(r+2)/4r^2}(\log q)^2 .
 \eeq
 \end{thm}

Supposing that $H = q^\be$, we see that for a fixed $r$, the exponent is $< \be$ (so that the bound is $o(H)$) when $\be > 1/4 + 1/2r$. Thus in particular, the bound only has a chance of being nontrivial if $\be> 1/4$, showing that we recovered the same core threshold as the classical Burgess bound. Now let us fix $\be = 1/4 + \kappa$ for some small $\kappa>0$
and compute the optimal choice of $r$. Up to a factor of $(\log q)^2$, the upper bound  in (\ref{bad_Burgess}) is $\ll Hq^{-\del}$ with $\del = (2\kappa r-1)/(2r^2)$. Computing the maximum of $\del$ as a function of $r$, it is advantageous to choose $r$ to be the nearest integer to $1/\kappa$, and as $\kappa \maps 0$ the savings is on the order of $\del \approx \kappa^2/2$.   
 See \S \ref{sec_compare_Burgess} for further comparison to the Burgess bound.

   We proved this weak bound in order to demonstrate the three core principles. 
 To recover the classical Burgess bound, we now introduce further averaging to (\ref{stopping_time}) and prove a second moment bound analogous to (\ref{2moment_0}) but for a  different function $a^\sharp(m)$.  

\subsection{Property (ii) revisited: additional averaging to prove the strong Burgess bound}

In our schematic argument, when passing from (\ref{rearrange}) to (\ref{TIq_pp'}) via H\"older's inequality, we lose less if $a(m)$ is nonzero for as many $m \in [1,q]$ as possible (but also not too large at any $m$). 
Within our precise argument, this motivates us to   further average (\ref{stopping_time}) over an even larger family of short-short intervals that are relatively well-distributed across $[1,q]$. A version of the further averaging we now describe  appeared in Burgess's original work.  

Any interval $(A, A+B]$ can be written as a difference $(A-A_0,A+B]\setminus (A-A_0,A],$ for every $A_0 \geq 0$. Moreover, as long as $A_0 \leq B$, then the longer interval $(A-A_0,A+B]$ is still of length at most $2B$, and hence comparable to the length of the original interval. There are $B$ ways to write $(A,A+B]$ as such a difference. We
apply this to $( \frac{N-aq}{\ell} , \frac{N-aq}{\ell}+ \frac{H}{\ell}]$ in order to average (\ref{stopping_time}) over $H/\ell \approx q^{1/2r}$ more short-short intervals. 

Precisely, we observe that the right-hand side of (\ref{stopping_time}) is equal to
\begin{multline*}
 |\Lscr|^{-1} \sum_{\ell \in \Lscr} \sum_{0 \leq a < \ell}  (H/\ell)^{-1}  \sum_{m \in (\frac{N-aq}{\ell} - \frac{H}{\ell}, \frac{N-aq}{\ell} ]} | \sum_{x \in (m, \frac{N-aq}{\ell} + \frac{H}{\ell}]}F( x) -  \sum_{x \in (m, \frac{N-aq}{\ell}]}F( x)| 
	 \\
	   \ll	  |\Lscr|^{-1} q^{-1/2r} \sum_{\ell \in \Lscr} \sum_{0 \leq a < \ell}     \sum_{m \in (\frac{N-aq}{\ell} - \frac{H}{L}, \frac{N-aq}{\ell} ]}
	  2 \max_{k \leq 2 H/L} | \sum_{x \in (m, m+k]}F( x) |.
	\end{multline*}
Now we define $a^\sharp(m)$ to count the redundancies of the starting points, 
 \beq\label{a_dfn'}
 a^\sharp(m) = \# \{ \ell \in \Lscr, 0 \leq a < \ell: m \in (\frac{N-aq}{\ell} - \frac{H}{L}, \frac{N-aq}{\ell} ] \}.
 \eeq
We conclude that
 \[| \sum_{x \in (N,N+H] } F(x) | \ll |\Lscr|^{-1} q^{-1/2r} \sum_m a^\sharp(m) \max_{k \leq 2H/L} | \sum_{x \in (m, m+k]}F( x) |. \]
We prepare to apply H\"older's inequality so we can exploit Type II quasi-superorthogonality in $\ell^{2r}(\Z/q\Z)$. First note that the sum is  finite since  $a^\sharp(m)$ is supported inside $[-2q,2q]$. 
Thus for   $p=2r$ with $1/p + 1/p'=1$, upon recalling $p' \leq 2$, since $a^\sharp(\cdot)$ takes its values in non-negative integers we can write
 \beq\label{step2step3'}
 | \sum_{x \in (N,N+H] } F(x) | \ll  |\Lscr|^{-1} q^{-1/2r} (\sum_m a^\sharp(m)^2)^{1-1/2r}\| \max_{k \leq 2H/L} | \sum_{x \in (0, k]}F( \cdot +m)| \; \|_{\ell^{2r}(\Z/q\Z)}. \eeq
Compared to (\ref{step2step3}), we have an extra savings $q^{-1/2r}$, but possibly larger   values $a^\sharp(m)$ compared to $a(m)$. 

We will again bound the second moment comparable to its average:
\[
 \sum_{m} a^\sharp(m)^2 \ll HL.
\]
 This is larger than  (\ref{2moment_0}) by a factor of $q^{1/2r}$, but since this is raised to the power $(1-1/2r)$ in (\ref{step2step3'}), we still gain a total of $q^{-1/4r^2}$ in extra savings in (\ref{step2step3'}).
Note that 
\[ \sum_m a^\sharp(m)^2 \ll (H/L) \sum_{\ell, \ell' \in \Lscr} M^\sharp(\ell,\ell')  , 
\]
where we now define
\[ M^\sharp(\ell,\ell') = \# \{ 0 \leq a < \ell, 0 \leq a' < \ell': |(N-aq)/\ell - (N-a'q)/\ell'| \leq H/L\}.\]
We can then apply  Lemma \ref{lemma_M} with $B=H/L$; note that $BL^2 = HL<q$ is satisfied since $H \leq q^{1/2+1/4r}$, and this verifies the second moment bound.
 With this and (\ref{max}) in hand, (\ref{step2step3'}) immediately proves
\[ | \sum_{x \in (N,N+H] } F(x) | \ll_r |\Lscr|^{-1} q^{-1/2r} (LH)^{1 - 1/2r} q^{3/4r} \log q \ll_r H^{1-1/r} q^{(r+1)/4r^2}(\log q)^2 ,\]
proving Theorem \ref{thm_Burgess}.
 
   \subsection{The Menchov-Rademacher inequality: proof of the lemmas}\label{sec_MR}
   \begin{proof}[Proof of Lemma \ref{lemma_MR}]
 If $n=2^t$ then $|b(n)  - b(0)|^{p}$ appears on the right-hand side as the summand with $i=v=0$, and thus we may fix our attention on the maximum over $1 \leq n< 2^t$. 
 Fix $1 \leq n< 2^t$ and write its binary expansion  as
 \[
 n = \sum_{i=0}^t \ep_i 2^{t-i}, \qquad  \ep_i  = \ep_i(n) \in \{0,1\}, \quad \ep_0 = \ep_0(n)=0.
 \]
 We write a telescoping sum for the difference of interest:
\[
  b(n) - b(0)  =\sum_{i=1}^t \{ b(\sum_{j \leq i} \ep_j 2^{t-j}) - b( \sum_{j<i} \ep_j 2^{t-j}) \}  =\sum_{i=1}^t \{ b(2^{t-i}\sum_{j \leq i} \ep_j 2^{i-j}) - b(2^{t-i} \sum_{j<i} \ep_j 2^{i-j}) \} .\]
   For each $1 \leq i \leq t$ it is then convenient to define
 \[ v_i =v_i(n) = \sum_{0 \leq j < i}\ep_j 2^{i-j}. \]
We also define $v_0=0$. Observe that $0 \leq v_i < 2^i$ for each $1 \leq i \leq t$, and (recalling $\ep_0=0$) we can write
\[
  b(n) - b(0)    =\sum_{i=0}^t \{ b( v_i 2^{t-i} + \ep_i2^{t-i}) - b( v_i 2^{t-i}  ) \}.\]
 Fix $1 \leq p < \infty$. Taking absolute values and applying H\"older's inequality,
  \[ |b(n) - b(0)|^{p} \leq (t+1)^{p-1} \sum_{i=0}^t | b( v_i 2^{t-i} + \ep_i2^{t-i}) - b( v_i 2^{t-i}  ) |^{p}.\]
We only possibly increase the right-hand side  if we  sum  over all possible values of $v_i < 2^i$; additionally, all nonzero terms on the right-hand side have  $\ep_i =1$, and we only possibly increase the right-hand side if we assume this always is the case. 
Thus
 \[ |b(n) - b(0)|^{p} \leq  (t+1)^{p-1} \sum_{i=0}^t \sum_{0 \leq v <2^i} | b( (v+1)  2^{t-i}) - b( v 2^{t-i}  )|^{p} .\]
Now we note that the right-hand side is independent of $1 \leq n < 2^t$, and the lemma is proved.
\end{proof}

 \begin{proof}[Proof of Corollary \ref{cor_norm_MR}]
Fix $1 \leq p < \infty$. Given $K \geq 2$, let $t \geq 1$ be such that $2^{t-1} \leq K < 2^t$. Then the left-hand side of the claimed inequality is dominated by 
$ \| \max_{0 \leq k \leq 2^t} |S_k (\cdot) |\|_{L^p(\Mcal)}$ while the putative right-hand side is comparable to $(2^{t})^{\al_p} (\log (2^t)).$ Thus it suffices to prove the inequality for the case $K=2^t$.

 We apply the lemma for each fixed $u$, with the choice $b(k) = S_k(u)$, followed by the uniform upper bound for the $L^p$ norm in the hypothesis. (Note that by construction $S_0(u)\con 0$ since it is an empty sum, so $|b(k)-b(0)| = |b(k)|$.) Thus we reason that:
 \begin{align*}
  \| \max_{0 \leq k \leq 2^t} |S_k (\cdot) |\|_{L^p(\Mcal)}^p
 	& = \int_{\Mcal} \max_{0 \leq k \leq 2^t} |S_k (u) |^p d\mu(u)\\
	& \leq \int_{\Mcal} (t+1)^{p-1} \sum_{i=0}^t \sum_{0 \leq v < 2^i} |S_{v2^{t-i},(v+1)2^{t-i}}(u)|^p d\mu(u)\\
	&= (t+1)^{p-1} \sum_{i=0}^t \sum_{0 \leq v < 2^i} \|S_{v2^{t-i},(v+1)2^{t-i}}(\cdot)\|_{L^p(\Mcal)}^p\\
	&\leq {c_p}^p (t+1)^{p-1} \sum_{i=0}^t \sum_{0 \leq v < 2^i} (2^{t-i})^{p\al_p}\\
	& \leq {c_p}^p  (t+1)^{p}    2^{tp\al_p}.
	 \end{align*}
 Here we used $\al_p p \geq 1$ so that the  factor $2^{-ip\al_p}$ at least dominates the $O(2^i)$ contribution of summing trivially over $v$.
Thus we have shown $  \| \max_{0 \leq k \leq 2^t} |S_k (\cdot) |\|_{L^p(\Mcal)} \leq c_p (t+1) 2^{t \al_p}.$ 
Since $(t+1) \ll 2\log (2^t)$  as long as $t \geq 1$,  this suffices for the case $K=2^t$ under consideration.

 \end{proof}

\subsection{Further remarks on the Burgess bound}\label{sec_remarks_Burgess}
  \subsubsection{Comparison of the weak bound to Burgess bound}\label{sec_compare_Burgess}

For a given $r$, the weak bound (\ref{bad_Burgess}) is nontrivial if $H> q^{1/4 + 1/2r}$ while the Burgess bound (\ref{good_Burgess}) is nontrivial if $H> q^{1/4 + 1/4r}$. Thus in the limit of arbitrarily large $r$, each has a threshold around $H>q^{1/4 + \ep}$ for $\ep>0$ arbitrarily small. But for any fixed $r$, and in particular for small $r$, the difference between (\ref{bad_Burgess}) and (\ref{good_Burgess}) is significant. Up to logarithmic factors, the weak bound is worse than P\'olya-Vinogradov if $r=1$, meets it if $r=2$, and improves on it for $r \geq 3$; the Burgess bound meets P\'olya-Vinogradov for $r=1$ and improves on it for $r \geq 2$.  This behavior for small $r$ also matters for composite $q$; the Burgess bound (\ref{good_Burgess}) is only known (via a more intricate proof) for $r \leq 3$ unless $q$ is cubefree, and one would expect similar restrictions for the weak bound. 

We specify the impact on subconvexity bounds for the Dirichlet $L$-function $L(1/2 + it, \chi)$ with $\chi$ of modulus $q$. Assume an upper bound of the form
\[S(x) = \sum_{1 \leq n \leq x} \chi (n) \ll q^\ep \min \{ q^{1/2}, x^\al q^\be\}\]
 for some $\al \leq 1, \be \leq 1/2$. An application of the approximate functional equation  \cite[Ch. 12]{IK} shows that 
   if $\al \leq 1/2$ then 
 \[ |L(1/2 +it, \chi)| \ll q^{\ep} \max\{ q^{(\al + \be -1/2)/(2\al)}, q^\be\},\]
and   if $\al > 1/2$ then
  \[  |L(1/2 +it, \chi)| \ll q^\ep \max\{ q^{(\al + \be -1/2)/(2\al)}, q^{\be + (1/2 + \be)(\al-1/2)/\al}\}.\]
The Burgess bound (\ref{good_Burgess})   provides $\al=1-1/r, \be = (r+1)/4r^2$ and the optimal choice occurs at $r=2$, thus proving Burgess's famous subconvexity bound $q^{1/4-1/16+ \ep}$. But the weaker bound (\ref{bad_Burgess}) provides $\al=1-1/r, \be = (r+2)/4r^2$ and the optimal choice occurs at $r=3$, leading to the much weaker bound $q^{1/4 - 1/48 + \ep}$.

\subsubsection{Influences and expositions}
 One can speculate how Burgess arrived at his  clever method. 
 Burgess's first paper cites  Davenport and Erd\H{o}s \cite{DavErd52} as a point of inspiration \cite[Lemma 2, p. 108]{Bur57}. Davenport and Erd\H{o}s  addressed Vinogradov's question on the least quadratic nonresidue modulo a prime $q$.
 In \cite[Lemma 1]{DavErd52} they consider (\ref{2r-moment}) in the case $r=1$, proving the identity 
 \[ \sum_{m \modd{q}} | \sum_{x \in (0,k]} \chi (m+x)|^2 = qk -k^2.\]
This does not require the Weil or Deligne bounds;  Davenport and Erd\H{o}s cite a 1906 thesis of Jacobsthal, and conjecture in a footnote it could have been known to Gauss. 
In \cite[Lemma 3]{DavErd52}   they consider the $2r$-th moment   for any $r \geq 1$ and prove what we call here an approximate direct inequality, referencing Weil's very recent work at that time (Burgess cites \cite[\S IV]{Wei45}). But they state that ``it does not seem to throw any light on the problem of the magnitude of the least quadratic non-residue;'' Burgess changed this.

 In this exposition, we introduce the new perspective that Burgess's argument is an application of superorthogonality, which incidentally we have seen was ``in the air'' in the 1920's and 1930's.
Additionally, we incorporated elements of two    treatments that streamline  Burgess's original method. Unpublished notes of H. Montgomery from the 1970's,   later   developed into  \cite{GalMon10}, introduced the use of the Menchov-Rademacher argument; this allows a more direct approach than Burgess described, and unifies the treatment when $N=0$ and $N \neq 0$, at the cost of a factor of $(\log q)^2$ instead of $(\log q)$ in the final Burgess bound. (In Burgess's work, certain disjointness properties of the short-short intervals were easier to prove when $N=0$.)  We also applied ideas of Heath-Brown   \cite{HB12}, which completely removed the need to show the short-short intervals are disjoint, by instead bounding the second moment (\ref{2moment_0}). There are other modern approaches of alternative flavors, such as  \cite[Thm. 12.6]{IK} in terms of multiplicative shifts, and a smoothed version in \cite[\S 17]{FKMS19}.

 Recent work has succeeded in applying Burgess-type arguments in other settings that involve multiplicative Dirichlet characters: see among other works \cite{DavLew63,Cha08,Cha09,BouCha10,HB12,HB16}. See also \cite[\S 17.2, \S 17.3]{FKMS19} for an exposition applying some of these ideas to so-called Type II and Type III sums, after introducing further averaging.
  Burgess arguments have also now been developed for ``mixed'' character sums, in which $F(x) = \chi(x) e^{2\pi i g(x)}$ where $\chi$ is a multiplicative Dirichlet character and $g$ is any real-valued polynomial; interestingly, these use the resolution of the Vinogradov Mean Value Theorem; see \cite{HBP15, Pie16, PieXu19, Pie20x}, and also the earlier \cite{Cha10}.
But the step (\ref{mult}), in  which we assumed that $F$ is totally multiplicative, prevents this argument from working more generally for  trace functions.  It would be of great interest to expand these ideas to apply to non-multiplicative trace functions.
 
  \subsection{Further types: short sums of random multiplicative functions}

  In this section we studied short sums of multiplicative trace functions. Short sums of other multiplicative functions are also of great interest; for example, the Riemann Hypothesis is equivalent to the claim that 
$ \sum_{n \leq x} \mu(n) = O(x^{1/2 + \ep})$ for all $x \geq 1$,  and all $\ep>0$, where $\mu(\cdot)$ is the M\"obius function. 

Wintner \cite{Win44} initiated a more general study of short sums of ``random multiplicative functions.''  One model is given by Rademacher  random multiplicative functions. These are built from the Rademacher distributions we have already seen, as follows. As $p$ varies over primes, $f_p$ is a sequence of independent random variables taking values $\pm 1$ with probability $1/2$. For square-free $n$, the random variable $f_n$ is defined by $f_n = \prod_{p|n} f_p$. 
Another model is a Steinhaus random multiplicative function:  as $p$ varies over primes, $f_p$ is a sequence of independent random variables uniformly distributed on the unit circle, with $f_n = \prod_{p^a || n} f_p^a$. 

 Let $\{f_n\}_n$ denote a sequence of such independent random multiplicative functions. Recent work    has computed (among other striking results) asymptotics for   
$ \| \sum_{n \leq N} f_n \|_{L^{k}} $, see \cite{HNR15,HeaLin15}. In the Steinhaus case, for $k=2r$ an even integer, the first step of the proof is an observation of superorthogonality, namely that a term $\int f_{n_1} \overline{f}_{n_2} \cdots f_{2r-1} \overline{f}_{2r}$ vanishes unless $n_1n_3 \cdots n_{2r-1} = n_2n_4 \cdots n_{2r}$.  We can think of this as a ``multiplicative diagonal'' constraint.
 In the Rademacher case, for any integer $k$ the first step  of the proof reveals yet another type of superorthogonality, namely  that a term 
 $\int f_{n_1} \cdots f_{n_k}$ vanishes unless $n_1 \cdots n_k$ is a perfect square  and each $n_1, \ldots, n_k$ is square-free.
 Each of these can be compared to Type I* superorthogonality.

\setcounter{section}{1}
 \renewcommand{\thesection}{\Alph{section}}
 
 \setcounter{subsection}{0}
 
  \section*{Appendix A: Further remarks on Walsh-Paley series }\label{sec_app_Paley}
  
    \setcounter{equation}{0}
\renewcommand{\theequation}{A.\arabic{equation}}

We deferred a few details on the direct and converse inequalities in the setting of Walsh-Paley series  in \S \ref{sec_Paley}.
Here, we first remark on the limiting argument to obtain (\ref{TII_dir_con_together}) for $p=2r$ from the truncated version (\ref{TII_direct}).
  Second, we remark on deducing the cases for $1<p<\infty$ from the cases with $p$ an even integer; this illustrates  a further application of   Khintchine's inequality.
Third, we  show how to deduce the operator bound (\ref{non_dyadic}) from the dyadic direct and converse inequalities (\ref{TII_dir_con_together}).

\subsection{Limiting arguments for direct and converse inequalities}
Fix $p=2r$. In the main text we showed that uniformly in $N$,
\[
\| \sum_{n=0}^N f_n  \|_{L^p} \leq c_p \| (\sum_{n=0}^N f_n^2)^{1/2} \|_{L^p} \leq c_p  \| (\sum_{n=0}^\infty f_n^2)^{1/2} \|_{L^p}.
\]
The same method of proof used to obtain this shows that for any $N_1< N_2$,
\[ \|  S_{2^{N_2}}f - S_{2^{N_1}} f \|_{L^p}\leq c_p \| ( \sum_{n=N_1+1}^{N_2} f_n^2)^{1/2} \|_{L^p}.\]
If $f$ is such that the right-hand side of the direct inequality converges, then this tail must vanish
  as $N_1, N_2 \maps \infty$, so that as $N \maps \infty$, $S_{2^N}f$ converges in $L^p$ norm to some function, say $F$, 
  which satisfies
$ \|F\|_{L^p} \leq   c_p\|( \sum_{n=0}^\infty f_n^2)^{1/2} \|_{L^p} .$
By the Dominated Convergence Theorem, for each $m$
\[c_m(F) = \int_0^1 F(\theta) w_m(\theta) d\theta = \int_0^1  f(\theta)w_m(\theta) d\theta  = c_m(f),\] 
and since  $\{w_m\}$ is a complete orthonormal system on $[0,1]$,   we conclude $F=f$, verifying the direct inequality.
For the converse inequality,  we apply the maximal bound (\ref{max_bound_WP}) to see that $\| ( \sum_{n=0}^N f_n^2)^{1/2} \|_{L^p} \leq c_p' \| \sum_{n=0}^N f_n \|_{L^p} \ll_p \|f\|_{L^p}$ uniformly in $N$, which suffices.
 
 \subsection{Linearization}
We have verified the direct and converse inequalities (\ref{TII_dir_con_together}) in $L^p$ for each even integer $p \geq 2$. To conclude the results for all $1 < p <\infty$, we recall Paley's arguments (now standard), in which the Rademacher functions again   make an appearance, via Khintchine's inequality.

One would like to interpolate either the direct inequality (or the converse inequality, respectively), but one must first linearize. For any fixed $1<p<\infty$, the truth for all $f \in L^p$ of the direct and converse inequalities 
 \beq\label{TII_direct_recap}
 \| (\sum_{n=0}^\infty f_n^2)^{1/2} \|_{L^p} \ll_p \|f\|_{L^p} \ll_p  \| (\sum_{n=0}^\infty f_n^2)^{1/2} \|_{L^p}
 \eeq
 is equivalent to the truth
 of the statement that 
 \beq\label{TII_linear}
 \|f^*\|_{L^p}\ll_p \|f\|_{L^p} \ll_p \|f^*\|_{L^p}
 \eeq
  holds for all $f \in L^p$, uniformly for all choices of $\ep_n \in \{ \pm 1\}$, where
 \[ f^*(t) = \sum_{n=0}^\infty \ep_n f_n(t).\]
 The advantage of  (\ref{TII_linear})  is that the expressions in this inequality are linear, and thus well-suited to interpolation.

Let us verify the equivalence. If (\ref{TII_linear}) holds, to deduce (\ref{TII_direct_recap}), we use the Rademacher functions. Given $f$ and its associated sequence $\{f_n\}$ we define an auxiliary function $F(t,\theta) = \sum_{n=0}^\infty r_n(\theta) f_n(t)$ for each $\theta \in [0,1]$. By assumption of (\ref{TII_linear}), for each fixed $\theta$, 
 \[  \int_0^1 | F(t,\theta)|^p dt \ll_p \int_0^1 |f(t)|^p dt \ll_p \int_0^1 | F(t,\theta)|^p dt.\]
We integrate this over $\theta \in [0,1]$ to conclude by Fubini's theorem that
\[ \int_0^1 \int_0^1 | \sum_{n=0}^\infty r_n(\theta) f_n(t)|^p d\theta dt \ll_p \int_0^1 |f(t)|^p dt \ll_p \int_0^1 \int_0^1 | \sum_{n=0}^\infty r_n(\theta) f_n(t)|^p d\theta dt.\]
Now for each fixed $t$ we apply Khintchine's inequality (\ref{TI_Khintchine}), and this proves that (\ref{TII_direct_recap}) holds, as desired.

The converse is more elementary. Given $f \in L^p$, and any choice of $\{\ep_n\}$, $f^*$ is the function with associated expansion $\sum_{n=0}^\infty g_n$ with $g_n=\ep_n f_n$, so that applying the direct inequality followed by the converse inequality assumed in (\ref{TII_direct_recap}) shows that
\[  \| f^* \|_{L^p} \ll_p \| (\sum_n g_n^2)^{1/2} \|_{L^p}  =\| (\sum_n f_n^2)^{1/2} \|_{L^p} \ll_p \|f\|_{L^p}.   \]
One obtains $\|f\|_{L^p} \ll_p \|f^*\|_{L^p}$ in an analogous fashion.

\subsection{Remarks for $2 \leq  p < \infty$} 
We know that (\ref{TII_direct_recap}) and hence (\ref{TII_linear}) holds for each $p=2r$ with $r \geq1$ an integer. 
 We fix a sequence $\{ \ep_n\}_n$ with $\ep_n \in \{\pm 1\}$ and consider a truncation $(S_{2^N}f)^* (t)  = \sum_{0 \leq n  \leq N} \ep_n f_n(t)$. 
 Then applying the left-hand side of (\ref{TII_linear}),  for every even integer $p \geq 2$, 
 \[ \|(S_{2^N}f)^* \|_{L^p} \ll_p \| S_{2^N}f\|_{L^p} \ll_p\|f\|_{L^p},\]
 in which the last inequality holds  uniformly in $N$, by the maximal theorem in (\ref{max_bound_WP}). 
By Riesz-Thorin interpolation between $p=2$ and any even integer, we conclude that this inequality holds for all $2 \leq p < \infty$. For a fixed $p \geq 2$, we can then deduce  that $(S_{2^N}f)^*$ converges in $L^p$ norm to a limit function, say $F^*$. By the Dominated Convergence Theorem, the  coefficients $c_m(F^*)$ agree with those of $f^*$, and since the Walsh functions form a complete system, we learn that $F^*=f^*$. We conclude that $\|f^*\|_{L^p} \ll_p \|f\|_{L^p}$, obtaining the left-hand  inequality of (\ref{TII_linear}) for each $2\leq p < \infty$.
For the other inequality, we simply observe that given $f$  and a fixed sequence $\{\ep_n\}$, then $(f^*)^*=f$, so the right-hand inequality of (\ref{TII_linear}) follows.
 
 \subsection{Remarks for $1< p  \leq 2$}\label{sec_linear}
One again uses the linearized inequalities (\ref{TII_linear}) in order to apply duality. 
Fix $1<p \leq 2$, and fix a sequence of $\ep_n \in \{\pm 1\}$, and accordingly define $f_N^* = \sum_{0 \leq n  \leq N} \ep_n f_n$. By duality, to show that $\| f_N^*\|_{L^p} \ll_p \|f\|_{L^p}$ it suffices to show that for all $g \in L^{p'}$ with  $1/p + 1/p'=1$, 
$\| f_N^* g\|_{L^1} \ll_p \|g\|_{L^{p'}} \|f\|_{L^p}.$
Precisely, 
\[ \| (\sum_{n=0}^{N} \ep_n f_n)  g \|_{L^1} = \| (\sum_{n=0}^{N} \ep_n g_n) f \|_{L^1} 
\leq \| \sum_{n=0}^{N} \ep_n g_n\|_{L^{p'}} \| f\|_{L^p},
\]
with the last inequality due to  H\"older's inequality.
We apply   the known case for  $p' \geq 2$, so that  $\| \sum_{n=0}^{N} \ep_n g_n\|_{L^{p'}}\ll_p \|g\|_{L^{p'}}$, uniformly in the choice of signs $\{ \ep_n\}$. We conclude that $\| f_N^*\|_{L^p} \ll_p \|f\|_{L^p}$ uniformly in $N$, and uniformly in the choice of $\{\ep_n\}$. Thus we may argue as before that $f_N^*$ converges in $L^p$ norm to a function, which we may check is indeed $f^* = \sum \ep_n f_n$, and this verifies that $\| f^*\|_{L^p} \ll_p \|f\|_{L^p}$ holds. For the other inequality, we again note that for each fixed choice of signs, $(f^*)^*=f$,  and thus we obtain $\| f\|_{L^p} \ll_p \|f^*\|_{L^p}$, concluding the proof.

  \subsection{Combining the direct and converse inequalities}\label{sec_combine} 
Fix $1<p<\infty$ and $n \geq 1$. To combine the direct and converse inequalities for the dyadic differences  $f_n = S_{2^{n}}f - S_{2^{n-1}}f$ in order to bound 
$S_n f$ on $L^p$, we must be able to express the partial sum $S_nf$ in terms of dyadic differences. Paley employs an identity of the following flavor.
Write the binary expansion $n=2^{n_1} + \cdots  + 2^{n_s}$ with $n_1> \cdots > n_s$. We claim
\beq\label{dyad}
w_n(t) w_n(\theta) \sum_{m=0}^{n-1} w_m(t)w_m(\theta) = \sum_{m \in [2^{n_1}, 2^{n_1+1})} w_m(t)w_m(\theta) + \cdots +\sum_{m \in [2^{n_s}, 2^{n_s+1})} w_m(t)w_m(\theta).
\eeq
Once we have verified this, the deduction is simple. 
Recall
\[ 
S_n f (t)  =  \sum_{m=0}^{n-1} c_m(f) w_m(t) =   \int_0^1 f(\theta) \sum_{m=0}^{n-1} w_m(\theta)  w_m(t) d\theta.
\]
To introduce the extraneous factor $w_n$ which is critical to the identity (\ref{dyad}), given any $f \in L^p[0,1]$ we define the function $g(\theta) = f(\theta) w_n(\theta)$ with identical $L^p$ norm; we will also use the notation $g_m  = S_{2^{m}} g - S_{2^{m-1}} g.$ Then using $w_n(\theta)^2 \con 1$ followed by (\ref{dyad}), 
\[ 
w_n(t) (S_n f)(t) =   \int_0^1 g(\theta) w_n(\theta) w_n(t) \sum_{m=0}^{n-1} w_m(\theta)  w_m(t) d\theta =  g_{n_1+1}(t) + \cdots + g_{n_s+1}(t).
\]
Now applying first the direct inequality and then the converse inequality for the functions $\{g_n\}$ we obtain the desired result:
\[ \| S_nf \|_{L^p}  = \| \sum_{j=1}^s g_{n_j+1} \|_{L^p} \ll_p \| ( \sum_{j=1}^s g_{n_j+1}^2)^{1/2} \|_{L^p}
	\leq   \| ( \sum_{n=0}^\infty g_n^2)^{1/2} \|_{L^p} 
	\ll_p  \| g \|_{L^p} 
	=   \|f \|_{L^p}.\]

To verify (\ref{dyad}), it suffices to observe an equivalent identity about  sets of numbers written  in binary (also expressible in terms of  properties of the Walsh group or ``dyadic group,'' see \cite[\S 2]{Fin49} or \cite{Bil67}).
Precisely, fix $n$ and $m \leq n$  and suppose $n=2^{n_1} + \cdots + 2^{n_s}$ (with $n_1> \cdots > n_s$)  and  $m=2^{m_1} + \cdots + 2^{m_r}$ (with $m_1 > \cdots > m_r$), and let the $(n_1+1)$-digit representation of $n$ and $m$ in binary be $\underline{n}, \underline{m}$, respectively.  Then $w_n w_m = w_u$ where $\underline{u} = \underline{n} \oplus \underline{m}$; here $\oplus$ denotes exclusive-or summation. (Since the square of any Rademacher function is identically one,   if any exponent occurs in both the binary expansion of $n$ and of $m$, then it does not appear as an  exponent in the binary expansion of $u$ for the function $w_u$ such that $w_u = w_nw_m$.) 

Consequently, (\ref{dyad}) is equivalent to the following identity on sets of distinct binary numbers:  
\[ \{ \underline{n} \oplus \underline{m} : 0 \leq m < n \} = \bigsqcup_{j=1}^s \{ \underline{m} : 2^{n_j} \leq m < 2^{n_j+1}\}. \]
We can first verify that for $j=1$,
$ \{ \underline{n} \oplus \underline{m} : 0 \leq m < 2^{n_1} \}  = \{ \underline{m} : 2^{n_1} \leq m < 2^{n_1+1}\}.$
This is because the map acting on $ 0 \leq m < 2^{n_1}$    by $m \mapsto  \underline{n} \oplus \underline{m} $ is injective and maps into $\{ \underline{m} : 2^{n_1} \leq m < 2^{n_1+1}\}$; since the cardinalities match, it is a bijection. 
Similarly, one can see that for each $2 \leq j \leq s$,
\[ \{ \underline{n} \oplus \underline{m} : 2^{n_1} + \cdots + 2^{n_{j-1}} \leq m < 2^{n_1} + \cdots + 2^{n_{j-1}} + 2^{n_j} \}  = \{ \underline{m} : 2^{n_j} \leq m < 2^{n_j+1}\},\]
and the claim holds.

In Remark \ref{remark_WPR}, we claimed that while the functions $\{w_n\}$ are orthogonal, they do not themselves possess superorthogonality properties for $2r$-tuples with $r\geq 2$. This referred to the fact that for any $r \geq 2$,  we can pick $2r$ functions $w_n$ with $2r$ distinct values of $n$ (so  the tuple $(n_1,\ldots, n_{2r})$ satisfies the hypothesis of Type I or Type II or Type III)   such that $\int w_{n_1} \cdots w_{n_{2r}} = 1$. Using the notation introduced above, this follows from the fact that we can choose $2r$ pairwise distinct integers $n_1, \ldots, n_{2r}$ such that when written in binary, $\underline{n}_1 \oplus \cdots \oplus \underline{n}_{2r}=0$.
  
  \section*{Appendix B: The source of quasi-superorthogonality for trace functions}
  \begin{center}
\emph{ Appendix by Emmanuel Kowalski}\footnote{ETH Z\"urich, R\"amistrasse 101, 8092 Z\"urich, Switzerland. Email: \tt{kowalski@math.ethz.ch}}
  \end{center}
  
 \setcounter{subsection}{0}
    
  \setcounter{equation}{0}
\renewcommand{\theequation}{B.\arabic{equation}}
This short note will attempt to explain the source of the
quasi-superorthogonality of trace functions that appears in Section \ref{sec_trace},
and in particular it will highlight that it arises from ``exact''
superorthogonality (of the corresponding type) for other functions,
combined with Deligne's very deep work on the Riemann Hypothesis over
finite fields. We then explain briefly the source of the exact
superorthogonality in the type of examples considered in the
survey~\cite{FKM15} of Fouvry, Kowalski and Michel.

\emph{Remark.}
  The presentation is not fully rigorous, since we did not want to
  obscure the key conceptual point with technical aspects, such as the
  need to work with continuous $\ell$-adic representations, etc.

Let $q$ be a prime number. The key data is a certain compact
topological group $\Pi_q$ associated to~$q$, with a normal
subgroup~$\Pi_q^g$ (both are algebraic variants of the classical
fundamental group of topology, but mainly viewed as classifying
coverings of the space, instead of groups of homotopy classes of
loops). Moreover, for every $x\in\Z/q\Z$, there exists a conjugacy
class $\theta_{q}(x)$ in $\Pi_q$ (called the \emph{Frobenius}
conjugacy class at~$x$), and $\Pi_q^g$ is big enough that it and a
single Frobenius conjugacy class generate~$\Pi_q$ topologically.

A trace function $F$ modulo~$q$ always has the
following form: there exists a finite-dimensional vector space $V$ on
which~$\Pi_q$ acts linearly (i.e., a finite-dimensional representation
of the group) in such a way that
\begin{equation}\label{eq-1}
  F(x)=\Tr(\theta_{q}(x)\mid V),
\end{equation}
the trace of the endomorphism of~$V$ associated to the Frobenius
conjugacy class at~$x$. This is well-defined, since the trace is
invariant under conjugation.

We view the action as a homomorphism $\rho\colon\Pi_q\to \GL(V)$. Then
the formula~(\ref{eq-1}) shows that \emph{a trace function is the
  restriction of the character of a representation to a certain subset
  of conjugacy classes of that group}.\footnote{\ To be more precise, this
  applies exactly in this way only when all $x\in\Z/q\Z$ are
  ``unramified'' for~$\rho$; since exceptions to this are rare for the
  cases that interest us, and since there is in any case a similar
  (but slightly more complicated) description even when~$x$ is
  ramified, we do not dwell on this issue.}

The Grothendieck--Lefschetz trace formula combined with Deligne's
Riemann Hypothesis can then be shown to imply (for suitable trace
functions) the statement that
\begin{equation}\label{eq-deligne}
  \sum_{x\in \Z/q\Z}
  F(x)=\Bigl(\int_{\Pi_q^g}\Tr(\rho(y))dy\Bigr)cq+O(\sqrt{q}),
\end{equation}
for some complex number~$c$ with $|c|\leq 1$, where the integral is
with respect to the probability Haar measure on the compact
group~$\Pi_q^g$ and the implied constant in the $O(\cdot)$ symbol
depends only on ``local'' invariants of~$\rho$ which are usually easy
to bound.

\emph{Remark.}
  In many cases of interest, one deals with an action of~$\Pi_q$ which
  has the property that~$\rho(\Pi_q^g)=\rho(\Pi_q)$.
  Then~(\ref{eq-deligne}) holds with~$c=1$, and thus it indicates that
  the discrete sum of the trace of~$\rho$ over the finitely many
  Frobenius classes $\theta_q(x)$ is close to the integral over the
  whole group (note that~$\rho(\theta_q(x))\in\Pi^g_q$ because of the
  assumption on~$\rho$). However, the formula~(\ref{eq-deligne}) holds
  in general in the stated form.

We can now explain how this, together with algebraic properties of
certain compact Lie groups, leads to quasi-superorthogonality.

Suppose we have finitely many trace functions $F_1$, \ldots, $F_{2r}$,
each associated to a representation $\rho_i$ (on the space~$V_i$),
satisfying suitable conditions.
We want to understand the sum
$$
\sum_{x\in \Z/q\Z} F_1(x)\overline{F_2(x)}\cdots
F_{2r-1}(x)\overline{F_{2r}(x)}.
$$

Part of the unspecified properties required of~$\rho_i$ imply that the
contragradient or dual representation $D(\rho_i)$ of~$\rho_i$
satisfies
$$
\Tr(D(\rho_i)(y))=\overline{\Tr(\rho_i(y))}.
$$
So, according to~(\ref{eq-deligne}), applied to the representation
$$
\rho=\rho_1\otimes D(\rho_2)\otimes \cdots\otimes \rho_{2r-1}\otimes
D(\rho_{2r}),
$$
we get
\begin{multline*}
\sum_{x\in \Z/q\Z} F_1(x)\overline{F_2(x)}\cdots
F_{2r-1}(x)\overline{F_{2r}(x)}
\\=
\Bigl(\int_{\Pi_q^g}\Tr(\rho_1(y))
\overline{\Tr(\rho_2(y))}\cdots
\Tr(\rho_{2r-1}(y))
\overline{\Tr(\rho_{2r}(y))}
dy\Bigr)c'q+O(\sqrt{q}),
\end{multline*}
for some complex number~$c'$ with~$|c'|\leq 1$.

Thus, we will obtain quasi-superorthogonality, of any type, for the
trace functions, provided the characters $\Tr(\rho_i)$ of the $\rho_i$
(restricted to the subgroup~$\Pi_q^g$) satisfy \emph{exact}
superorthogonality of the same type.

We present now one source of such superorthogonality that lies behind
many examples (but not all --- for Dirichlet characters, such as in the
inequality (\ref{Dir_sum}), the mechanism is a bit different).

In fact, at this point, we can replace~$\Pi^g_q$ by any fixed compact
group~$G$, with the~$\rho_i$ being unitary (continuous)
finite-dimensional representations of~$G$.
\par
According to the character theory of compact groups the integral
\begin{equation}\label{eq-integral}
  \int_{G}\Tr(\rho_1(y)) \overline{\Tr(\rho_2(y))}\cdots
  \Tr(\rho_{2r-1}(y)) \overline{\Tr(\rho_{2r}(y))} dy
\end{equation}
is equal to the dimension of the space of invariant vectors in the
tensor product representation~$\rho$. Now suppose that each $V_i$ has
dimension at least~$2$ and that the image of each $\rho_i$, which is a
subgroup of the unitary group of the space $V_i$, happens to be the
special unitary group $\SU(V_i)$.  Consider the map
$$
y\mapsto (\rho_1(y),\ldots,\rho_{2r}(y))
$$
from~$G$ to
$$
\SU(V_1)\times \cdots\times \SU(V_{2r}).
$$
Let~$H$ be its image. It is again a compact group, and it has the
property that the projection of~$H$ to each factor~$\SU(V_i)$ is
surjective. Now a special case of what Katz~\cite[\S 1.8,
Prop. 1.8.2]{Kat90} has called the Goursat--Kolchin--Ribet
property is that such a subgroup~$H$ is \emph{equal} to the product
$$
\SU(V_1)\times \cdots\times \SU(V_{2r}),
$$
unless at least two of the representations are equivalent, in which
case at least two of the characters $\Tr(\rho_i)$ are the same
functions. (To see that this may be the case, consider the special
case where all~$V_i$ have different dimensions; then the groups
$\SU(V_i)$ are pairwise non-isomorphic ``almost'' simple groups, and
the projection assumption implies that the group~$H$ has to contain
all of them as ``Jordan--H\"older factors'', which is only possible
if~$H$ is the full product.)

Thus, if no two of the characters are equal, then we have a splitting
of the integral
\begin{multline*}
  \int_{G}\Tr(\rho_1(y)) \overline{\Tr(\rho_2(y))}\cdots
  \Tr(\rho_{2r-1}(y)) \overline{\Tr(\rho_{2r}(y))} dy = \int_H
  \Tr(y_1,y_2^*,\ldots,y_{2r-1},y_{2r}^*)dy_1\cdots dy_{2r}
  \\
  = \Bigl(\int_{\SU(V_1)}\Tr(y_1)dy_1\Bigr)\cdots
  \Bigl(\int_{\SU(V_{2r})}\overline{\Tr(y_{2r})}dy_{2r}\Bigr),
\end{multline*}
which vanishes. In other words, in these conditions, we obtain
superorthogonality of Type II, and in fact really in the same way
suggested at the beginning of the paper, i.e., from independent random
variables, these being the different characters
$y\mapsto\Tr(\rho_i(y))$.

One can be more precise about conditions on the
representations~$\rho_i$ that lead to vanishing of the
integral~(\ref{eq-integral}), but we hope that this sketch has given
some idea of how this may arise.

\section*{Acknowledgements}
Pierce is partially supported by NSF   CAREER grant DMS-1652173, a Sloan Research Fellowship, and the AMS Joan and Joseph Birman Fellowship.

\bibliographystyle{alpha}
\bibliography{NoThBibliography}
\label{endofproposal}

\end{document}